\documentclass[sts,
	reqno,
]{imsart}

\usepackage{amsmath, amssymb, color}
\usepackage{amsthm} 
\usepackage[colorlinks,citecolor=blue,urlcolor=blue]{hyperref}
\usepackage{xcolor,colortbl}
\numberwithin{equation}{section}

\usepackage{color}
\usepackage{amsfonts, fancybox}
\usepackage{amssymb}
\usepackage[american]{babel}
\usepackage{enumitem}
\setenumerate{label = (\roman*)}
\usepackage{psfrag,color}
\usepackage{dcolumn}
\usepackage[format=hang, justification=justified, singlelinecheck=false]{subcaption}
\usepackage{flafter, bm, dsfont}
\usepackage[section]{placeins}

\usepackage{multirow}
\usepackage{color}
\usepackage{amsmath}
\usepackage{mathrsfs}
\usepackage{amsfonts}
\usepackage{dsfont}
\usepackage{amssymb}
\usepackage{algorithmicx, algorithm, algpseudocode}

\usepackage{graphicx}

\startlocaldefs

\newcommand{\supp}{\operatorname{supp}}
\newcommand{\bc}{\begin{center}}
\newcommand{\ec}{\end{center}}
\newcommand{\ba}{\begin{array}}
\newcommand{\ea}{\end{array}}
\newcommand{\be}{\begin{eqnarray}}
\newcommand{\ee}{\end{eqnarray}}
\newcommand{\bel}{\begin{eqnarray}\label}
\newcommand{\eel}{\end{eqnarray}}
\newcommand{\bes}{\begin{eqnarray*}}
\newcommand{\ees}{\end{eqnarray*}}
\newcommand{\bn}{\begin{enumerate}}
\newcommand{\en}{\end{enumerate}}

\newcommand{\unif}{\mathsf{Unif}}

\newcommand{\domain}{\operatorname{dom}}

\definecolor{MIT}{cmyk}{.24, 1.00, .78, .17} 
\definecolor{pink}{cmyk}{0, 1, 0, 0} 
\definecolor{darkgreen}{cmyk}{1,0, 1, 0}

\newcommand{\iid}{{\it i.i.d.\ }}

\newtheorem{theorem}{Theorem}
\newtheorem*{theorem*}{Theorem}
\newtheorem{lemma}[theorem]{Lemma}
\newtheorem{propdef}[theorem]{Proposition-Definition}

\newtheorem{proposition}[theorem]{Proposition}
\newtheorem*{proposition*}{Proposition}

\newtheorem{corollary}[theorem]{Corollary}
\newtheorem{remark}[theorem]{Remark}
\newtheoremstyle{TheoremNum}
		{\topsep}{\topsep}              %%%
		{\itshape}                      %%%
		{}                              %%%
		{\bfseries}                     %%%
		{.}                             %%%
		{ }                             %%%
		{\thmname{#1}\thmnote{ \bfseries #3}}%%%

\makeatletter
\newtheorem*{rep@theorem}{\rep@title}
\newcommand{\newreptheorem}[2]{%
\newenvironment{rep#1}[1]{%
 \def\rep@title{#2 \ref{##1}}%
 \begin{rep@theorem}}%
 {\end{rep@theorem}}}
\makeatother
\newreptheorem{lemma}{Lemma}
\newreptheorem{theorem}{Theorem}
\newreptheorem{proposition}{Proposition}

\theoremstyle{plain}
\newtheorem{assumpgen}{} %

\newtheorem{assumpspec}{} %

\newtheorem{assumpalt}{} %

\usepackage[utf8]{inputenc}
\newcommand{\leadeq}[2][4]{\MoveEqLeft[#1] #2 \nonumber}

\usepackage{stmaryrd}
\usepackage{bbm}

\newcommand{\mockalph}[1]{}

\usepackage{mathtools}
\mathtoolsset{showonlyrefs}
\makeatletter
\newcommand{\subalign}[1]{%
  \vcenter{%
    \Let@ \restore@math@cr \default@tag
    \baselineskip\fontdimen10 \scriptfont\tw@
    \advance\baselineskip\fontdimen12 \scriptfont\tw@
    \lineskip\thr@@\fontdimen8 \scriptfont\thr@@
    \lineskiplimit\lineskip
    \ialign{\hfil$\m@th\scriptstyle##$&$\m@th\scriptstyle{}##$\crcr
      #1\crcr
    }%
  }
}
\makeatother
\makeatletter
\newcommand{\cf}{cf\@ifnextchar.{}{.\ }}
\makeatother

\newcommand{\DS}{\displaystyle}

\newcommand{\cD}{\mathcal{D}}

\newcommand{\cF}{\mathcal{F}}

\newcommand{\cP}{\mathcal{P}}

\newcommand{\cS}{\mathcal{S}}
\newcommand{\cT}{\mathcal{T}}

\newcommand{\cX}{\mathcal{X}}

\newcommand{\R}{\mathbbm{R}}
\newcommand{\p}{\mathbbm{P}}
\newcommand{\E}{\mathbbm{E}}

\newcommand{\1}{\mathbbm{1}}

\newcommand{\pn}{\p_{\kern-0.25em n}}
\newcommand{\pnm}{\p_{\kern-0.25em n,m}}
\newcommand{\psubm}{\p_{\kern-0.25em m}}
\newcommand{\psubp}{\p_{\kern-0.25em p}}
\newcommand{\cfi}{\cF_{\kern-0.25em \infty}}

\newcommand{\argmin}{\mathop{\mathrm{argmin}}}

\newcommand{\ud}{\mathrm{d}}

\newcommand{\eps}{\varepsilon}

\newlength{\minipagewidth}
\newcommand{\extension}{\operatorname{ext}}
\newcommand{\ext}{\operatorname{ext}}

\newcommand{\potential}{f}
\newcommand{\tmap}{T}
\newcommand{\dualpotential}{g}

\newcommand{\risk}{\mathcal{S}}
\newcommand{\emprisk}{\hat{\risk}}

\newcommand{\mongeprimal}{\cP_{\mathrm{M}}}
\newcommand{\kantoprimal}{\cP_{\mathrm{K}}}
\newcommand{\kantoprimalmod}{\tilde \cP_{\mathrm{K}}}

\newcommand{\id}{\mathrm{id}}

\newcommand{\funcclass}{\mathcal{F}}
\newcommand{\funcclassstd}{\mathcal{X}}
\newcommand{\rkhs}{\mathcal{H}}
\newcommand{\funcclassgeneral}{\mathcal{G}}
\newcommand{\funcclassmap}{\mathcal{T}}
\newcommand{\measureclass}{\mathcal{M}}

\newcommand{\convexset}{U}
\newcommand{\measureone}{P}
\newcommand{\measuretwo}{Q}
\newcommand{\density}{\rho}
\newcommand{\diam}{\operatorname{diam}}
\newcommand{\interior}{\operatorname{int}}

\newcommand{\norm}[1]{\| #1 \|}

\newcommand{\constbd}{M} %
\newcommand{\constsmooth}{R} %
\newcommand{\constempproc}{C_1} %
\newcommand{\constlog}{C_2}
\newcommand{\constjapprox}{C_3}
\newcommand{\constsobolev}{C_4} %
\newcommand{\constbdder}{C_5}
\newcommand{\constlinfty}{C_6}
\newcommand{\constdudleycombined}{C_7}

\newcommand{\fvec}{\bm{f}}
\newcommand{\Tvec}{\bm{T}}
\newcommand{\xvec}{\bm{x}}
\newcommand{\yvec}{\bm{y}}
\newcommand{\legtrafo}{\mathcal{L}}
\newcommand{\interp}{\mathcal{P}}

\newcommand{\invwavtrafo}{W^\top}
\newcommand{\wavcoeffs}{\gamma}

\endlocaldefs

\arxiv{1905.05828}

\begin{document}

\begin{frontmatter}

	\title{Minimax estimation of smooth optimal transport maps}
	\runtitle{Minimax estimation of transport maps}

	\author{ \fnms{Jan-Christian} \snm{H\"utter}\thanksref{}\ead[label=huetter]{jhuetter@broadinstitute.org} 
		~and~
		\fnms{Philippe} \snm{Rigollet}\thanksref{}\ead[label=rigollet]{rigollet@math.mit.edu}
	}
	\affiliation{Broad Institute and Massachusetts Institute of Technology}

	\address{{Jan-Christian H\"utter}\\
		{Broad Institute of MIT and Harvard} \\
		{415 Main Street,}\\
		{Cambridge, MA 02142, USA}\\
		\printead{huetter}
	}

	\address{{Philippe Rigollet}\\
		{Department of Mathematics} \\
		{Massachusetts Institute of Technology}\\
		{77 Massachusetts Avenue,}\\
		{Cambridge, MA 02139-4307, USA}\\
		 \printead{rigollet}
	}

	\runauthor{H\"utter and Rigollet}

	\begin{abstract}
		\ Brenier's theorem is a cornerstone of optimal transport that guarantees the existence of an optimal transport map $T$ between two probability distributions $P$ and $Q$ over $\R^d$ under certain regularity conditions.
		The main goal of this work is to establish the minimax estimation rates for such a transport map from data sampled from $P$ and $Q$ under additional smoothness assumptions on \( T \).
		To achieve this goal, we develop an estimator based on the minimization of an empirical version of the semi-dual optimal transport problem, restricted to truncated wavelet expansions. 
		This estimator is shown to achieve near minimax optimality using new stability arguments for the semi-dual and a complementary minimax lower bound.
		Furthermore, we provide numerical experiments on synthetic data supporting our theoretical findings and highlighting the practical benefits of smoothness regularization.
		These are the first minimax estimation rates for transport maps in general dimension.
	\end{abstract}


\end{frontmatter}

\section{Introduction}
\label{SEC:intro}

Wasserstein distances and the associated problem of optimal transport date back to the work of Gaspard Monge \cite{Mon81} and have since then become important tools in pure and applied mathematics \cite{Vil03, Vil09, San15}. Tools from optimal transport have been successfully employed in machine learning~\cite{AlaGraCut18, PeyCut19,SchHeiBon18,FlaCutCou18,ArjChiBot17a,GenPeyCut18,JanCutGra18,MonMulCut16,RolCutPey16,SchHeiBon17,GraJouBer19, StaClaSol17,AlvJaaJeg18, ArjChiBot17a,DelGamGor19,CanRos12} computer graphics \cite{LavClaChi18, SolGoePey15, SolPeyKim16,FeyChaVia17}, statistics~\cite{SegCut15,AhiGouPar18, RigWee18, WeeBer19, ZemPan19, PanZem19, BigGouKlei17, CazSegBig18, RamTriCut17,ClaSol18,TamMun18,KlaTamMun18,BigCazPap17,BarGorLes19,KroSpoSuv19}, and, more recently, computational biology \cite{SchShuTab19, YanDamVen18}.

\medskip

Monge asked the following question:
Given two probability measures \( \measureone, \measuretwo \) in \( \R^d \), how can we transport \( \measureone \) to \( \measuretwo \) while minimizing the total distance traveled by this transport.
A classical instantiation of this problem over $\R^d$ is to find a map \( \tmap_0 \colon \R^d \to \R^d \) that minimizes the objective
\begin{equation}
	\label{eq:og}
	\min_{\tmap} \:  \int_{\R^d} \| \tmap(x) - x \|_2^2 \, \ud \measureone(x), \quad \text{s.t. } \tmap_\# \measureone = \measuretwo,
\end{equation}
which is known as the \emph{Monge problem}, where \( \tmap_\# \measureone \) denotes the push-forward of \( \measureone \) under \( \tmap \), that is,
\begin{equation}
	\label{eq:pe}
	\tmap_\# \measureone(A) = \measureone(\tmap^{-1}(A)), \quad \text{for all Borel sets } A.
\end{equation}

The highly non-linear constraint in \eqref{eq:pe} made the mathematical treatment of the Monge problem seem elusive for a long time, until the seminal work of Kantorovich \cite{Kan42,Kan48}, who considered the following relaxation.
Instead of looking for a map \( \tmap_0 \), look for a transport plan \( \Gamma_0 \) in the set of all possible probablity measures on \( \R^d \times \R^d \) whose marginals coincide with \( \measureone \) and \( \measuretwo \), which we denote by \( \Pi(\measureone, \measuretwo) \).
This leads to the optimization problem
\begin{equation}
	\label{eq:pg}
	\min_{\Gamma} \int \| x - y \|_2^2 \, \ud \Gamma(x,y) \quad \text{s.t. } \Gamma \in \Pi(\measureone, \measuretwo),
\end{equation}
which is known as the \emph{Kantorovich problem}, and whose value is the square of the 2-Wasserstein distance, \( W_2^2(\measureone, \measuretwo) \), between the two probability measures \( \measureone \) and \( \measuretwo \).
The two optimization problems are indeed linked: Brenier's Theorem (Theorem~\ref{thm:brenier} below), guarantees that under regularity assumptions on \( \measureone \), a solution \( \Gamma_0 \) to \eqref{eq:pg} is concentrated on the graph of a map \( \tmap_0 \).
That is, using a suggestive informal notation, \( \Gamma_0(x, y) = \measureone(x) \, \delta_{y = \tmap_0(x)} \), where \( \delta \) denotes a point mass.
Moreover,  \( \tmap_0 \) is the gradient of a convex function \( \potential_0 \).
While cost functions other than \( \| x - y \|_2^2 \) could be of interest, such as \( \| x - y \|_2^p \) for \( p \ge 1 \), this work entirely focuses on the quadratic cost, which allows leveraging the well-established theory of convex functions and formulating key assumptions in terms of strong convexity.

Statistical optimal transport describes a body of questions that arise when the measures $P$ and $Q$ are unknown but samples are available. While the question of estimation of various quantities such as $W_2(P,Q)$, for example, are of central importance, for applications such as domain adaptation and data integration \cite{DamKelFla18,CouFlaTui14,CouFlaHab17,SegDamFla18,PerCouFla16,CouFlaTui17, ForHutNit18, RigWee19}, the main quantity of interest is the transport map $T_0$ itself since it can be used to push almost every point in the support of $P$ to a point in the support of $Q$.
Moreover, the optimal transport map plays an important role in characterizing the Riemannian geometry that arises from endowing probability measures that have finite second moments with the 2-Wasserstein-distance.
In particular, it can be used to define the right-inverse to the exponential map in that space \cite{Gig11}, which in turn enables the generalization of PCA (principal component analysis) to spaces of probability measures \cite{BigGouKlei17, MasPanZem19}.
The goal of this paper is to study the rates of estimation of a smooth transport map $T_0$ from samples.

To fix a concrete setup assume that we have at our disposal \( 2n \) independent observations \( X_1, \dots, X_n \) from \( \measureone \) and \( Y_1, \dots, Y_n \) from \( \measuretwo \), based on which we would like to find an estimator \( \hat \tmap \) for \( \tmap_0 \). This statistical problem poses several challenges:

\noindent (i) The most straightforward estimator is obtained by  replacing \( \measureone \) and \( \measuretwo \) by their empirical counterparts~\cite{MunCza98}.
It leads to a finite-dimensional linear problem that can be approximated very efficiently due to recent algorithmic advances \cite{Cut13, AltWeeRig17, PeyCut19, DvuGasKro18}.
However, even if the resulting optimizer \( \hat \Gamma \) is actually a map (matching), which it is not in general, it is not defined outside the sample points. In particular, it does not indicate how to transport a point $x \notin \{X_1, \ldots, X_n\}$.
In contrast, we would like to obtain an estimator \( \hat \tmap \) with guarantees in \( L^2(\measureone) \), that is, with convergence of
\begin{equation}
	\label{eq:ql}
	\| \hat \tmap - \tmap_0 \|_{L^2(\measureone)}^2 := \int \| \hat \tmap(x) - \tmap_0(x) \|_2^2 \, \ud \measureone(x)\,.
\end{equation}
Note that such an estimator of \( \tmap_0 \) could be obtained by post-processing the above optimizer \( \hat \Gamma \), for example by interpolation techniques, see \cite{AndBor16}.
We also employ related techniques in Section \ref{sec:numerics} to obtain practical estimators for \( \tmap_0 \).
However, we are not aware of a statistical analysis of these procedures.

\noindent (ii) It is known that the estimator \( W_2^2(\hat \measureone, \hat \measuretwo) \) can be a poor proxy for \( W_2^2(\measureone, \measuretwo) \) if the underlying dimensionality of the distributions \( \measureone \) and \( \measuretwo \) is large, as it suffers from the so-called \emph{curse of dimensionality}.
	For example, if both \( \measureone \) and \( \measuretwo \) are absolutely continuous with respect to the Lebesgue measure in \( \R^d , d\ge 4\), it is known that up to logarithmic factors, $\E[|W_2(\hat \measureone, \hat \measuretwo) - W_2(\measureone, \measuretwo)|] \asymp n^{-1/d}$ (see \cite{NilRig19}).
In fact, we show in Theorem \ref{thm:lb} that without further assumptions on either \( \measureone, \measuretwo \), or \( \tmap_0 \), no estimator can have an expected squared loss~\eqref{eq:ql} uniformly better than \( n^{-2/d} \). 

\noindent (iii) While many heuristic approaches have been brought forward to address the previous point, a thorough statistical analysis of the rate of convergence has so far been lacking.
This can be partly attributed to the structure (or lack thereof) of problem \eqref{eq:pg}.
Being a linear optimization problem, it lacks simple stability estimates that are key to establish statistical guarantees by relating \( \| \hat \tmap - \tmap_0 \|_{L^2(\measureone)} \) to the sub-optimality gap
\begin{equation}
	\label{eq:pq}
	\int_{\R^d} \| \hat \tmap(x) - x \|_2^2 \, \ud \measureone(x) - \int_{\R^d} \| \tmap_0(x) - x \|_2^2 \, \ud \measureone(x).
\end{equation}

In this paper, we aim to address these problems by imposing additional assumptions on the transport map \( \tmap_0 \) that lead to a rate faster than \( n^{-2/d} \).
One assumption we impose on the transport map \( \tmap_0 \) is smoothness, a standard way of alleviating the curse of dimensionality in non-parametric estimation.
Another key assumption is based on an observation of Ambrosio published in an article by Gigli \cite{Gig11}.
They show that the optimization problem \eqref{eq:og} has quadratic growth, in the sense of a stability estimate
\begin{equation}
	\label{eq:qm}
	\| \tmap - \tmap_0 \|_{L^2(\measureone)}^2 \lesssim \int_{\R^d} \| \tmap(x) - x \|_2^2 \, \ud \measureone(x) - \int_{\R^d} \| \tmap_0(x) - x \|_2^2 \, \ud \measureone(x),
\end{equation}
provided \( \tmap_0 = \nabla \potential_0 \) is Lipschitz continuous on \( \R^d \) and \( \tmap_\# \measureone = \measuretwo \).
While this observation does not immediately lend itself to the analysis of an estimator due to the presence of the push-forward constraint, we show in Proposition~\ref{prop:gigli} that under similar assumptions, the so-called semi-dual problem (see \eqref{eq:pn} below) admits a stability estimate.

Due to the rising interest in Optimal Transport as a tool for statistics and machine learning, many empirical regularization techniques have been proposed, ranging from the computationally successful entropic regularization \cite{Cut13, GenCutPey16, AltWeeRig17}, \( \ell^2 \)-regularization \cite{BloSegRol17}, smoothness regularization \cite{PerCouFla16}, to regularization techniques specifically adapted to the application of domain adaptation \cite{CouFlaTui14, CouFlaTui17, CouFlaHab17}.
Notably, \cite{GenCutPey16} also consider regularization based on the semi-dual objective and reproducing kernel Hilbert spaces.
However, the statistical performance of these regularization techniques to estimate transport maps from sampled data has been largely unanswered, with the following exceptions.

The estimation of transport maps has been studied in the one-dimensional case under the name \emph{uncoupled regression}~\cite{RigWee19} where the sample $Y_1, \ldots, Y_n$ is subject to measurement noise. There, the main statistical difficulty arises from the presence of this additional noise and boils down to obtaining deconvolution guarantees in the Wasserstein distance. Such guarantees were recently obtained under smoothness assumptions on the underlying density~\cite{CaiChaDed11, DedMic13, DedFisMic15} but they do not translate directly into rates of estimation for the optimal transport map beyond the 1D case. Note that in the presence of Gaussian measurement noise, the rates of estimation are likely to become logarithmic rather than polynomial as deconvolution is a statistically difficult task.

Concurrently to this work,~\cite{FlaLouFer19} study the estimation of linear transport maps and establish a fast rate of convergence in this parametric setup.
Moreover, after finishing the first version of this paper, the authors became aware of the parallel work \cite{Gun18}.
There, Gunsilius analyzes the asymptotic variance of an estimator for \( \potential_0 \) under smoothness assumptions on the densities of the marginals \( \measureone \) and \( \measuretwo \) and states the problem of obtaining estimation rates for the transport map \( \tmap_0 \) explicitly as an open problem, which we address in this paper.
A similarity between \cite{Gun18} and this work is that the rates for the variance are obtained by applying empirical process theory to the semi-dual objective function.
However, we note the following key differences:
First, Gunsilius obtains curvature estimates for the semi-dual objective via variational techniques, while we use strong convexity of the candidate potentials.
Note that under slightly stronger regularity assumptions, i.e., uniform convexity of the supports, his assumptions would imply strong convexity for the ground truth potential, as shown in \cite{Gig11}.
Second, by assuming smoothness of the transport map instead of the distributions \( \measureone \) and \( \measuretwo \), our results are more flexible and can be applied in cases where neither distribution possesses H\"older smooth densities.
Third, by appealing to Cafarelli's global regularity theory, \cite[Theorem 12.50(iii)]{Vil09}, \cite{Caf92,Caf92a,Caf96}, we can obtain a variance rate of \( n^{-2 \alpha/(2 \alpha - 2 + d)} \) (up to log factors) under assumptions quantitatively comparable to Gunsilius's, while his results imply a sub-optimal rate of \( n^{-(2 \alpha - 2)/(2 \alpha - 2 + d) } \), see Section \ref{sec:caffarelli} in the appendix.

We note that both in the application of Caffarelli's theorem in Appendix \ref{sec:caffarelli} and in the statement of our main result, Theorem \ref{thm:minimax} below, currently available analytical tools limit the extent to which minimax results can be established.
In Appendix \ref{sec:caffarelli}, the lack of uniformity in Cafarelli's global regularity theory prevents us from claiming (near) minimax optimality over H\"older smooth densities, see Remark \ref{rem:caf-no-minimax}.
Similarly, in Theorem \ref{thm:minimax}, (near) minimax optimality is attained by considering fixed marginal supports because of the need for uniformly bounded constants in some of the classical inequalities we employ (for example, Poincar\'e inequality), see Remark \ref{rem:genassumptions}.
In effect, further strengthening these minimax results poses an interesting open problem involving deep analytical questions.

The rest of this paper is organized as follows.
In Section~\ref{sec:ot-overview}, we review some important concepts of optimal transport, mainly duality and Brenier's theorem.
These are instrumental in the definition of our estimator, which is postponed to Section~\ref{sec:upper}.
Indeed, since the main goal of this paper is to establish minimax rates of convergence for smooth transport maps, we present these rates in Section~\ref{sec:mainresults} and  prove lower bounds in the following Section~\ref{sec:lower}, since this proof illustrates well the source of the nonstandard exponent in the rates.
We then proceed to Section~\ref{sec:upper} where we define a minimax optimal estimator constructed as follows.
First, we define an estimator for the optimal Kantorovich potential as the solution to the empirical counterpart of the semi-dual problem restricted to a class of wavelet expansions.
Then, our estimator is defined as the gradient of this potential.
We prove that it achieves the near-optimal rate in the same section.
In Section~\ref{sec:numerics}, we present numerical experiments on synthetic data, introducing two estimators that exploit smoothness of the transport map.
The first illustrates that a version of the estimator considered in Section~\ref{sec:upper} can in fact be implemented, at least in low dimensions.
The second is heuristically motivated and based on kernel-smoothing the transport plan between empirical distributions, showing that practical gains in higher dimensions can be achieved for smooth transport maps.
Finally, some useful facts from convex analysis (Section \ref{SEC:convex-analysis}), approximation theory for wavelets (Section \ref{sec:wavelets}), empirical process theory (Section \ref{SEC:emp-process}), and tools for proving lower bounds (Section \ref{SEC:lower-bounds-tools}), are appendix.
Moreover, the appendix also contains a version of our upper bounds based on smoothness assumptions on the densities instead of the transport map (Section \ref{sec:caffarelli}), the deferred proofs (Section \ref{sec:proofs}), additional lemmas (Section \ref{sec:lemmas}), and more details on the numerical experiments (Section~\ref{sec:numerics-ctd}).

\noindent{\sc Notation.} For any positive integer $m$, define $[m]:=\{1, \ldots, m\}$.
We write \( |A| \) for the cardinality of a set \( A \).
The relation \( a \lesssim b \) is used to indicate that two quantities are the same up to a constant \( C \), \( a \le C b \).
The relation \( \gtrsim \) is defined analogously, and we write \( a \asymp b \) if \( a \lesssim b \) and \( a \gtrsim b \).
We denote by \( c \) and \( C \) constants that might change from line to line and that may depend on all parameters of the statistical problem except \( n \).
We abbreviate with \( a \vee b \), \( a \wedge b \) the maximum and minimum of \( a \in \R \) and \( b \in \R \), respectively.
For \( a \in \R \), the floor and ceiling functions are denoted by \( \lfloor a \rfloor \) and \( \lceil a \rceil \), indicating rounding \( a \) to the next smaller and larger integer, respectively.
We use \( \supp \potential \) to denote the support of a function or measure \( \potential \), and \( \diam \Omega \) for the diameter of a set \( \Omega \subseteq \R^d \).
We denote by \( B_1 \) the unit-ball with respect to the Euclidean distance in \( \R^d \), where \( d \) should be clear from the context.
For a real symmetric matrix \( A \) and \( \lambda \in \R \), we write \( \lambda \preceq A \) if all eigenvalues of \( A \) are bounded below \( \lambda \), and similarly for \( A \preceq \lambda \). 
Moreover, we denote the smallest and largest eigenvalues of \( A \) by \( \lambda_{\min}(A) \), \( \lambda_{\max}(A) \), respectively.

For \( p \in [1, \infty] \), we denote by \( \ell^p \) either the space \( \R^d \) endowed with the usual \( \ell^p \) norms \( \| \, . \, \|_p \), or, by abuse of notation, the spaces of multi-dimensional sequences \( \gamma : \mathbb{Z}^m \to \R \) with \( \| \gamma \|_{\ell^p}^p = \sum_{k \in \mathbb{Z}^m} | \gamma_k |^p \) for \( m \in \mathbb{N} \).
Further, for \( p \in [1, \infty] \), we denote by \( L^{p} \) the Lebesgue spaces of equivalence classes of functions on \( \R^d \) or subsets \( \Omega \subseteq \R^d \) with respect to the Lebesgue measure \( \lambda \) on \( \R^d \), whose norms we denote by \( \| \, . \, \|_{L^{p}(\R^d)} \) and \( \| \, . \, \|_{L^p(\Omega)} \), respectively.
By abuse of notation, for a different measure \( \measureone \), we denote the associated Lebesgue norms by \( \| \, . \, \|_{L^p(\measureone)} \).
We abbreviate with ``a.e.'' any statement that holds ``almost everywhere'' with respect to the Lebesgue measure.

For a differentiable one-dimensional function \( \potential \colon \R \supseteq \Omega \to \R \), we denote its derivative by \( \frac{d}{dx} \potential \).
For a function \( \potential \colon \R^d \supseteq \Omega \to \R \), we denote by \( \partial_{i} = \partial/(\partial x_i)\) its weak derivative in the sense of distributions in direction \( x_i \), which coincides with the usual (point-wise) derivative if \( \potential \) is differentiable in \( \Omega \).
For a multi-index \( \flat \in \mathbb{N}^{d} \), we set
\begin{equation}
	\label{eq:sh}
	\partial^\flat \potential = \frac{\partial}{\partial_{x_1}^{\flat_1}} \dots \frac{\partial}{\partial_{x_d}^{\flat_d}}\potential, \quad \text{ and } | \flat | = \sum_{i=1}^d \flat_i.
\end{equation}
The symbol \( \partial \potential \) is also used to denote the sub-differential of a convex function \( \potential \), while we use the symbols \( \nabla \potential \) for the gradient of a function \( \potential \) and \( D g\) for the derivative of a vector-valued function \( \dualpotential \colon \R^{d_1} \supseteq \Omega \to \R^{d_2} \), $\nabla \potential = (\partial_1 \potential, \dots, \partial_d \potential)^\top$ and  $D g = (\nabla g_1, \dots, \nabla g_{d_2})^\top$
respectively, and \( D^2 \potential = D \nabla \potential \) denotes the Hessian of \( \potential \).

If \( \Omega \subseteq \R^d \) is a closed set with non-empty interior and \( \alpha > 0 \), the H\"older spaces on \( \Omega \) as defined in Appendix \ref{sec:wavelets} are denoted by \( C^{\alpha}(\Omega) \) and their associated norms by \( \| \, . \, \|_{C^\alpha(\Omega)} \).
Similarly, the \( p \)-Sobolev spaces of order \( \alpha \) for \( p \in [1,\infty] \) are denoted by \( W^{\alpha,p}(\Omega) \) with norms \( \| \, . \, \|_{W^{\alpha,p}(\Omega)} \), as defined in Appendix \ref{sec:wavelets}.

We say that \( \Omega \subseteq \R^d \) is a Lipschitz domain if its boundary can be locally expressed as the sublevel set of Lipschitz functions \cite[Definition 1.103]{Tri06}.

\section{Brenier's theorem and the semi-dual problem}

\label{sec:ot-overview}

We begin by recalling the Monge and Kantorovich problems given in Section \ref{SEC:intro}.
Let \( \measureone, \, \measuretwo  \) be two Borel probability measures on \( \R^d \) with finite second moments.

The \emph{Monge (primal) problem} is defined as
\begin{align}
	\label{eq:qi}
	\min_{\tmap} {} & \mongeprimal(T) \quad \text{s.t. } \tmap_\# \measureone = \measuretwo\,, \quad \\
	\quad \text{where } &\mongeprimal(T) := \frac{1}{2} \int_{\R^d} \| \tmap(x) - x \|_2^2 \, \ud \measureone(x),
\end{align}
and the push-forward \( \tmap_\# \measureone \) is defined as $\tmap_\# \measureone(A) = \measureone(\tmap^{-1}(A))$ for all Borel sets $A$.

Its relaxation, the \emph{Kantorovich (primal) problem}, is given by
\begin{align}
	\label{eq:qo}
	\min_{\Gamma} {} & \kantoprimalmod(\Gamma) \quad \text{s.t. } \Gamma \in \Pi(\measureone, \measuretwo)\quad \\
	\text{where } {} & \kantoprimalmod(\Gamma) := \frac{1}{2} \int \| x - y \|_2^2 \, \ud \Gamma(x,y),
\end{align}
and \( \Pi(\measureone, \measuretwo) \) denotes the set of couplings between \( \measureone \) and \( \measuretwo \), that is, the set of probability measures $\Gamma$ on $\R^d\times \R^d$ such that $ \Gamma(A \times \R^d) = \measureone(A)$ and $\Gamma(\R^d \times A) = \measuretwo(A)$ for all Borel set $A \subset \R^d$.

The value of problem \eqref{eq:qo} is the square of the 2-Wasserstein distance, denoted by
\begin{equation}
	\label{eq:qs}
	W_2^2(\measureone, \measuretwo) := 
	\min_{\Gamma \in \Pi(\measureone, \measuretwo)} \kantoprimalmod(\Gamma).
\end{equation}

Note that we can expand the objective in~\eqref{eq:qo} as
\begin{align}
	\label{eq:pj}
	\kantoprimalmod(\Gamma) = {} & \frac{1}{2} \int \| x - y \|_2^2 \, \ud \Gamma(x, y)\\
	= {} & \frac{1}{2} \int \| x \|_2^2 \, \ud \measureone(x) + \frac{1}{2} \int \| y \|_2^2 \, \ud \measuretwo(y) - \int \langle x , y \rangle \, \ud \Gamma(x, y),
\end{align}
Since the first two terms above do not depend on \( \Gamma \), we obtain the equivalent optimization problem
\begin{equation}
	\label{eq:pk}
\quad 	\max_{\Gamma} {} \kantoprimal(\Gamma)  \quad \text{s.t. }  \Gamma \in \Pi(\measureone, \measuretwo)\,, \quad 
	\text{where } \quad \kantoprimal(\Gamma) := \int \langle x , y \rangle \, \ud \Gamma(x, y)\,.
\end{equation}
We focus on this equivalent formulation for the rest of the paper because it is more convenient to work with.

Problem \eqref{eq:pk} is a linear optimization problem, albeit an infinite-di\-men\-sional one.
Hence, it is natural to consider its dual problem:
\begin{align}
	\label{eq:pi}
	\min_{\potential, \dualpotential} {} & \int \potential(x) \, \ud \measureone(x) + \int \dualpotential(y) \, \ud \measuretwo(y) \quad \text{s.t. } \\
	&
\begin{array}{ll}
&	\potential(x) + \dualpotential(y) \ge \langle x , y \rangle \,,\   \measureone \otimes \measuretwo \text{-a.e}, \\
&\potential \in L^1(\measureone), \dualpotential \in L^1(\measuretwo).
\end{array}
\end{align}
The dual variables \( \potential \) and \( \dualpotential \) are  called \emph{potentials}, and for an optimal pair \( (\potential_0, \dualpotential_0) \), \( \potential_0 \) is called a \emph{Kantorovich potential}.

The dual problem \eqref{eq:pi} can be further simplified:
Assume we are given a candidate function \( \potential \) in \eqref{eq:pi} above.
Then, we can formally solve for the corresponding \( \dualpotential \) to get an optimal $g$ given by the Legendre-Fenchel conjugate (see Section~\ref{SEC:convex-analysis})  of \( \potential \):
\begin{equation}
	\label{eq:pm}
	g_\potential(y) = \sup_{x \in \R^d} \langle x , y \rangle - \potential(x) = \potential^\ast(y),
\end{equation}

Plugging solution \eqref{eq:pm} back into the optimization problem leads to the so-called \emph{semi-dual problem},
\begin{align}
	\label{eq:pn}
	\min_{\potential} {} & \risk(\potential)=\int \potential(x) \, \ud \measureone(x) + \int \potential^\ast(y) \, \ud \measuretwo(y)  \quad \text{s.t. } \potential \in L^1(P),
\end{align}
where the supremum in \eqref{eq:pm} is interpreted as an essential supremum with respect to \( P \).
By transitioning to the semi-dual, we effectively solved for all constraints in \eqref{eq:pi}, leaving us with an unconstrained convex problem that is not linear anymore. Under regularity assumptions, a solution to the semi-dual provides a solution to the Monge problem as indicated by the following theorem, which is a cornerstone of modern optimal transport.

\begin{theorem}
	[{Brenier's theorem, \cite{KnoSmi84, Bre91, RusRac90}}]
	\label{thm:brenier}

	Assume \( \measureone \) is absolutely continuous with respect to the Lebesgue measure and that both \( P \) and \( Q \) have finite second moments.
	Then, a unique optimal solution to \eqref{eq:pk} exists and is of the form \( \Gamma_0 = (\id, \tmap_0)_\# \measureone \), where \( \tmap_0 = \nabla \potential_0 \) is the gradient of a convex function \( \potential_0 \colon \R^d \to \R \).
	In fact, \( \potential_0 \) can be chosen to be a minimizer of the semi-dual objective in \eqref{eq:pn}.
\end{theorem}

Brenier's theorem implies that a solution to the semi-dual problem readily gives an optimal transport map.
Our strategy is to minimize an approximation of the semi-dual and establish stability results as well as generalization bounds to conclude that the minimizer to the approximation is close to the minimizer of the original problem.

\section{Main results}
\label{sec:mainresults}

Let $X_{1:n} = (X_1, \ldots, X_n)$ and $Y_{1:n} = (Y_1, \ldots, Y_n)$ be $n$ independent copies of $X \sim P$ and $Y\sim Q=(T_0)_\#P$ respectively.
Furthermore, assume that $X_{1:n}$ and $Y_{1:n}$ are mutually independent.
Our goal is to estimate $T_0$. 
To that end, we consider the following set of assumptions on $P, \, Q$ and $T_0$. Throughout, we fix a constant $M\ge 2$.

\begin{assumpspec}[Source distribution]
	\label{assumpspec:source}
	Let \( \measureclass = \measureclass(\constbd) \) be the set of all probability measures \( \measureone \) with support $\Omega=[0,1]^d$ that admit a density \( \density_\measureone \) with respect to the Lebesgue measure such that $\DS M^{-1} \le \density_\measureone(x) \le M$ for almost all $x \in \Omega$. 
	Assume that the source distribution \( \measureone \) is in \(\measureclass \).
\end{assumpspec}

\begin{assumpspec}[Transport map]
	\label{assumpspec:map}
	Let \( \tilde \Omega = [-1,2]^d \) denote the enlargement of $\Omega$ by \( 1 \) in every direction. Let $\funcclassmap = \funcclassmap(\constbd)$ be the set of all differentiable functions \( \tmap \colon \tilde \Omega \to \R^d \) such that $T = \nabla \potential$ for some differentiable convex function $ \potential \colon \tilde \Omega \to \R^d$ and
	\begin{enumerate}
		\item \label{itm:assumpspec-bd} $\DS |T(x)|\le M$ for all $x \in \tilde \Omega$,
		\item \label{itm:assumpspec-d2} $\DS M^{-1} \preceq DT(x) \preceq M$ for all $x \in \tilde \Omega$,
	\item \label{itm:fixed-target-distribution} $\DS \supp \tmap_{\#}P = \Omega = [0,1]^d$\,.
	\end{enumerate}

	For \( \constsmooth > 1 \) and \( \alpha > 1 \), assume that 	$T_0 \in \funcclassmap_\alpha =   \funcclassmap_\alpha(\constbd, \constsmooth)$, where
$$
\funcclassmap_\alpha(\constbd, \constsmooth)=\{ \tmap \in \funcclassmap(\constbd) : \tmap \text{ is } \lfloor \alpha \rfloor \text{-times differentiable and } \| \tmap \|_{C^\alpha(\tilde \Omega)} \le R \}.
$$
\end{assumpspec}

Our main result is the following theorem. It characterizes, up to logarithmic factors,  the minimax rate of estimation of an $\alpha$-smooth transport map $T_0 \in \cT_\alpha$ in the setup described above.

\begin{theorem}

	\label{thm:minimax}
	Fix $\alpha \ge 1$, then
	\begin{equation}
		\label{eq:lb}
		\inf_{\hat \tmap} \sup_{\measureone \in \measureclass, \, \tmap_0 \in \funcclassmap_\alpha} \E
		\left[ \int\| \hat \tmap(x) - \tmap_0(x) \|_2^2 \, \ud P(x) \right]
		\gtrsim n^{-\frac{2 \alpha}{2 \alpha - 2 + d}} \vee \frac{1}{n} \,.
	\end{equation}
	where the infimum is taken over all measurable functions \( \hat T \) of the data \( X_{1:n} = 	(X_1, \dots, X_n) \), \( Y_{1:n} = (Y_1, \dots, Y_n) \).
	Moreover, if \( \measureone \in \measureclass \) and \( \tmap_0 \in \funcclassmap_\alpha \), there exists an estimator \( \hat T \), given in Section~\ref{sec:upper}, that is near minimax optimal. More specifically, there exists an integer $n_0=n_0(d, \alpha, \constbd, \constsmooth)$ such that for any $n \ge n_0$, it holds,
	\begin{equation}
		\label{eq:je}
		\sup_{\measureone \in \measureclass, \, \tmap_0 \in \funcclassmap_\alpha} \E
		\left[ \int\| \hat \tmap(x) - \tmap_0(x) \|_2^2 \, \ud P(x) \right]
		\lesssim n^{-\frac{2 \alpha}{2 \alpha - 2 + d}} (\log(n))^2 \vee \frac{1}{n}\,.
	\end{equation}
\end{theorem}

\begin{remark}
	\label{rem:genassumptions}
	Assumption \ref{assumpspec:source}  can be significantly relaxed with respect to the geometry of \( \Omega \) and the density of \( \measureone \). In fact, the upper bounds are given under more general assumptions in Section \ref{sec:upper}. Similarly, the assumption \ref{assumpspec:map}\ref{itm:fixed-target-distribution} that $\supp(Q)=[0,1]^d$ can also be relaxed. 

	However, the constants in the resulting upper bounds exhibit a dependence on the geometry of the supports of both \( \measureone \) and \( \measuretwo \) as well as on the enclosing set \( \tilde \Omega \) through functional analytical results used in the proofs.
	While it may be possible to make this dependence explicit in terms of geometric features of the sets $\supp(P), \supp(Q)$ and $\tilde \Omega$---see for example \cite{Jon81, EspNitTro11, TanMizSek13} for such estimates under restrictive assumptions---providing a uniform control on these quantities in terms of easily interpretable properties of the sets is beyond the scope of this article. Instead,  we chose to present Theorem \ref{thm:minimax} under these simplified assumptions to make the results more readable.

To discuss the remaining assumptions, we note that the most essential ones to obtain upper bounds are the following:
first, the lower bound in \ref{assumpspec:map} \ref{itm:assumpspec-d2}, \( M^{-1} \preceq DT(x) \), in particular on the support of \( \measureone \), \( x \in \Omega \).
This yields convergence estimates for the optimal transport map as shown in \cite{Gig11}, see \eqref{eq:qm}, and might be necessary to obtain fast rates for transport map estimation since it provides curvature estimates commonly needed to prove error bounds for M-estimators \cite[Chapter 5]{Vaa00}.
Second, the Sobolev regularity of \( T_0 \) is what governs the approximation rates of \( T_0 \) by wavelet expansions (see Section \ref{sec:approx-error} below) and thus enables fast rates via a bias-variance trade-off.
All remaining assumptions in \ref{assumpspec:source} and \ref{assumpspec:map}, including the existence of extensions of \( \tmap_0 \) to a superset \( \tilde \Omega \), are of technical nature and serve to give explicit bounds as needed in the proof of the upper bound.
While one might be able to relax these assumptions, especially in specific problem instances, we do not pursue this here beyond the more general versions given in \ref{assumpgen:source} and \ref{assumpgen:map} below.

\end{remark}

We conjecture that the logarithmic terms appearing in the upper bound are superfluous and arise as an artifact of our proof techniques. We briefly make a qualitative comment on the rate $n^{-\frac{2 \alpha}{2 \alpha - 2 + d}}$. Note first that it appears from this rate that estimation of transport maps, like the estimation of smooth functions suffers from the curse of dimensionality. However, as $\alpha \to \infty$, this curse of dimensionality may be mitigated by extra smoothness with the parametric rate $n^{-1}$ as a limiting case. Note also that we can formally take the limit $\alpha \to 1$, which corresponds to the case where no additional smoothness condition holds beyond having a strongly convex Kantorovich potential with Lipschitz gradient. This is essentially the minimal structural condition arising from Brenier's theorem with additional bounds on the derivative of \( \tmap_0 \). In this case, one formally recovers the rate $n^{-2/d}$ and we conjecture that this is the minimax rate of estimation in the context where $T_0$ is only assumed to be the gradient of a strongly convex function with Lipschitz gradient. If either of these two additional requirements is not fulfilled, our stability results no longer hold. 

\begin{remark}
	Since the transport map \( \tmap_0 \) is the main focus of our results, our assumptions impose smoothness directly on $\tmap_0$. In fact, smoothness of $\tmap_0$ can also be seen as a consequence of smoothness of the source and target distribution using Caffarelli's regularity theory \cite{Caf92,Caf92a, Caf96}.
	For completeness, we also give a version of our upper bound results under smoothness assumptions on \( \measureone \) and \( \measuretwo \) in Theorem~\ref{thm:caffarelli}, Section~\ref{sec:caffarelli} of the Appendix.\end{remark}

\begin{remark}
	\label{rem:smooth-density-in-wasserstein}
By prescribing a known base measure \( \measureone \), such as the uniform distribution on \( [0,1]^d \), and considering \( \tilde \measuretwo = \hat \tmap_{\#} \measureone \), estimation rates for \( \hat T \) immediately translate into rates for estimating \( \measuretwo \) \emph{in} the \( 2 \)-Wasserstein distance \cite{WeeBac17, SinPoc18, WeeBer19}.
In fact, \( W_2^2(\tilde \measuretwo, \measuretwo) \) can be bounded by \( \E_{\measureone}[\| \hat \tmap - \tmap_0\|_2^2] \), since \( (\hat \tmap, \tmap_0)_{\#} \measureone \) is a candidate transport plan between \( Q \) and \( \tilde Q \).
Up to log factors, our rates obtained below match those obtained in \cite{WeeBer19} for the estimation of a smooth density on \( [0,1]^d \) in the \( 2 \)-Wasserstein distance.
\end{remark}

\section{Lower bound}
\label{sec:lower}
In this section we begin by proving the lower bound \eqref{eq:lb} as it sheds light on the source of the non-standard exponent $\frac{2 \alpha}{2 \alpha - 2 + d}$ in the minimax rate. 
We prove the following theorem.

\begin{theorem}
\label{thm:lb}
Fix $\alpha \ge 1$. It holds that
	\begin{equation}
		\inf_{\hat T} \sup_{\measureone \in \measureclass, \, \tmap_0 \in \funcclassmap_\alpha} \E \left[ \int\| \hat T(x) - \tmap_0(x) \|_2^2 \, \ud P(x) \right]
		\gtrsim n^{-\frac{2 \alpha}{2 \alpha - 2 + d}} \vee \frac{1}{n} \,. \tag{\ref{eq:lb} revisited}
	\end{equation}
	where the infimum is taken over all measurable functions of \( X_1, \dots, X_n \), \(Y_1, \dots, Y_n \).
\end{theorem}

\begin{proof}
	The proof uses standard tools to establish minimax lower bounds, \cite[Theorem~2.5, Lemma~2.9, Theorem~2.2]{Tsy09}, that we restate in Appendix \ref{SEC:lower-bounds-tools} for the convenience of the reader as Theorem~\ref{thm:lower-bounds-master}, Lemma~\ref{lem:varshamov-gilbert}, and Theorem~\ref{thm:lower-bounds-master-twopoint}, respectively.
	It relies on the following construction.

Set \( \measureone = \unif([0,1]^d) \in \measureclass \), the uniform distribution on the hypercube.
For $\alpha>1$, let \( \xi \colon \R \to \R \) be a non-zero function in $C^{\infty}(\R)$ with support contained in \( [0,1] \) such that there exists \( x_0 \in [0,1] \) with \( \xi(x_0) \neq 0 \), \( \frac{d}{dx} \xi(x_0) \neq 0 \), for example a bump-function.
Define $g \colon \R^d \to \R$ by
\begin{equation}
	\label{eq:am}
	 g(x) = \prod_{i \in [d]} \xi(x_i)\,, \qquad x=(x_1, \ldots, x_d)\,,
\end{equation}
and note that \( \nabla g(x_0, \dots, x_0) \neq 0 \) and \( \supp(g) = [0,1]^d \) by the above assumptions on \( \xi \).

Let $m=\lceil \theta n^{\frac{1}{2\alpha - 2 +d}}\rceil$ be a positive integer where $\theta$ is a universal constant to be chosen later. We form a regular discretization of the space $[0,1]^d$ by defining the collection of vectors $\{x^{(j)}\,:\, j \in [m]^d\} \subset [0,1]^d$ to have coordinates $x^{(j)}_i = (j_i - 1)/m$, $i=1,\ldots, d$
and let
\begin{equation}
	\label{eq:an}
	g_j(x) = \frac{\kappa}{m^{\alpha+1}} g(m(x - x^{(j)})),
\end{equation}
for a constant \( \kappa > 0 \) to be chosen later.
Note that \( \supp(g_j) \subseteq x^{(j)} + [0, 1/m]^d \) and hence that the supports of the functions \( \{ g_j \}_{j \in [m]^d} \) are pairwise disjoint.

Next, let $\flat\in \mathbb{N}^d$ be a multi-index and observe that the differential operator $\partial^\flat$ applied to $g_j$ yields $\partial^\flat g_j(\cdot)=m^{|\flat| - \alpha - 1} \partial^{\flat} g(m(\cdot - x_j))$. Since $\xi \in C^{\alpha+1}$, if \( \alpha > 1 \), a second-order Taylor expansion yields that $g_j$ has uniformly vanishing Hessian: $\| D^2 g_j \|_{L^\infty(\R^d)} \to 0$ as $m \to \infty$. 
In particular, in that case, there exists $m_0$ such that $\| D^2 g_j(x) \|_{\mathrm{op}} \le 1/2$ for all \( x \in \R^d \), $m \ge m_0$, $j \in [m]^d$.
If \( \alpha = 1 \), the same can be obtained by choosing \( \kappa \) small enough.
By the same reasoning, we can also guarantee \( \| \nabla g_j \|_{L^\infty(\R^d)} \le 1 \).

For $m^d\ge 8$, the Varshamov-Gilbert lemma, Lemma~\ref{lem:varshamov-gilbert}, guarantees the existence of binary vectors  $\tau^{(0)}, \tau^{(1)}, \ldots, \tau^{(K)} \in \{0,1\}^{[m]^d}$, \( \tau^{(0)} = (0, \dots, 0) \), $K \ge 2^{m^d/8}$ such that $\|\tau^{(k)}-\tau^{(k')}\|_2^2\ge m^d/8$ for $0\le k \neq k' \le K$. With this, we define the following collection of Kantorovich potentials:
\begin{align}
	\label{eq:ao}
	\phi_k(x) = \frac{1}{2} \| x \|^2 + \sum_{j \in [m]^d} \tau^{(k)}_j g_j(x)\,, \qquad k=0, \ldots, K\,.
\end{align}

It is easy to see (Lemma~\ref{LEM:1to1}) that for any $k=0, \ldots, K$ and $m \ge m_0$, 
	\( \nabla \phi_k \) is a bijection from \( [0,1]^d \) to \( [0,1]^d \).
	Moreover, by Weyl's inequality and the above bound \( \| D^2 g_j(x) \|_{\mathrm{op}} \le 1/2 \), for all \( k \),
\begin{equation}
	\label{eq:tg}
	\lambda_\mathrm{min}(D^2 \phi_k(x)) \ge 1 - \sum_{j \in [m]^d} \lambda_{\mathrm{max}}(D^2 g_j(x)) \ge \frac{1}{2}.
\end{equation}
Similarly, we obtain \( \| \nabla \phi_k \|_{L^\infty([0,1]^d)} \le 2 \) and \( \lambda_\mathrm{max}(D^2 \phi_k(x)) \le 2 \) for all \( x \in \R^d \).
Hence, \( T_k := \nabla \phi_k \in \funcclassmap_\alpha(\constbd, \constsmooth) \) for $M>2$ and \( \kappa \) small enough.
We now check the conditions of Theorem~\ref{thm:lower-bounds-master}, where we consider the distance measure
\begin{equation}
	\label{eq:tl}
	d(T_k, T_{k'})^2 = \int_{[0,1]^d} \| T_k - T_{k'} \|^2 \, \ud x, \quad, 0 \le k, k' \le K.
\end{equation}

First, observe that for $0\le k \neq k' \le K$, it holds that
\begin{align}
	\label{eq:az}
	\leadeq{\int_{[0,1]^d} \| \nabla \phi_k(x) - \nabla \phi_{k'}(x) \|^2 \, \ud x}\\
	= {} & \frac{\kappa^2}{m^{2 \alpha + d}} \sum_{j\in [m]^d} (\tau^{(k)}_j - \tau^{(k')}_j)^2 \int_{\R^d} \| \nabla g(x) \|^2 \, \ud x
	\gtrsim {}  \frac{1}{m^{2 \alpha}}\,.
\end{align}
This yields
$$
\int_{[0,1]^d} \| \nabla \phi_k(x) - \nabla \phi_{k'}(x) \|^2 \, \ud x \gtrsim {n^{-\frac{2 \alpha}{2 \alpha - 2 + d}}},
$$
which completes checking the separation condition (i) of Theorem~\ref{thm:lower-bounds-master}.

To check condition (ii) of Theorem~\ref{thm:lower-bounds-master}, recall the Kullback-Leibler (KL) divergence between two measures $Q, P$ such that $Q$ is absolutely continuous with respect to $P$ is defined by
$$
D(Q\|P)=\E \log\big(\frac{\ud Q}{\ud P}(W)\big)\,, \quad W \sim Q.
$$

	In view of Lemma~\ref{LEM:1to1}, for any $k=0, \ldots, K$, the measure $\measuretwo_k=(\nabla \phi_k)_\# \measureone$ is supported on $[0,1]^d$ and in particular, it is  absolutely continuous with respect to $\measureone$.
	By the change of variables formula, it admits the density
\begin{align}
	\label{eq:monge-ampere}
	\frac{\ud \measuretwo_k}{\ud \measureone}(y):=  \frac{1}{\det D^2 \phi_k((\nabla \phi_k)^{-1}(y))}\1((\nabla \phi_k)^{-1}(y) \in [0,1]^d).
\end{align}
Moreover, let $X\sim P$ and $Y\sim Q_k$ be two random variables. It holds
\begin{align}
\label{EQ:KLD2}
D(\measuretwo_k \| \measureone)
= {} & \E\log \left( \frac{\ud \measuretwo_k}{\ud \measureone} (Y)\right)
= \E\log \left( \frac{\ud \measuretwo_k}{\ud \measureone} \big(\nabla \phi_k(X)\big)\right)\\
= {} & -\int_{[0,1]^d} \log \left( \det D^2 \phi_k(x) \right) \, \ud x\,.
\end{align}
Recall that  $D^2 \phi_k = I_d + \sum_{j \in [m]^d} \tau^{(k)}_j D^2g_j$ 
where $I_d$ denotes the identity matrix in $\R^d$. Therefore, since the functions $g_j$ have disjoint support, we have for all $x \in [0,1]^d$ that
\begin{align}
\log \left( \det D^2 \phi_k(x) \right)
= {} & \sum_{l=1}^d \log\Big(1+ \lambda_l\big( \sum_{j \in [m]^d} \tau^{(k)}_j D^2g_j(x)\big)\Big)\\
= {} & \sum_{l=1}^d \sum_{j \in [m]^d} \log\left(1+ \tau^{(k)}_j \lambda_l\left( D^2g_j(x)\right)\right),
\end{align}
where $\lambda_l(A)$ denotes the $l$th eigenvalue of a matrix $A$. Since $\log(1+z) \ge z - z^2/2$ for all $z>0$,
$$
\log \left( \det D^2 \phi_k(x) \right)\ge\sum_{j \in [m]^d} \tau^{(k)}_j \mathsf{Tr}(D^2g_j(x))-\frac12 \sum_{j \in [m]^d}\|D^2g_j(x)\|_F^2
$$
where $\|\cdot\|_F$ denotes the Frobenius norm. Thus,
\begin{align}
D(\measuretwo_k \| \measureone)
\le {} & - \sum_{j \in [m]^d} \tau^{(k)}_j \int_{[0,1]^d}\mathsf{Tr}(D^2g_j(x)) \, \ud x\\
{} & \quad +\frac12 \sum_{j \in [m]^d}\int_{[0,1]^d}\|D^2g_j(x)\|_F^2 \, \ud x.
\end{align}
On the one hand, by the divergence theorem and the fact that $g_j$ has bounded support,
$$
\int_{[0,1]^d}\mathsf{Tr}(D^2g_j(x))\, \ud x
=\int_{\partial[0,1]^d} \langle v(x) , \nabla g_j(x) \rangle \, \ud x=0\,,
$$
where $\partial[0,1]^d$ denotes the boundary of the unit hypercube and \( v(x) \) its outward-pointing unit normal vector. On the other hand
\begin{align*}
\sum_{j \in [m]^d}\int_{[0,1]^d}\|D^2g_j(x)\|_F^2\, \ud x
= \frac{\kappa^2}{m^{2 \alpha - 2 + d}} \sum_{j \in [m]^d}  \int_{\R^d} \| D^2 g(x) \|_F^2 \, \ud x \lesssim \frac{1}{m^{2 \alpha - 2}}.
\end{align*}

The above three displays yield
\begin{equation}
\label{EQ:boundKL}
D(\measureone^{\otimes n}\otimes \measuretwo_k^{\otimes n} \| \measureone^{\otimes n}\otimes \measureone^{\otimes n})=nD(\measuretwo_k\|\measureone) \lesssim \frac{n}{m^{2 \alpha - 2}} \le \frac{m^d}{\theta} \le \frac{\log K}{9}
\end{equation}
for $\theta$ large enough.
This completes checking (ii) in Theorem~\ref{thm:lower-bounds-master} and hence the proof of the first part of the minimax lower bound.

To show the remaining lower bound of \( 1/n \), repeat the same argument as above with the two potentials $\phi_0(x)=\| x \|_2^2/2$ and $ \phi_1(x)= \phi_0(x)+(\tilde \theta/\sqrt{n})g(x)$
for \( \tilde \theta \) chosen to ensure \( \phi_0, \phi_1 \in \funcclassmap_\alpha \), applying Theorem~\ref{thm:lower-bounds-master-twopoint} in Appendix \ref{SEC:lower-bounds-tools}.
The separation condition is  given by
\begin{equation}
	\label{eq:pb}
	\int_{[0,1]^d} \| \nabla \phi_0(x) - \nabla \phi_{1}(x) \|^2 \, \ud x
	= \frac{\tilde \theta^2}{n} \int_{[0,1]^d} \| \nabla g(x) \|_2^2 \, \ud x \gtrsim \frac{1}{n},
\end{equation}
and the KL divergence between the associated probability distributions can be estimated by
\begin{equation}
	\label{eq:pd}
	D(\measureone^{\otimes n}\otimes \measuretwo_1^{\otimes n} \| \measureone^{\otimes n}\otimes \measureone^{\otimes n})=nD(\measuretwo_1\|\measureone)
	\lesssim \frac{\tilde \theta^2 n}{n} \int_{[0,1]^d} \| D^2 g(x) \|_F^2 \, \ud x \lesssim \frac{1}{9},
\end{equation}
for \( \tilde \theta \) large enough.
\end{proof}

\bigskip

Looking back at this proof, we get a better understanding of the exponent in the minimax rate $n^{-\frac{2 \alpha}{2 \alpha - 2 + d}}$.
Given that $n^{-\frac{2 (\beta - k)}{2 \beta + d}}$ is the minimax rate of estimation of the $k$th derivative of a \( \beta \)-smooth density in \( L^2 \) \cite{MueGas79},
the rate that we obtain is formally that of an ``antiderivative'' (\( k = -1) \) of a \( \beta = \alpha - 1 \)-smooth signal in this model.
This is due to the fact that, on the one hand, the information structure, measured in terms of Kullback-Leibler divergence, is governed by the \emph{derivative} of the signal \( T_0 \), i.e., the Hessian of the $\alpha + 1$-smooth Kantorovich potential, see~\eqref{EQ:KLD2}, which is \( \alpha - 1 \)-smooth.
This follows directly from the Monge-Amp\`ere equation~\eqref{eq:monge-ampere}.
On the other hand, we measure the performance of the estimator in terms of the \( L^2(P) \) distance between \( \hat \tmap \) and \( \tmap_0 \), corresponding to the antiderivative of the Hessian.
Of course, in the absence of the classical fundamental theorem of calculus in dimension \( d > 1 \) for arbitrary maps \( \tmap \colon \R^d \to \R^d \), the existence of an antiderivative needs to be guaranteed a priori, as in the case of transport maps by assuming \( \tmap = \nabla \potential \) for \( \potential \colon \R^d \to \R \).

Similar rates arise in the estimation of the invariant measure of a diffusion process when smoothness is imposed on the drift~\cite{DalRei07,Str18}. This is not surprising as the drift is the gradient of the logarithm of the density of the invariant measure in an overdamped Langevin process.

Finally, note that the multivariate case is singularly different from the traditional univariate case where the rate of estimation of linear functionals such as anti-derivatives is known to be parametric regardless of the smoothness of the signal~\cite{IbrHas81}.

\section{Upper bounds}
\label{sec:upper}

In this section, we give an estimator \( \hat T \) that achieves the near-optimal rate \eqref{eq:je}.
We present this estimator under the following more general assumptions on the distribution and the geometry of the support of both \( \measureone \) and \( \measuretwo = (\tmap_0)_\# \measureone \). We also need slightly weaker conditions on the regularity of the transport map (Sobolev instead of H\"older regularity).  After stating these weaker assumptions, we present our estimator and restate the main upper bound.
Its proof relies on a separate control of approximation error and stochastic error, similar to a standard bias-variance tradeoff.

\subsection{Assumptions}

Throughout, we fix two constants \( M\ge 2, \beta > 1 \).
\begin{assumpgen}[Source distribution]
	\label{assumpgen:source}
	Let \( \measureclass = \measureclass(\constbd) \) be the set of all probability measures \( \measureone \) whose support $\Omega_P \subseteq \constbd  B_1$ is a bounded and connected Lipschitz domain, and that admit a density \( \density_\measureone \) with respect to the Lebesgue measure such that $\DS M^{-1} \le \density_\measureone(x) \le M$ for almost all $x \in \Omega_P$. 
	Assume that the measure \( \measureone \in \measureclass \).
\end{assumpgen}

\begin{assumpgen}[Transport map]
	\label{assumpgen:map}
	For any $P \in \measureclass$ with support $\Omega_P$, let \( \tilde \Omega_P \) denote a convex set with Lipschitz boundary such that \( \tilde \Omega_\measureone \subseteq \constbd B_1 \), and $\Omega_P + M^{-1}B_1 \subseteq \tilde \Omega_P$. Let $\tilde \funcclassmap = \tilde \funcclassmap(\constbd)$ be the set of all differentiable functions \( \tmap \colon \tilde \Omega_P \to \R^d \) such that $T = \nabla \potential$ for some differentiable convex function $ \potential \colon \tilde \Omega_P \to \R^d$ and
	\begin{enumerate}
		\item \label{itm:assump-boundedness} $\DS |T(x)|\le M$ for all $x \in \tilde \Omega_P$,
		\item \label{itm:assump-d2} $\DS M^{-1} \preceq DT(x) \preceq M$ for all $x \in \tilde \Omega_P$.
	\end{enumerate}
	For \( \constsmooth > 1 \) and \( \alpha > 1 \), assume that
	\begin{equation}
		\label{eq:lz}
		T_0 \in \tilde \funcclassmap_\alpha = \tilde \funcclassmap_\alpha(\constbd, \constsmooth) = \{ \tmap \in \tilde \funcclassmap(\constbd) :  \| \tmap \|_{C^\beta(\tilde \Omega_P)} \vee \| \tmap \|_{W^{\alpha,2}(\tilde \Omega_P)} \le R \}.
	\end{equation}
\end{assumpgen}

These new conditions have two implications.
First, they imply regularity of the Kantorovich potential $f_0$, where $T_0=\nabla f_0$, and second, they imply some conditions on the pushforward measure $Q=(T_0)_{\#}P$ that subsume the generalization of Assumption~\ref{assumpspec:map}\ref{itm:fixed-target-distribution}. These results are gathered in the following proposition (see Section~\ref{Proof:lem:immediate-bounds} for a proof).

\begin{propdef}
	\label{lem:immediate-bounds}
	Assume that $P$ satisfies \ref{assumpgen:source}, $T_0$ satisfies \ref{assumpgen:map} and let $\cX=\cX(M)$ be the set of all twice continuously differentiable functions $f\colon\tilde \Omega_P \to \R$  such that
	\begin{enumerate}
		\item \label{itm:immediate-bounds-d01} $\DS |f(x)|\le 2 M^2$ and $\DS| \nabla f(x)|\le M$ for all $x \in \tilde \Omega_P$,
	\item \label{itm:immediate-bounds-d2} $\DS M^{-1} \preceq D^2f(x) \preceq M$ for all $x \in \tilde \Omega_P$.
	\end{enumerate}
	
	 Then there exists a Kantorovich potential \( \potential_0 \in \cX(M)\) such that $T_0=\nabla f_0$,
\begin{equation}
\label{EQ:smoothpotentials}
 \| f_0 \|_{C^{\beta+1}(\tilde \Omega_\measureone)}\vee \| f_0 \|_{W^{\alpha+1,2}(\tilde \Omega_\measureone)} \le \constsmooth + 2 M^3.
\end{equation}
Further, $Q=\nabla (f_0)_{\#}P$ has a connected and bounded Lipschitz support $\Omega_Q \subseteq \constbd B_1$ and admits a density $\rho_Q$ with respect to the Lebesgue measure that satisfies $M^{-(d+1)}\le \rho_Q(y) \le M^{d+1}$ for all $y \in \Omega_Q$. 	
\end{propdef}

Note that the simplified Assumptions \ref{assumpspec:source} and \ref{assumpspec:map} from Section \ref{sec:mainresults} follow from  \ref{assumpgen:source} and \ref{assumpgen:map} in the case \( \Omega_\measureone = \Omega = [0,1]^d \) and \( \tilde \Omega_\measureone = \tilde \Omega = [-1, 2]^d \). Additionally, the simplified assumptions  restrict the class of transport maps to those such that  \( \Omega_\measuretwo = [0,1]^d \) and for which \( \beta = \alpha \). Indeed, noting that \( \| \tmap_0 \|_{W^{\alpha,2}(\tilde \Omega)} \lesssim \| \tmap_0 \|_{C^\alpha(\tilde \Omega)} \), we can fold the two smoothness conditions into one.

\subsection{Estimator}
\label{sec:estimator}

To construct an estimator for \( \tmap_0 \), we observe that if we had access to a Kantorovich potential \( \potential_0 \), then \( \tmap_0 = \nabla \potential_0 \) by Brenier's Theorem, Theorem \ref{thm:brenier}.
In turn, \( \potential_0 \) is the minimum of the semi-dual objective \eqref{eq:pn}.
Hence, we replace population quantities with sample ones in its definition to obtain an empirical loss function.
Moreover, to account for the assumed smoothness of the transport map and to ensure stability of the objective, we constrain our minimization problem to smooth and strongly convex Kantorovich potentials, restricted to a compact superset of the support of \( \measureone \).
Then, our estimator is the gradient of the solution to this stochastic optimization problem.

More precisely, for a measurable function \( \potential \), let us write
\begin{alignat}{2}
	\label{eq:iy}
	\measureone \potential = {} & \int \potential(x) \, \ud \measureone(x),& \quad \measuretwo \potential = {} & \int \potential(y) \, \ud \measuretwo(y),\\
	\quad \hat \measureone \potential = {} & \frac{1}{n} \sum_{i=1}^n \potential(X_i),& \quad \hat \measuretwo \potential = {} & \frac{1}{n} \sum_{j=1}^{n} \potential(Y_i),
\end{alignat}
where, as in Section \ref{sec:mainresults}, \( X_{1:n} = (X_1, \dots, X_n) \) and \( Y_{1:n} = (Y_1, \dots, Y_n) \) are \( n \) \iid samples from \( \measureone \) and \( \measuretwo \), respectively, that are mutually independent as well.
Recall from Section~\ref{sec:ot-overview} that the semi-dual objective is defined as $\risk(\potential) = \measureone \potential + \measuretwo \potential^\ast$ for $f 
\in L^1(P)$,
where \( \potential^\ast \) denotes the convex conjugate of~$\potential$.
Replacing both \( \measureone \) and \( \measuretwo \) by their empirical counterparts, we obtain the empirical semi-dual,
\begin{equation}
	\label{eq:qb}
	\emprisk(\potential) = \hat \measureone \potential + \hat \measuretwo \potential^\ast.
\end{equation}

In order to incorporate smoothness regularization into the minimization of \eqref{eq:qb}, we consider the restriction of potentials \( \potential \) to a wavelet expansion of finite degree, a strategy that is frequently used in non-parametric estimation \cite{HarKerPic98, GinNic16}.
For the purpose of this section, it is enough to think about wavelets as a graded orthogonal basis of \( L^2(\R^d) \), leading to nested subspaces
\begin{equation}
	\label{eq:qw}
	V_0 \subseteq V_1 \subseteq \dots \subseteq V_J \subseteq \dots \subseteq L^2(\R^d),
\end{equation}
that roughly correspond to increasing frequency ranges of the continuous Fourier transform of a function \( \potential \in L^2(\R^d) \).
Truncated wavelet decompositions yield good approximations for smooth functions and we control their approximation error in Lemma~\ref{lem:wavelet-approximation}.
We only consider the span \( V_J(\tilde \Omega_\measureone) \) of those basis functions of the wavelet expansion whose support has non-trivial intersection with~\( \tilde \Omega_\measureone \). This is a finite-dimensional vector space  as long as the elements of the wavelet basis have compact support.
The cut-off parameter \( J \) is chosen according to the regularity of \( \potential_0 \) in assumption \ref{assumpgen:map}, see Section \ref{sec:upper-bounds-proof}, or alternatively can be chosen adaptively by a straightforward but technical extension using a penalization scheme \cite{Mas07} that we omit for readability. 
Alternatively, other selection methods such as Lepski's method \cite{Lep91,Lep92,Lep93} could be used.
In order to ensure the necessary regularity and the compact support of the elements of the wavelet basis, we assume throughout that the wavelet basis is given by Daubechies wavelets of sufficient order.
For a more detailed treatment of wavelets, we refer the reader to Section \ref{sec:wavelets}.

To ensure stability of the minimizer of the semi-dual with respect to perturbations of the input distributions $P$ and $Q$, we further restrict the potentials \( \potential \) to mimic the assumptions in Proposition-Definition \ref{lem:immediate-bounds}, in particular, we enforce  upper and lower bounds on the Hessian \( D^2 \potential \) on \( \tilde \Omega_\measureone \) by demanding \( \potential \in \funcclassstd(2 M) \).

Combined, both wavelet regularization and strong convexity lead to the set
\begin{equation}
	\label{eq:fc}
	\funcclass_J = \funcclassstd(2M) \cap V_J(\tilde \Omega_\measureone)
\end{equation}
of candidate potentials, based on which we define the estimators
\begin{align}
	\label{eq:jb}
	\hat \potential_J \in \argmin_{\potential \in \funcclass_J} \emprisk(\potential), \quad \hat \tmap_J = \nabla \hat \potential_J,
\end{align}
for the Kantorovich potential and transport map, respectively.

Note that since we consider  candidate potentials only on the compact set \( \tilde \Omega_\measureone \), \( \potential^\ast \) above is defined as
\begin{equation}
	\label{eq:iz}
	\potential^\ast(y) = \sup_{x \in \tilde \Omega_\measureone} \langle x , y \rangle - f(x)
	= \sup_{x \in \R^d} \langle x , y \rangle - (f + \iota_{\tilde \Omega_\measureone})(x), \quad y \in \R^d,
\end{equation}
where \( \iota_{\tilde \Omega_\measureone} \) is the usual indicator function in convex analysis (see Section~\ref{SEC:convex-analysis}).

With this, we can restate the upper bound of Theorem \ref{thm:minimax}.

\begin{theorem}
	\label{thm:upper-bounds}
	Under assumptions \ref{assumpgen:source} and \ref{assumpgen:map}, there exists \( n_0 \in \mathbb{N} \) and \( J \) such that for \( n \ge n_0 \),
	\begin{equation}
		\label{eq:ls}
		\E_{(X_{1:n}, Y_{1:n})} \left[ \int\| \hat \tmap_J(x) - \tmap_0(x) \|_2^2 \, \ud \measureone(x) \right]
		\le C \left[n^{-\frac{2 \alpha}{2 \alpha - 2 + d}} (\log(n))^2 \vee \frac1n \right],
	\end{equation}
	where \( n_0, \, C \), and \( J \) may depend on \( d, \constbd, \constsmooth, \Omega_\measureone, \Omega_\measuretwo, \tilde \Omega_\measureone \), \( n_0 \) may additionally depend on \( \beta \), \( C \) may additionally depend on \( \alpha \), and \( J \) depends on \( n \).

	The cutoff \( J \) depends on \( \alpha \) if \( d \ge 3 \), but in the cases \( d = 1 \) and \( d = 2 \), \( J \) can be chosen independently from \( \alpha \).
\end{theorem}

\begin{remark}
	\label{rem:ub}
A few remarks are in order.
\begin{enumerate}
\item Similar upper bounds hold with high probability and can be inferred from the proof.
\item As written, the estimator \( \hat f_J \) is not directly implementable since the calculation of \( f^\ast \) involves computing a maximum over a continuous subset of \( \R^d \). However, this limitation can be overcome by a discretization of the space, although this is not practical even in moderate dimensions.
	We provide such an example implementation in Section \ref{sec:numerics}, along with a more practical estimator for which we give no theoretical upper bounds.

\item The numerical experiments in Section~\ref{sec:numerics} suggest that restricting the optimization in \eqref{eq:jb} to \( \funcclassstd(2M) \), while necessary for our proofs, might not be necessary in practice.

\item The estimator employed in Theorem~\ref{thm:upper-bounds} can be made adaptive to the unknown smoothness parameter $\alpha$ using a standard penalization scheme, see \cite{Mas07}. We omit this straightforward extension and instead focus on establishing minimax rates of estimation. For a more detailed account, we refer the reader to \cite{Hue19}.
\item \label{itm:tracking-m} For the sake of readability, we do not explicitly track the dependence on the parameter \( \constbd \).
	However, an inspection of the proof yields that the final rate scales like \( \constbd^{c_1 \, d + c_2} \) for constants \( c_1, c_2 \ge 1 \), i.e., there is an exponential dependence on the dimension \( d \).
	Further, the dependence on \( \constsmooth \) is captured in \eqref{eq:ub} below and amounts to \( \constsmooth^{\frac{2(d-2)}{2 \alpha - 2 + d}} \log(\constsmooth) \) in the case \( d \ge 3 \).
	We do not claim an optimal dependence of our rates on these parameters.
\end{enumerate}
\end{remark}

In the rest of this section, we present the proof of Theorem~\ref{thm:upper-bounds}. We begin by stating our key result, which relates the semi-dual objective to the square of our measure of performance. This result also allows us to employ a fixed-point argument when controlling the risk of our estimator using empirical process theory. Combined with approximation results for truncated wavelet expansions, these lead to a bias-variance tradeoff that achieves the minimax lower bound of Theorem~\ref{thm:lb} up to log factors.

\subsection{Stability of optimal transport maps}
\label{sec:stability}

In this section, we leverage the assumed regularity of the optimal transport map to relate  the suboptimality gap in the semi-dual objective function $\cS$ and the $L^2$-distance of interest.
\begin{proposition}
	\label{prop:gigli}
	Under assumptions \ref{assumpgen:source} -- \ref{assumpgen:map}, for all \( \potential \in \funcclassstd(2 M) \) as defined in Proposition-Definition \ref{lem:immediate-bounds}, we have
	\begin{equation}
		\label{eq:no}
	\frac1{8M}\| \nabla \potential(x) - \nabla \potential_0(x) \|_{L^2(P)}^2 \le \risk(\potential) - \risk(\potential_0)\le 2M\| \nabla \potential(x) - \nabla \potential_0(x) \|_{L^2(P)}^2 	\end{equation}
	and
	\begin{equation}
		\label{eq:np}
  \frac{1}{4M} \| \nabla \potential^\ast(y) - \nabla \potential^\ast_0(y) \|_{L^2(Q)}^2	\le  \risk(\potential) - \risk(\potential_0)
	\,.
	\end{equation}
\end{proposition}
\begin{proof}
	It follows from Proposition-Definition~\ref{lem:immediate-bounds}\ref{itm:immediate-bounds-d2}  and a second-order Taylor expansion that $f$ is of \emph{quadratic type} \cite{Kol11} around every $x \in \Omega_P$:
\begin{equation}
\label{EQ:fquad}
 \frac{1}{2L}\| z - x \|_2^2 \le	\potential(z) - \potential(x) - \langle \nabla \potential(x), z - x \rangle \le  \frac{L}{2}\| z - x \|_2^2, \quad \text{for } x \in \Omega_\measureone, z \in \tilde \Omega_\measureone,
\end{equation}
for all $L \ge 2M$. 
It turns out that these conditions are sufficient to obtain the desired result. 

The upper bound in~\eqref{EQ:fquad} is of the form 
$$
f(z)\le q_x(z)=\potential(x) + \langle \nabla \potential(x), z - x \rangle + \frac{L}{2}\| z - x \|_2^2 + \iota_{\tilde \Omega_\measureone}(z)\,.$$ 
Since the convex conjugate is order reversing and because the convex conjugate $q_x^*$ of the quadratic function $q_x$ can be computed explicitly (Lemma~\ref{lem:quadratic-conjugate}), we have 
\begin{align}
f^*(\nabla f_0(x))
\ge  q^*_x(\nabla f_0(x))
= {} & \frac{1}{2L}\|\nabla f_0(x)-\nabla f(x)\|_2^2+\langle x, \nabla f_0(x)\rangle-f(x)\\
{} & \quad -\frac{L}{2}d^2\Big(\frac{\nabla f_0(x)-\nabla f(x)}{L}-x, \tilde \Omega_P\Big)\,.
\end{align}
The squared distance term vanishes for  $L=4M$:  by the triangle inequality $\|\nabla f_0(x)-\nabla f(x)\|_2\le 4M$ and since $x \in \Omega_P$, it holds that
$$
\frac{\nabla f_0(x)-\nabla f(x)}{L}-x \in \tilde \Omega_P.
$$
Together with the fact that  $Q=(\nabla f_0)_{\#}P$, this yields
\begin{align}
\cS(f)
= {} & Pf+Qf^*
= \int[f(x)+f^*(\nabla f_0(x))] \ud P(x)\\
\ge {} & \frac1{8M}\|\nabla f -\nabla f_0\|_{L^2(P)}^2 + \int \langle x, \nabla f_0(x)\rangle \ud P(x)\,.
\end{align}
Moreover, by strong duality, we have
$$
\cS(f_0)=Pf_0+Qf_0^* = \int \langle x, \nabla f_0(x)\rangle \ud P(x)\,.
$$
The above two displays yield $\cS(f)-\cS(f_0) \ge  (8M)^{-1}\|\nabla f -\nabla f_0\|_{L^2(P)}^2$.
In the same way, using the lower bound in~\eqref{EQ:fquad}, we get that $\cS(f)-\cS(f_0) \le  2M\|\nabla f -\nabla f_0\|_{L^2(P)}^2$, which concludes the proof of~\eqref{eq:no}.

It turns out that~\eqref{eq:np} is even easier to prove.
Indeed, by Proposition-Definition~\ref{lem:immediate-bounds}\ref{itm:immediate-bounds-d2} and Lemma~\ref{lem:conj-lipschitz}, we get that the upper bound in~\eqref{EQ:fquad} is also true for $f^*$ on all of $\R^d$. In this case, we can simply take $L=2M$ and get similar results.
\end{proof}

There are many ways to leverage strong convexity in order to obtain faster rates of convergence, often known as \emph{fixed-point arguments}~\cite{Mas07, Kol11, GinNic16}. In this work, we employ van de Geer's ``one-shot" localization technique originally introduced in~\cite{Gee87} and stated in a form close to our needs in~\cite{Gee02}.

\subsection{Control of the stochastic error via empirical processes}

In light of Proposition~\ref{prop:gigli}, the performance of our estimator $\hat T_J=\nabla f_J$ defined in~\eqref{eq:jb} requires the control of $\cS(\hat f_J)-\cS(f_0)$, which can be achieved using tools from empirical process theory. 
To that end, for any $f$, define
$$
\cS_0(f)=\cS(f)-\cS(f_0) \qquad \text{and} \qquad \hat \cS_0(f)=\hat \cS(f)- \hat \cS(f_0)\,,
$$
and let \( \bar f_J \in \funcclass_J \).
We observe that by optimality of \( \hat f_J \) for \( \emprisk \),
\begin{equation}
	\label{EQ:basic-ineq}
	\cS_0(\hat f_J)  -\cS_0(\bar f_J) \le  \big[\cS_0(\hat f_J)-\hat \cS_0(\hat f_J)\big] + \big[\hat \cS_0(\bar f_J)- \cS_0(\bar f_J)\big].
\end{equation}
To proceed, we control the localized empirical process
$$
\sup_{f \in \cF_J\,:\, \cS_0(f) \le \tau^2} |\cS_0(f)-\hat \cS_0(f)|.
$$
for some fixed $\tau^2>0$. 
More precisely, we prove the following result in Appendix~\ref{sec:proof-expectation-bound}.

\begin{proposition}	\label{lem:expectation-bound}
	Let assumptions \ref{assumpgen:source} -- \ref{assumpgen:map} be fulfilled and define \( \funcclass_J \) as in \eqref{eq:fc}.
For any $\tau>0$, define
	\begin{equation}
		\mathcal{F}_{J}(\tau^2) := \{ f \in \funcclass_J : \risk_0(f)  \le \tau^2 \}.
	\end{equation}
Then, there exists $\constempproc = \constempproc(d, \Omega_\measureone, \tilde \Omega_\measureone, \Omega_\measuretwo, \constbd) >0$ such that with probability at least \( 1 - \exp(-t) \),
	\begin{equation}
		\label{eq:jo}
\sup_{\potential \in \mathcal{F}_{J}(\tau^2)} | \risk_0(\potential) - \emprisk_0(\potential) |
		\le \constempproc \left(\phi_{J}(\tau^2) + \tau \sqrt{\frac{t}{n}} +  \frac{t}{n}\right) \,,
	\end{equation}
	where
	\begin{align}
		\phi_J(\tau^2)
		= {} &  \frac{2^{J(d-2)/2} \, \tau}{\sqrt{n}} \sqrt{J \log\left(1 + \frac{\constempproc}{\tau}\right)}
		 + \frac{2^{J(d-2)} J}{n} \log\left(1 + \frac{\constempproc}{\tau}\right)
		+ \frac{\tau}{\sqrt{n}}.
	\label{EQ:defphi}
	\end{align}

\end{proposition}

Equipped with this result, we can apply van de Geer's localization technique. 
To simplify the presentation, assume that $\bar \potential = \bar \potential_J \in \argmin_{\potential \in \cF_J} \risk(\potential)$	
	exists. If not, we may repeat the proof with an \( \varepsilon \)-approximate minimizer and let $\eps \to 0$. 
Throughout the proof, we write  $\hat f=\hat f_J$ and $\|\cdot\|=\|\cdot\|_{L^2(P)}$. 

Fix \( \sigma>0 \) to be defined later and set
	\begin{equation}
		\label{eq:rb}
		\hat \potential_s = s \hat \potential + (1-s) \bar \potential, \qquad s = \frac{ \sigma}{\sigma + \norm{\nabla \hat \potential - \nabla \bar \potential}}\,,
	\end{equation}
Note that since \( s \in [0,1] \) and   \( \funcclass_J \) is convex, we have \( \hat \potential_s \in \cF_J \).

On the one hand, $\hat f_s$ is localized in the sense that
$$
		\norm{ \nabla\hat \potential_s - \nabla\bar \potential } = s \norm{ \nabla\hat \potential - \nabla\bar \potential } = \frac{\sigma \norm{ \nabla\hat \potential - \nabla\bar \potential }}{\sigma + \norm{\nabla \hat \potential -\nabla \bar \potential }} \le \sigma.
$$
By Proposition~\ref{prop:gigli} and the triangle inequality respectively, this yields
\begin{equation}
\label{eq:rd}
\cS_0(\hat f_s) \le  2M\norm{ \nabla\hat \potential_s - \nabla\potential_0 }^2 \le 4M \left(\sigma^2 + \norm{\nabla \bar \potential - \nabla\potential_0 }^2 \right)=:\tau^2\,.
\end{equation}
Therefore, $\hat f_s \in \cF_J(\tau^2)$. For the same reason, we also have that $\bar f\in \cF_J(\tau^2)$.

On the other hand, $\hat f_s$, akin to $\hat f$, has empirical risk smaller than $\bar f$. Indeed, by convexity of \( \emprisk \) and the fact that \( \hat \potential \) minimizes \( \emprisk \) over \( \funcclassgeneral \), we obtain
	\begin{align}
		\label{eq:re}
		\emprisk(\hat \potential_s) \le s \emprisk(\hat \potential) + (1-s) \emprisk(\bar \potential) \le \emprisk(\bar \potential),
	\end{align}
which yields
	\begin{align}
		\risk_0(\hat \potential_s)
		\le {} & \risk_0(\bar \potential) +2 \sup_{f \in \cF_J(\tau^2)}\big| \risk_0(f) - \emprisk_0(f) \big|.
		\label{eq:rr}
	\end{align}
Together with Jensen's inequality and Proposition~\ref{prop:gigli} respectively, the above display yields
\begin{align*}
	\norm{\nabla  \hat \potential_s - \nabla \bar \potential }^2 &\le 2\norm{\nabla  \hat \potential_s - \nabla  \potential_0 }^2+2\norm{\nabla   \potential_0 - \nabla \bar \potential }^2\le 16M\cS_0(\hat f_s) + 16M \cS_0(\bar f) \\
	&\le 32M \cS_0(\bar f) + 32M\sup_{f \in \cF_J(\tau^2)}\big| \risk_0(f) - \emprisk_0(f) \big|.
\end{align*}

Next,	note for \( s \) as in \eqref{eq:rb}, we have that $\norm{\nabla \hat \potential_s - \nabla \bar \potential }\ge \sigma/2$ iff $\norm{\nabla \hat \potential -\nabla \bar \potential} \ge \sigma$. Hence
	\begin{align*}
		\leadeq[1]{\p \big( \norm{ \nabla \hat \potential - \nabla \potential_0 } \ge  \sigma + \norm{ \nabla \bar \potential - \nabla \potential_0 } \big)\le  \p \big( \norm{\nabla  \hat \potential_s - \nabla \bar \potential }^2 \ge \sigma^2/4 \big)}\\
		&\le \p \big( \sup_{\potential \in \cF_J(\tau^2)} | \risk_0(\potential) - \emprisk_0(\potential) | \ge \frac{\sigma^2}{128 M}-\risk_0(\bar \potential)  \big)\\
		&= \p \big( \sup_{\potential \in \cF_J(\tau^2)} | \risk_0(\potential) - \emprisk_0(\potential) | \ge \frac{\tau^2}{512 M^2}-\frac{1}{128 M}\norm{\nabla \bar \potential - \nabla\potential_0 }^2-\risk_0(\bar \potential)  \big).
	\end{align*}
Recalling Proposition~\ref{lem:expectation-bound}, we take $\sigma$ such that
\begin{equation}
\label{EQ:choicetildesigma}
\frac{\tau^2}{512 M^2} \ge   \risk_0(\bar \potential)+\frac{1}{128 M}\norm{\nabla \bar \potential - \nabla\potential_0 }^2  +  \constempproc\left(\phi_{J}(\tau^2) + \tau \sqrt{\frac{\constbd \, t}{n}} + \frac{\constbd^2\, t}{n}\right)\,,
\end{equation}
so that we get
$$
\p \big( \norm{ \nabla \hat \potential - \nabla \potential_0 } \ge\sigma + \norm{ \nabla \bar \potential - \nabla \potential_0 } \big)\le e^{-t}\,.
$$
In particular, we can check that~\eqref{EQ:choicetildesigma} is fulfilled if we choose $\sigma$ such that
$$
\sigma^2
\gtrsim {\cS_0(\bar f)} +\frac{2^{J(d-2)} J}{n} \log\left(1 + \constlog n \right)+ \frac{1 + t}{n}\,,
$$
for a suitable choice of \( \constlog > 0 \).

With this, and applying Theorem \ref{prop:gigli} again, we get that with probability at least $1-e^{-t}$, it holds
$$
\norm{ \nabla \hat \potential - \nabla \potential_0 }^2 \lesssim  \norm{ \nabla \bar \potential - \nabla \potential_0 }^2 + \frac{2^{J(d-2)} J}{n} \log\left(1 + \constlog n \right)+ \frac{1+t}{n}\,.
$$
Moreover, integrating the tail with respect to $t$ readily yields by Fubini's theorem that
\begin{equation}
\label{EQ:bias-variance}
\E\norm{ \nabla \hat \potential - \nabla \potential_0 }^2 \lesssim  \norm{ \nabla \bar \potential - \nabla \potential_0 }^2 + \frac{2^{J(d-2)} J}{n} \log\left(1 + \constlog n \right)+ \frac{1}{n}\,.
\end{equation}

We have proved the following result.

\begin{proposition}	\label{prop:bias-variance}
Let \ref{assumpgen:source} -- \ref{assumpgen:map} hold and define \( \funcclass_J \) as in \eqref{eq:fc}. Then, writing
\begin{equation}
	\label{eq:a}
	E_J :=  \frac{2^{J(d-2)} J}{n} \log\left(1 + \constlog n \right)+ \frac{1}{n},
\end{equation}
the estimator $\hat T_J$ defined in~\eqref{eq:jb} satisfies
\begin{equation}
\label{EQ:bias-varianceE}
\E\norm{\hat T_J - T_0 }_{L^2(P)}^2
\lesssim  \inf_{\potential \in \funcclass_J} \| \nabla \potential -T_0 \|_{L^2(P)}^2  + E_J.
\end{equation}
	Moreover, with probability at least \( 1 - \exp(-t) \),
$$
\norm{\hat T_J - T_0 }_{L^2(P)}^2\lesssim  \inf_{\potential \in \funcclass_J} \| \nabla \potential - T_0 \|_{L^2(P)}^2 + E_J + \frac{t}{n}\,.
$$

\end{proposition}

\subsection{Control of the approximation error}
\label{sec:approx-error}

Next, we control the approximation error \( \inf_{\potential \in \funcclass_J} \| \nabla \potential - \nabla f_0 \|_{L^2(\measureone)} \) that appears in Proposition~\ref{prop:bias-variance}. 
In fact, it is sufficient to control \( \| \nabla \bar \potential - \nabla \potential_0 \|_{L^2(\measureone)} \) where $\bar f=\Pi_J \extension \potential_0$ is the truncation of \( \potential_0 \) to its first \( J \) wavelet scales after extending $f_0$ to all \( \R^d \). In light of Theorem~\ref{lem:extension}, we may assume that $\extension \bar f$ has the same  \( C^{\beta} \)- and \( W^{\alpha,2} \)-norm as $\bar f$ up to a constant depending on \( \tilde \Omega_\measureone \).

To control the approximation associated with truncating a wavelet decomposition, we rely on the following lemma for Besov functions.

\begin{lemma}
	\label{lem:wavelet-approximation}
	Let \( f \in B_{p, q}^{s}(\R^d) \) and denote by \( \Pi_J \) its projection onto the first scale \( J \) wavelet coefficients.
	That is, if
	\begin{equation}
		\label{eq:kv}
		f = \sum_{j=0}^{\infty} \sum_{g \in G^{j}} \sum_{k \in \mathbb{Z}^d} \gamma_k^{j, g} \Psi_k^{j, g}\,, \quad \text{we set } \quad
		\Pi_J f = \sum_{j=0}^{J} \sum_{g \in G^{j}} \sum_{k \in \mathbb{Z}^d} \gamma_k^{j, g} \Psi_k^{j, g},
	\end{equation}
	where \( \Psi^{j, g}_k \) are multi-dimensional Daubechies wavelets and \( G^j \) the associated index sets as in Section~\ref{sec:wavelets}.
	Then, for all \( 1 \le p, q \le \infty \), \( s \ge 0 \),
	\begin{align}
		\label{eq:kx}
		\| \Pi_J f \|_{B_{p, q}^s(\R^d)} {} & \le \| f \|_{B_{p, q}^s},\\
		\| \Pi_J f - f \|_{B_{p, q}^s(\R^d)} {} & \le \| f \|_{B_{p, q}^s}.
		\label{eq:ky}
	\end{align}
	Moreover, for every \( q' \in [1, \infty] \), \( s' > 0 \), and \( 1 \le p \le p' \le \infty \) such that $s-d/p>s'-d/p'$,
	\begin{equation}
		\label{eq:cd}
		\| \Pi_J f - f \|_{B^{s'}_{p', q'}} \lesssim 2^{-J(s - d/p - (s' - d/p'))} \| f \|_{B^{s}_{p, q}}.
	\end{equation}
	In particular: If \( f \in C^{\alpha + 1} \) for \( \alpha > 1 \), then $\| f - \Pi_J f \|_{C^2} \lesssim 2^{-J (\alpha - 1)} \| f \|_{C^{\alpha + 1}}$ and if \( f \in W^{\alpha + 1, 2} \), for \( \alpha > 0 \), then $\| f - \Pi_J f \|_{W^{1,2}} \lesssim 2^{-J \alpha} \| f \|_{W^{\alpha+1,2}}$.
\end{lemma}

\begin{proof}
  Write \( \gamma \) for the wavelet coefficients of \( f \).
	The statements \eqref{eq:kx} and \eqref{eq:ky} follow immediately from the wavelet characterization of the Besov norms, \eqref{eq:fe}.

	To prove the remaining statements, note that for every \( j \), because \( \| \, . \, \|_{\ell^{p'}} \le \| \, . \, \|_{\ell^p} \) and \( \| \, . \, \|_q \) and \( \| \, . \, \|_{q'} \) are comparable up to a constant due to \( |G^j| \le 2^{d} \) being finite,
	\begin{align}
		\label{eq:hq}
2^{j (s' + \frac{d}{2} - \frac{d}{p'})}& \Big[ \sum_{g \in G^j} \Big( \sum_{k \in \mathbb{Z}^d} | \gamma^{j, g}_k |^{p'} \Big)^{q' / p'} \Big]^{1/q'}\\
= {} & 2^{j(s' - \frac{d}{p'} - (s - \frac{d}{p}))}\, 2^{j (s + \frac{d}{2} - \frac{d}{p})} \Big[ \sum_{g \in G^j} \Big( \sum_{k \in \mathbb{Z}^d} | \gamma^{j, g}_k |^{p'} \Big)^{q' / p'} \Big]^{1/q'}\\
		\lesssim {} & 2^{j(s' - \frac{d}{p'} - (s - \frac{d}{p}))}\, 2^{j (s + \frac{d}{2} - \frac{d}{p})} \Big[ \sum_{g \in G^j} \Big( \sum_{k \in \mathbb{Z}^d} | \gamma^{j, g}_k |^{p} \Big)^{q / p} \Big]^{1/q}\\
		\le {} & 2^{j(s' - \frac{d}{p'} - (s - \frac{d}{p}))} \| f \|_{B^{s}_{p,q}}.
	\end{align}
	Then, plugging this into the wavelet expansion of \( \Pi_J f - f \), we obtain
	\begin{align}
		\label{eq:ho}
		\| \Pi_J f - f \|_{B^{s'}_{p', q'}}^{q'}
		= {} & \sum_{j = J+1}^\infty 2^{j q' (s' + \frac{d}{2} - \frac{d}{p'})} \sum_{g \in G^j} \Big( \sum_{k \in \mathbb{Z}^d} | \gamma^{j, g}_k |^{p'} \Big)^{q' / p'}\\
		\lesssim {} & \sum_{j=J+1}^{\infty} 2^{j q' (s' - \frac{d}{p'} - (s - \frac{d}{p}))} \| f \|_{B^s_{p,q}}^{q'}\lesssim 2^{J q' (s' - \frac{d}{p'} - (s - \frac{d}{p}))} \| f \|_{B^s_{p,q}}^{q'},
	\end{align}
	Finally, to obtain the special cases, note that \( \| \, . \, \|_{B^{s}_{\infty, \infty}} \lesssim \| \, . \, \|_{C^s} \lesssim \| \, . \, \|_{B^{s}_{1, \infty}} \) and \( \| \, . \, \|_{W^{s,2}} = \| \, . \, \|_{B^{s}_{2,2}} \) by Theorem~\ref{thm:restrict-extend}.
\end{proof}

The above lemma together with Proposition-Definition~\ref{lem:immediate-bounds} allows us to check that $\bar f \in \cF_J$. Indeed,	by Weyl's inequality, we have for any \( x \in \tilde \Omega_\measureone \) that
\begin{align}
\lambda_{\min} (D^2 \Pi_J \extension \potential_0(x))
\ge {} & \lambda_{\min} (D^2 \potential_0(x)) - \| D^2 \Pi_J \extension \potential_0(x) - D^2 \potential_0(x) \|_{\mathrm{op}}\\
\ge {} & \constbd^{-1} - C \| \Pi_J \extension \potential_0-f_0 \|_{C^2(\tilde \Omega_\measureone)}.
\end{align}
It follows from Lemma~\ref{lem:wavelet-approximation} that $\| \Pi_J \extension \potential_0 - \potential_0 \|_{C^2(\tilde \Omega_\measureone)} \lesssim 2^{-(\beta - 1)J} (\constbd^3 + \constsmooth)\le 1/(2\constbd)$, if
	\begin{equation}
		\label{eq:if}
		J \ge J_0:=\constjapprox \frac{1}{\beta - 1} \log \left( 2\constbd^4 + 2\constsmooth \constbd \right) \,,
	\end{equation}
	and \( \constjapprox = \constjapprox(d, \beta, \tilde \Omega_\measureone) \) is large enough.
	This yields $\lambda_{\min} (D^2 \Pi_J \extension \potential_0(x))\ge 1/(2M)$ and $\bar f$ is strongly convex. Similarly, we get $\lambda_{\max} (D^2 \Pi_J \extension \potential_0(x))\le 2M$ and hence that $\bar f \in \cF_J$.
	Thus,
	\begin{align}
		\inf_{\potential \in \funcclass_J} \| \nabla \potential - \nabla \potential_0 \|_{L^2(\measureone)}^2
		\le {} & \| \nabla \bar \potential - \nabla \potential_0 \|_{L^2(\measureone)}
		\lesssim \int_{\Omega_\measureone} \| \nabla \bar f - \nabla f_0 \|_2^2 \, \ud \lambda(x)\\
		\le {} & \| \bar f - f_0 \|_{W^{1,2}(\Omega_\measureone)}^2 \lesssim \constsmooth^2 \, 2^{-2 J\alpha}.
	\end{align}
	where we used Assumption \ref{assumpgen:map} and Lemma~\ref{lem:wavelet-approximation}.
	
	We have thus proved that
\begin{equation}
\label{EQ:approx-error}
\inf_{\potential \in \funcclass_J} \| \nabla \potential - \nabla \potential_0 \|_{L^2(\measureone)}^2 \lesssim \constsmooth^2 \, 2^{-2 J\alpha} \quad \text{for } J \ge J_0.
\end{equation}

\subsection{Bias-variance tradeoff}
\label{sec:upper-bounds-proof}

We are now in a position to complete the proof of Theorem~\ref{thm:upper-bounds}. Combining the bounds~\eqref{EQ:bias-varianceE} and~\eqref{EQ:approx-error}, we get
\begin{equation}
		\label{eq:ip}
		\E \|\hat T_J - T_0 \|_{L^2(\measureone)}^2 
		\lesssim \constsmooth^2 \, 2^{-2J\alpha} + 2^{J(d-2)}J \frac{\log(n)}{n} + \frac{1}{n}\,.
\end{equation}
We conclude the proof by optimizing with respect to $J$. It yields
\begin{equation}
	\label{eq:ub}
\E \| \hat T_J - T_0  \|_{L^2(\measureone)}^2	\lesssim \left\{
\begin{array}{ll}
\DS n^{-1} & \text{if } d=1,\\
\DS \log(\constsmooth) \, n^{-1}(\log (n))^2  & \text{if } d=2,\\
\DS \constsmooth^2 \, (1/\constsmooth^2)^{\frac{2 \alpha}{2 \alpha - 2 + d}} \log(\constsmooth) \, n^{\frac{-2 \alpha}{2 \alpha - 2 + d}} (\log(n))^2& \text{if } d\ge3.
\end{array}\right.
\end{equation}
We note that since \( \alpha > 1 \), in the first two cases, \( d \in \{1, 2\} \), the cut-off \( J \) can be picked independently from \( \alpha \).
Finally, high-probability bounds can be obtained in a similar manner.

\section{Numerical experiments}
\label{sec:numerics}

In this section, we provide numerical experiments with synthetic data that illustrate how leveraging the smoothness of the underlying transport map can lead to dramatically improved rates.
We give two estimators exploiting smoothness: the first, \( \hat T_{\mathrm{wav}} \) below, is an approximation to the estimator \( \hat T_J \) in \eqref{eq:jb}, illustrating that \eqref{eq:jb} can be implemented in low dimensions and that this approximation achieves favorable rates in \( d = 3 \).
The second, \( \hat T_{\mathrm{ker}} \) below, is a more practical heuristic two-step procedure based on smoothing the optimal matching between the empirical distributions via radial basis functions.
We compare these to a baseline estimator given by the optimal transport plan between the empirical distributions.

Additional implementation details and comments on these experiments are provided in Section \ref{sec:numerics-ctd} of the Appendix.

\subsection{Estimators}

\subsubsection{Baseline estimator}
In order to highlight the benefit of regularization, we consider the following simple estimator based on the optimal transport matrix between the empirical distributions as a baseline.
Denote the empirical distributions of \( \measureone \) and \( \measuretwo \) by
\begin{equation}
	\label{eq:tw}
	\hat \measureone = \frac{1}{n} \sum_{i = 1}^{n} \delta_{X_i}, \quad
	\hat \measuretwo = \frac{1}{n} \sum_{i = 1}^{n} \delta_{Y_i},
\end{equation}
respectively, and calculate the optimal transport matrix \( \Gamma \in \R^{n \times n} \)
\begin{align}
	\label{eq:tx}
	\hat \Gamma = \argmin \left\{ \sum_{i, j = 1}^{n} \| X_i - Y_j \|^2 \Gamma_{i, j} \, : \quad
		\begin{aligned}
	&\Gamma_{i, j} \ge 0 \; \forall i, j \in [n], \;\\
	&\Gamma \1 = \1/n, \; \Gamma^\top \1 = \1 / n
	\end{aligned}
	\right\}
\end{align}
which can be solved exactly by linear optimization toolkits or approximated by entropic regularization \cite{PeyCut19}.
An estimated transport function on the observations \( X_i \) is then obtained by
\begin{equation}
	\label{eq:te}
	\hat \tmap_{\mathrm{emp}}(X_i) = n \sum_{j = 1}^{n} \hat \Gamma_{i, j} Y_j,
\end{equation}
which corresponds to the conditional mean of the coupling given \( X_i \).
Note that since we assume that the sample sizes from \( \measureone \) and \( \measuretwo \) are both \( n \), the optimal \( \hat \Gamma \) is in fact a (rescaled) permutation matrix, leading to a matching \( \hat \pi : [n] \to [n] \), and hence to \( \hat \tmap_{\mathrm{emp}}(X_i) = Y_{\pi(i)} \).

Because \( \hat \tmap_{\mathrm{emp}}(X_i) \) above is only defined on the sample points and we do not want to introduce additional bias against the estimator, we consider the following error measure, approximating the \( L^2(\measureone) \) norm analyzed in Theorem \ref{thm:minimax}:
\begin{equation}
	\label{eq:ty}
	\mathrm{MSE}_n(\hat T_\mathrm{emp}) = \frac{1}{n} \sum_{i = 1}^{n} \Big\| \hat \tmap_{\mathrm{emp}} (X_i) - \tmap_0(X_i) \Big\|_2^2.
\end{equation}

\subsubsection{Wavelet estimator}

Next, we turn to an approximation of \eqref{eq:jb}.
Assume that in addition to a superset \( \tilde \Omega_\measureone \supseteq \Omega_\measureone \), we are also given a superset \( \tilde \Omega_\measuretwo \supseteq \Omega_\measuretwo \), and that both \( \tilde \Omega_\measureone \) and \( \tilde \Omega_\measuretwo \) are boxes (hypercubes).
We consider all functions originally defined over \( \tilde \Omega_\measureone \) and \( \tilde \Omega_\measuretwo \) as given by their samples on grids with resolution \( N \in \mathbb{N} \), \( \xvec = (x_i)_{i \in [N]^d} \in (\R^{d})^{N^d} \) and \( \yvec = (y_i)_{i \in [N]^d} \in (\R^d)^{N^d} \), respectively.
In particular, we write \( \fvec = (f_i)_{i \in [N]^d} \in \R^{N^d} \) for the discretization of the potential \( f \) and \( \Tvec = (T_i)_{i \in [N]^d} \in (\R^d)^{[N]^d} \) for the discretization of the transport map \( T \) on the grid \( \xvec \).
Here, we pick \( N = 65 \).

We employ the following discretization/approximation schemes:
\begin{enumerate}
	\item The restrictions to functions up to wavelet scale \( J \in \mathbb{N} \) can be obtained by parametrizing \( \fvec \) by the inverse discrete wavelet transform up to order \( J \), which we write as \( \fvec = \invwavtrafo_J \gamma_J \) for wavelet coefficients \( \gamma_J \in \R^{m_J} \) with \( m_J \in \mathbb{N} \).
		For these experiments, we use db4 Daubechies wavelets, i.e., Daubechies wavelets with four vanishing moments.
	\item An approximation to the (continuous) Legendre transform is given by the discrete Legendre transform,
		\begin{align}
			\label{eq:ua}
			\legtrafo(\fvec)_j := {} & \legtrafo(\fvec)(y_j) := \legtrafo_{\xvec \to \yvec}(\fvec)(y_j) \\
			:= {} & \sup_{i \in [N]^d} \{ \langle x_i , y_i \rangle - f_i :  i \in [N]^d \},
		\end{align}
		for \( \fvec \in (\R^d)^{N^d} \) and \( j \in [N]^d \).
		\noeqref{eq:ua}
	\item The gradient of \( \fvec \) on grid-points can be calculated by a finite-difference scheme, which we write as \( \nabla_{\xvec} \fvec \).
	\item To obtain values of these approximations at non-grid points, we appeal to linear interpolation, which we write as \( \interp_{\xvec \to \{X_1, \dots, X_n\}} \fvec \) for the interpolated value of \( \fvec \) defined on the grid \( \xvec \) at the point \( \{X_1, \dots, X_n\} \in \R^d \), collected in one vector.
\end{enumerate}

Given \iid observations \( X_{1:n} = \{X_1, \dots, X_n\} \) and \( Y_{1:n} = \{Y_1, \dots, Y_n\} \) from \( \measureone \) and \( \measuretwo \), respectively, we arrive at the following optimization problem as approximation to \eqref{eq:jb}:
\begin{align}
	\label{eq:td}
	\hat \gamma_J = {} & \min_{\wavcoeffs_J \in \R^{m_J}} \frac{1}{n} \1^\top \interp_{\xvec \to X_{1:n}} \, \invwavtrafo_J \wavcoeffs_J + \frac{1}{n} \1^\top \interp_{\yvec \to Y_{1:n}} \, \legtrafo(\invwavtrafo_J \wavcoeffs_J).
\end{align}
Note that we dropped all boundedness and convexity constraints that were given by \( \funcclassstd(2M) \) in \eqref{eq:jb}.
In practice, this can lead to degraded estimation quality of the gradient of \( \invwavtrafo_J \gamma_J \) near the boundary of \( \Omega_\measureone \), see Section \ref{sec:double-legendre} in the Appendix, which we remedy by computing the convex envelope of the resulting estimator, yielding an estimator of \( \potential_0 \) that is convex.
This envelope can, for example, be calculated by applying the Legendre transform twice, and we set
\begin{equation}
	\label{eq:tf}
	\hat \Tvec_J = \nabla_{\xvec} [\legtrafo_{\yvec \to \xvec}(\legtrafo_{\xvec \to \yvec}(\invwavtrafo_J \hat \wavcoeffs_J))].
\end{equation}

With this, we denote by \( \hat \tmap_{\mathrm{wav}}^{(J)} \) the function obtained by linearly interpolating \( \hat \Tvec_{J} \),
\begin{equation}
	\label{eq:ue}
	\hat T_{\mathrm{wav}}^{(J)}(x) = \interp_{\xvec \to x} \hat \Tvec_{J}, \quad x \in \tilde \Omega_\measureone.
\end{equation}
For the purpose of these experiments, we select the wavelet scale \( \hat J \) by an oracle choice, i.e., as the minimizer of an approximation to the population semi-dual \eqref{eq:pn}, see Section \ref{sec:implementation}, while in practice, one would resort to cross-validation methods for this purpose.
Finally, we set
\begin{equation}
	\label{eq:e}
	\hat T_{\mathrm{wav}} = \hat T_{\mathrm{wav}}^{(\hat J)}.
\end{equation}

To obtain an error measure for \( \hat \tmap_{\mathrm{wav}} \) that is easily comparable to \( \mathrm{MSE}_n(\hat T_{\mathrm{emp}}) \), we consider the empirical \( L^2(\hat P) \) norm on the sample points,
\begin{equation}
	\label{eq:tz}
	\mathrm{MSE}_n(\hat T_{\mathrm{wav}}) = \frac{1}{n} \sum_{i = 1}^{n} \Big\| \hat \tmap_{\mathrm{wav}}(X_i) - \tmap_0(X_i) \Big\|_2^2.
\end{equation}

Note that the objective function in \eqref{eq:td} can be calculated in linear time with respect to the underlying grid, that is, in \( O(N^d) \), thanks to efficient algorithms for the discrete wavelet transform \cite{Mal99} and the Linear-Time Legendre Transform algorithm \cite{Luc97}.
It can be checked that the objective is convex, rendering first-order methods provably convergent.
We use the L-BFGS Quasi-Newton method to find \( \hat \gamma_J \), even though the objective is not smooth for every \( \gamma_J \) due to the form of the discrete Legendre transform.
In practice, we observe that it converges faster than simple \mbox{(sub-)gradient} descent methods.

Since \( \hat T_{\mathrm{wav}} \) is mainly used to illustrate the practical behavior of a wavelet-based regularization of the semi-dual objective, we do not explicitly analyze the convergence of \( \hat T_{\mathrm{wav}} \) to the estimator \( \hat T_J \) considered in Section~\ref{sec:upper}.
We remark, however, that our approximations closely follow the definition of \( \hat T_J \) and that the omitted constraints defining the set \( \funcclassstd(2M) \) could be incorporated into the approximation by means of a finite-difference discretization as well, albeit at an additional computational cost.

\subsubsection{Kernel smoothing estimator}

While the estimator \( \hat \tmap_{\mathrm{wav}} \) closely follows our theoretical analysis, its applicability is limited to low dimensions by the fact that a discretization grid is needed.
As a heuristic alternative, we consider smoothing the assignment obtained by \( \hat \tmap_{\mathrm{emp}} \).
The idea of smoothing the empirical transport matrix has been previously used in practice, see for example \cite{SchShuTab19}, where regularized linear regression was used as a post-processing step, and kernels have been applied for smoothing potential functions in \cite{GenCutPey16}.
Here, we obtain a smoother version of \( \hat \tmap_{\mathrm{emp}} \) by smoothing it via kernel-ridge regression \cite{Mur12}.

Let \( \rkhs \) denote a reproducing kernel Hilbert space (RKHS) with associated kernel \( k(x, y) \), \( x, y \in \R^d \), and norm \( \| T \|_{\rkhs} \) for \( T \in \rkhs \).
Here, we consider the RKHS given by Gaussian radial basis functions,
\begin{equation}
	\label{eq:ui}
	k(x, y) = \exp(-\nu_{\mathrm{kernel}} \| x - y \|^2), \quad x, y \in \R^d, \quad \nu_{\mathrm{kernel}} > 0.
\end{equation}
We fit an RKHS function \( \tmap \) to the pairs \( (X_i, \tilde Y_i = \hat \tmap_{\mathrm{emp}}(X_i)) \) by solving the regularized kernel regression objective
\begin{equation}
	\label{eq:uf}
	\hat T_{\mathrm{ker}}^{(\nu_{\mathrm{ridge}}, \nu_{\mathrm{kernel}})} = \argmin_{T \in \rkhs} \sum_{i = 1}^{n} \| \tilde Y_i - T(X_i) \|_{2}^2 + \nu_{\mathrm{ridge}} \| T \|_{\rkhs}^2,
\end{equation}
for \( \nu_{\mathrm{ridge}} > 0 \).
By the representer theorem \cite{Mur12}, \eqref{eq:uf} has a solution
\begin{equation}
	\label{eq:f}
	\hat T_{\mathrm{ker}}^{(\nu_{\mathrm{ridge}}, \nu_{\mathrm{kernel}})}(x) = \sum_{i = 1}^{n} \hat w_i k(x_i, x),
\end{equation}
where 
\begin{equation}
	\label{eq:uh}
	\hat W = \begin{pmatrix}
		\hat w_1^\top \\
		\vdots\\
		\hat w_n^\top
	\end{pmatrix}
	= (K + \nu_{\mathrm{ridge}} I)^{-1} \tilde Y, \quad
	\text{with } K_{i, j} = k(X_i, X_j)
	\text{ and }
	\tilde Y =
	\begin{pmatrix}
		\tilde Y_1\\
		\vdots\\
		\tilde Y_n
	\end{pmatrix}.
\end{equation}
We measure its performance by
\begin{equation}
	\label{eq:ug}
	\mathrm{MSE}_n(\hat T_{\mathrm{ker}}^{(\nu_{\mathrm{ridge}}, \nu_{\mathrm{kernel}})}) = \frac{1}{n} \sum_{i = 1}^{n} \Big\| \hat \tmap_{\mathrm{ker}}^{(\nu_{\mathrm{ridge}}, \nu_{\mathrm{kernel}})}(X_i) - \tmap_0(X_i) \Big\|_2^2.
\end{equation}
Similar to \( \hat T_{\mathrm{wav}} \), we select the tuning parameters \( \nu_{\mathrm{kernel}} \) and \( \nu_{\mathrm{ridge}} \) by an oracle procedure, picking those parameters from a finite grid that minimize \eqref{eq:ug} on an independent hold-out sample \( \tilde X_{1:n} \), and denote the resulting estimator by \( \hat T_{\mathrm{ker}} \).

\begin{figure}[!ht]
	\centering
	\subcaptionbox{
		Ground truth transport map
		\label{fig:vis_grad_gt_interior}
		}{
		\includegraphics[width=0.4\textwidth]{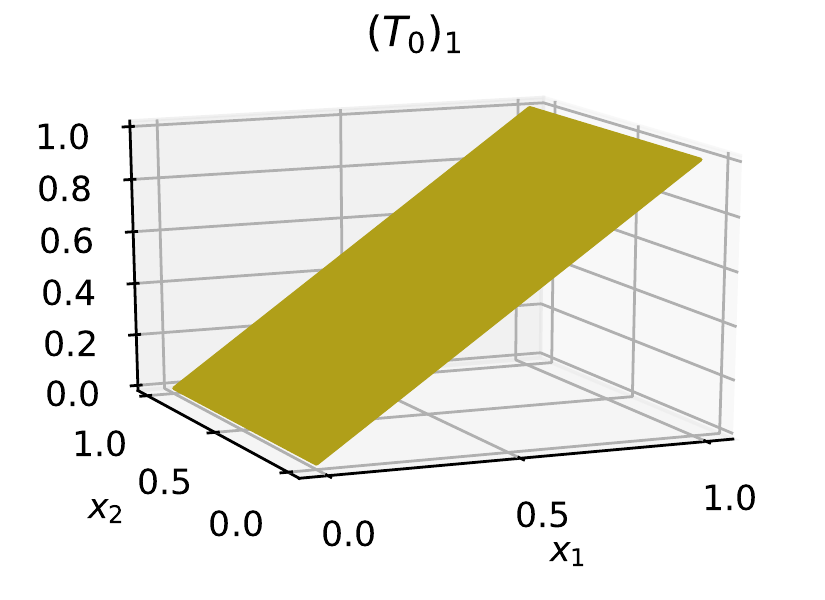}
	}
	\subcaptionbox{
		Transport map estimator from empirical distributions and 1-NN interpolation
		\label{fig:vis_tmap_vanilla}
		}{
		\includegraphics[width=0.4\textwidth]{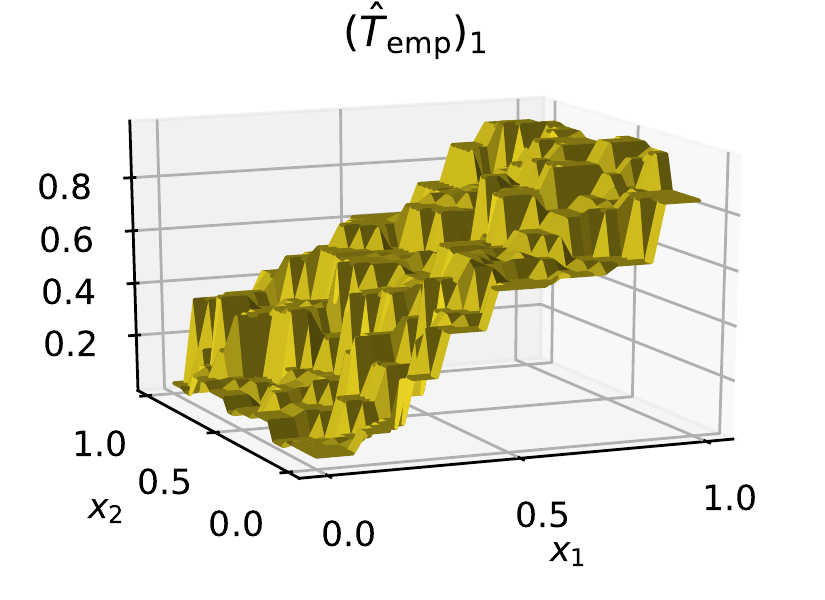}
	}
	\subcaptionbox{
		Wavelet estimator, \( J = 1 \)
		\label{fig:vis_grad_double_interior}
		}{
		\includegraphics[width=0.4\textwidth]{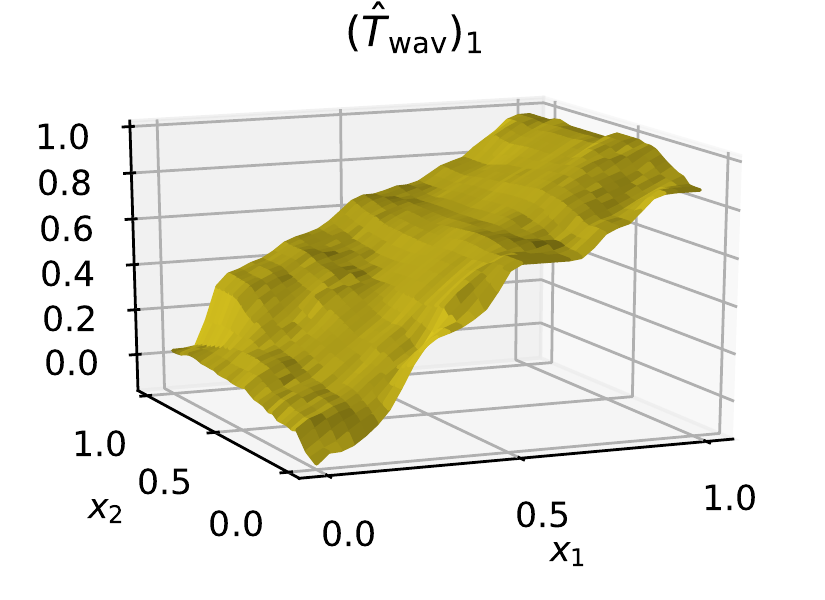}
	}
	\subcaptionbox{
		Kernel estimator, \( \nu_{\mathrm{ridge}} = \nu_{\mathrm{kernel}} = 10^{-3} \)
		\label{fig:vis_grad_kernel}
		}{
		\includegraphics[width=0.4\textwidth]{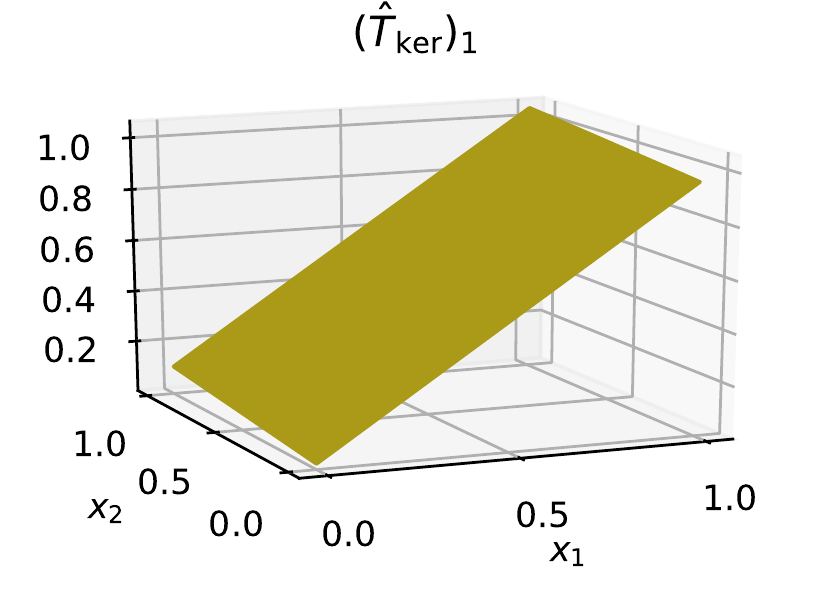}
	}

	\caption{Qualitative comparison between \( \hat T_{\mathrm{emp}} \), \( \hat T_{\mathrm{wav}} \), and \( \hat T_{\mathrm{ker}} \).
		Both the wavelet-based estimator \( \hat T_{\mathrm{wav}} \) and the kernel estimator \( \hat T_{\mathrm{ker}} \) produce a qualitatively smoother output then the optimal coupling between the empirical measures. 
	Visualizations of the first coordinate of the transport maps.}
	\label{fig:vis_smooth}
\end{figure}

\subsection{Setup}
\label{sec:numerics-setup}

For \( d \in \mathbb{N} \), we consider the following examples of smooth potentials and transport maps:
\begin{alignat}{2}
	\qquad \potential_0^{(1)}(x) = {} & \frac{1}{2} \| x \|_2^2, & \quad \tmap_0^{(1)}(x) = {} & x, \quad x \in \R^d; \tag{id} \label{eq:identity}\\
	\qquad \potential_0^{(2)}(x) = {} & \sum_{i=1}^{d} \exp(x_i), & \quad (\tmap_0^{(2)}(x))_i = {} & \exp(x_i), \quad x \in \R^d, \; i \in [d]; \tag{exp} \label{eq:exponential}
\end{alignat}
where for \eqref{eq:identity}, \( \measureone^{(1)} = \measuretwo^{(1)} = \unif([0,1]^d) \).
For \eqref{eq:exponential}, \( \measureone^{(2)} = \unif([0,1]^d) \), and the target measure is defined as \( \measuretwo^{(2)} = (\tmap_0^{(2)})_\# \measureone^{(2)} \).
Note that these potentials and transport maps are \( C^{\infty} \) and strongly convex on any compact convex subset of \( \R^d \).
For the purpose of qualitative comparisons, we consider case \eqref{eq:identity} in \( d = 2 \).
In order to determine the quantitative behavior of the estimators, we study both cases for \( d \in \{3, 10\} \).

The runtime of computing the optimal transport matching via linear programming scales unfavorably with the sample size, a shortcoming that could be remedied by employing recent numerical approximation techniques \cite{AltBacRud19}.
Similarly, computing the kernel regression \eqref{eq:uf} with off-the-shelf methods scales with \( O(n^3) \).
For the sake of this comparison, we simply restrict our experiments on \( \hat T_{\mathrm{emp}} \) and \( \hat T_{\mathrm{ker}} \) to sample sizes \( \le 10^4 \).
Likewise, we do not compute \( \hat T_{\mathrm{wav}} \) in \( d = 10 \) due to the large computational cost.

\subsection{Results}

\subsubsection{Qualitative comparison in 2D}

To give a qualitative idea of the considered estimators compared to the baseline \( \hat T_{\mathrm{emp}}\), we visualize the first coordinate of the transport map estimators for case \eqref{eq:identity} with \( d = 2 \) and \( n = 100 \) observations from \( \measureone = \measuretwo = \unif([0,1]^2) \) in Figure \ref{fig:vis_smooth}, together with the ground truth transport map \( (T_0^{(1)})_1 \).
To depict the coupling \( \hat T_{\mathrm{emp}} \) induced by the empirical distributions, we employ 1-nearest-neighbor (1-NN) interpolation to obtain a map on \( [0,1]^2 \).
We observe that the wavelet-based regularization in Figure \ref{fig:vis_grad_double_interior} produces a visibly smoother map compared to the unregularized \( \hat T_{\mathrm{emp}} \) in Figure \ref{fig:vis_tmap_vanilla}.
The kernel estimator in Figure \ref{fig:vis_grad_kernel} is even smoother and visually very similar to the ground truth in Figure \ref{fig:vis_grad_gt_interior}, due to the possibility of employing a large amount of regularization.

\begin{figure}[!ht]
	\centering
	\subcaptionbox{
		Identity transport map, \( d = 3 \)
		\label{fig:errors_uniform_3d}
	}{
		\includegraphics[width=0.45\textwidth]{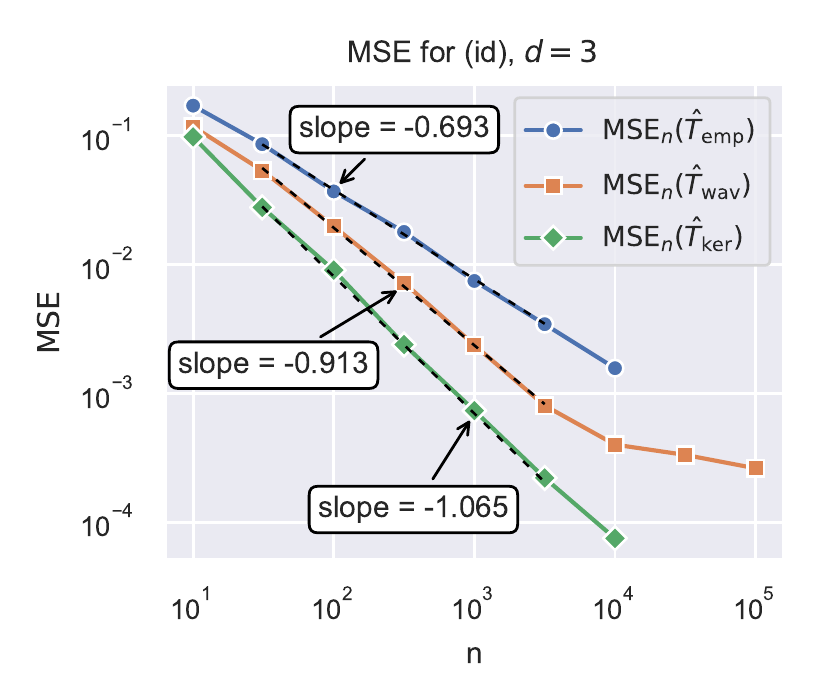}
	}
	\subcaptionbox{
		Identity transport map, \( d = 10 \)
		\label{fig:errors_uniform_10d}
	}{
		\includegraphics[width=0.45\textwidth]{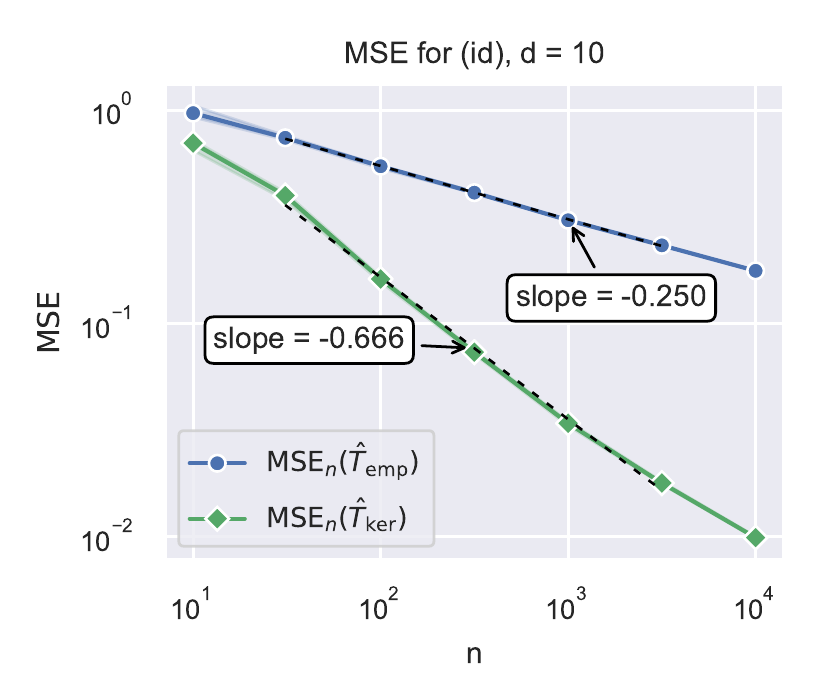}
	}
	\subcaptionbox{
		Exponential transport map, \( d = 3 \)
		\label{fig:errors_exp_3d}
	}{
		\includegraphics[width=0.45\textwidth]{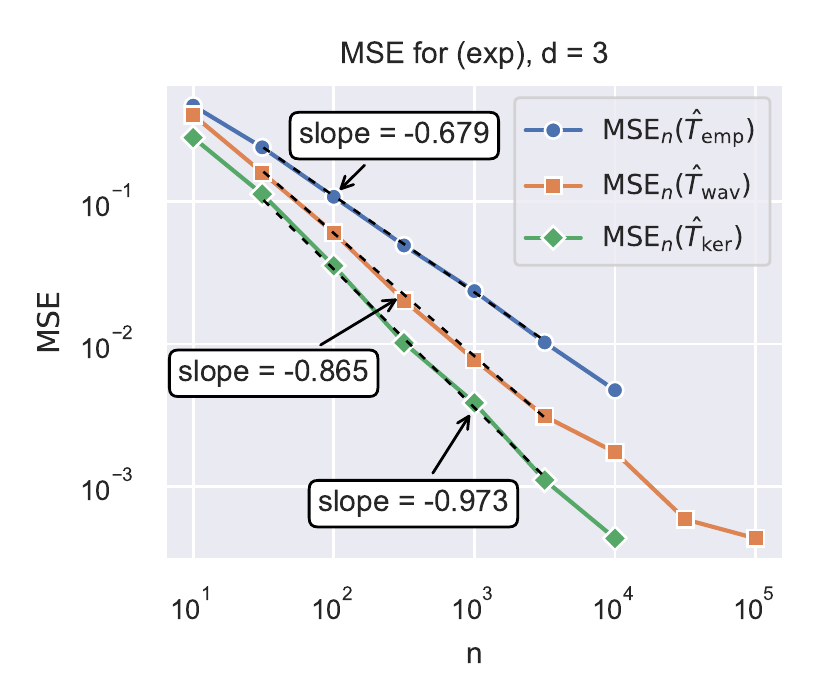}
	}
	\subcaptionbox{
		Exponential transport map, \( d = 10 \)
		\label{fig:errors_exp_10d}
	}{
		\includegraphics[width=0.45\textwidth]{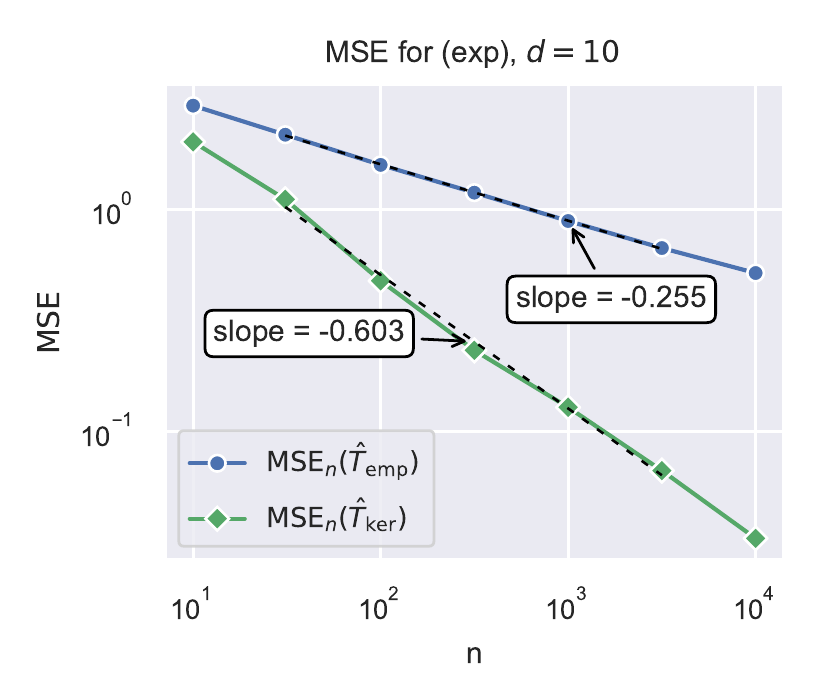}
	}
	\caption{Log-log plot of MSE plotted against \( n \), showing the median error over 32 replicates. See Section \ref{sec:numerics-setup} for details.}
	\label{fig:errors}
\end{figure}

\subsubsection{Quantitative comparison in 3D and 10D}
To obtain a quantitative comparison, for both test cases and \( d \in \{3, 10\} \), we compute \( \hat T_{\mathrm{emp}} \), \( \hat T_{\mathrm{wav}} \) (only \( d = 3 \)), and \( \hat T_{\mathrm{ker}} \) over 32 replicates and a logarithmically spaced selection of sample sizes \( n \), reporting the median error over the replicates in Figure \ref{fig:errors}.
Here, the dashed lines indicate the result of linear regression on the logarithmically transformed sample sizes and error results for a selected subset of \( n \).

In 3D, for both test cases, the error curves for the standard empirical measure-based estimator \( \hat T_{\mathrm{emp}} \) roughly follow a \( n^{-2/3} \) rate.
This corresponds to the decay of the average \( \ell^2 \) cost of optimal matchings between samples from the uniform distribution on the cube \cite{Tal14, Yuk06} (and Gaussian distributions \cite{Led19}), and also matches the \( n^{-2/d} \) rate for the convergence of \( W_2^2(\hat P, P) \).
We are, however, not aware of a generalization of this rate to the error measure \( \mathrm{MSE} \) in the case of a transport map different from the identity.

The error curves for the wavelet estimator \( \hat T_{\mathrm{wav}} \) all follow a similar trend.
For low sample sizes, we obtain rates faster than \( n^{-0.85} \), showing the large statistical benefit of fitting only functions that have wavelet expansions of low order.
For large sample sizes, the error flattens out, which can be explained by the numerical approximation errors dominating the statistical ones.
This can be readily seen from repeating the experiment with a smaller grid resolution (\( N = 33 \)), which shows the same trend for smaller values of \( n \), see Section \ref{sec:error-flattening} in the Appendix.
Moreover, for all sample sizes we considered, the error curves for \( \hat T_{\mathrm{wav}} \) all lie below the baseline estimator.
The kernel estimator \( \hat T_{\mathrm{ker}} \) performs even better, attaining rates close to \( n^{-1} \) and yielding consistently better rates than \( \hat T_{\mathrm{wav}} \).

We observe that the favorable behavior of \( \hat T_{\mathrm{wav}} \) suggests that the restriction of candidate potentials to \( \funcclassstd(2M) \) in \eqref{eq:jb} might not be necessary and could possibly be omitted.

In 10D, for both test cases, \( \hat T_{\mathrm{emp}} \) shows a convergence rate of about \( n^{-0.25} \), which is slightly better than the expected \( n^{-2/d} = n^{-0.2} \) rate.
It is vastly outperformed by the kernel-based estimator that achieves rates better than \( n^{-0.6} \) in both examples.

Further plots showing individual error curves for varying values of the regularization parameters can be found in Section \ref{sec:num-params} of the appendix, illustrating the sample size-dependent performance gain achieved by \( \hat T_{\mathrm{wav}} \) and \( \hat T_{\mathrm{ker}} \).

To summarize, in cases where smoothness of the transport map can be assumed, its estimation greatly benefits from smoothness regularization.
In particular, these experiments suggest further research on proving error bounds for the kernel estimator \( \hat T_{\mathrm{ker}} \) under regularity assumptions on \( T_0 \), for which the minimax rates established in this work can serve as a benchmark.

\medskip
\noindent{\bf Acknowledgments.} The authors would like to thank Richard Nickl for pointing out relevant references on the estimation of invariant measures in the multivariate Gaussian white noise model, and Jonathan Niles-Weed for suggesting references on the rate between random matchings. The authors are also indebted to the Associate Editor and three referees for providing insightful comments that greatly improved the quality of the paper.

	Philippe Rigollet was supported by NSF awards IIS-1838071, DMS-1712596, DMS-TRIPODS- 1740751, and ONR grant N00014-17- 1-2147.
\appendix

\section{Convex analysis}
\label{SEC:convex-analysis}

In this section, we recall some useful facts from convex analysis. We refer the reader to~\cite{HirLem01} for a comprehensive treatment.

Recall that a set \( \convexset \subseteq \R^d \) is convex if for all \( x, y \in \convexset \), \( t \in [0,1] \), \( t x + (1-t) y \in \convexset \).
A function \( \potential \colon \convexset \to \R \cup \{+\infty\} \) is  convex if for all \( x, y \in \convexset \), \( t \in [0,1] \), it holds that
\begin{equation}
	\label{eq:mw}
	\potential(t x + (1 - t) y) \le t \potential(x) + (1-t) \potential(y).
\end{equation}
Moreover, we call a function \( \mu \)-strongly convex if for all \( x, y \in \convexset \), \( t \in [0,1] \), we have
\begin{equation}
	\label{eq:mx}
	\potential(t x + (1 - t) y) \le t \potential(x) + (1-t) \potential(y) - \frac{\mu}{2} t (1-t) \| x - y \|_2.
\end{equation}

For twice differentiable functions, it is often more convenient to employ the following analytic criterion for strong convexity.
\begin{lemma}
	[{\cite[Theorem B.4.3.1]{HirLem01}}]
	\label{lem:diffconvex}
	Let \( \convexset \subseteq \R^d \) be an open convex set and \( \potential \colon \convexset \to \R \) twice differentiable.
	Then, \( \potential \) is \( \mu \)-strongly convex if and only if $\lambda_{\min}(D^2 \potential(x)) \ge \mu$ for all  $x \in \convexset$.
\end{lemma}

Note that even if \( \convexset \subsetneq \R^d \) is a proper convex subset of \( \R^d \), we can always consider a function \( \potential \colon \convexset \to \R \cup \{+\infty\} \) to be defined on all of \( \R^d \) by setting it to \( + \infty \) outside of \( \convexset \). To that end, let $\iota_{U}$ be the \emph{indicator} function defined by 
 \begin{equation}
	\label{eq:mn}
	\iota_{U}(x) =
	\left\{
	\begin{aligned}
		&0, &\quad &x \in U \\
		&+\infty, &\quad &\text{otherwise.}
	\end{aligned}
	\right.
\end{equation}
We define the extension of $f$ outside $U$ by $f+\iota_U$, which by abuse of notation we also denote by $f$.
Note that if $f$ is (strongly) convex on $U$ then its extension outside $U$ is also (strongly) convex.
We call $\domain(\potential) = \{ x \in \R^d : f(x) < + \infty \}$ the domain of a convex function.

We now recall two important notions associated with  convex functions $f$.

First, the subdifferential of $f$ at $x \in \domain(f)$ is defined as
\begin{equation}
	\label{eq:nb}
	\partial f(x) = \{ a \in \R^d : f(y) \ge \langle a , y - x \rangle + f(x) \text{ for all } y \in \R^d \}.
\end{equation}
As indicated by the following lemma, the subdifferential reduces to the gradient for differentiable functions.
\begin{lemma}
	[{\cite[Corollary D.2.1.4]{HirLem01}}]
	Let \( \potential \colon \R^d \to \R \cup \{ + \infty \} \) be a convex function.
	If \( \potential \) is differentiable at \( x \in \R^d \) with gradient \( \nabla \potential(x) \), then \( \partial \potential(x) = \{ \nabla \potential(x) \} \).

	Conversely, if \( \partial \potential(x) = \{ a \} \) consists of only a single element, then \( \potential \) is differentiable at \( x \) with gradient \( \nabla \potential(x) = a \).
\end{lemma}

Second, the convex conjugate, or Legendre-Fenchel conjugate, is defined for any function \( \potential \colon \R^d \to \R \cup \{ + \infty \} \) as
\begin{equation}
	\label{eq:my}
	\potential^\ast(y) = \sup_{x \in \R^d} \langle x , y \rangle - \potential(x), \quad y \in \R^d.
\end{equation}
By considering $f+\iota_U$, this definition extends to functions \( \potential \colon \convexset \to \R \cup \{ + \infty \} \).

We recall the following standard facts about the convex conjugate, stated here without proof (see~\cite[Part E]{HirLem01} for details).

\begin{lemma}
	\label{lem:biconjugate}
	If \( \potential \colon \R^d \to \R \cup \{ + \infty \} \) is convex and lower semi-continuous, then \( \potential^{\ast \ast} = \potential \).
\end{lemma}

\begin{lemma}
	\label{lem:conj-lipschitz}
	Let \( \convexset \) be a closed, convex set.
	If \( f \colon \convexset \to \R \) is \( \mu \)-strongly convex, then \( \domain(f^\ast) = \R^d \), and \( \nabla f^\ast \) is \( \mu^{-1} \)-Lipschitz:
	$$
	\| \nabla f^\ast(x) - \nabla f^\ast(y)\|_2 \le \frac1\mu \|x-y\|_2\,, \qquad \forall\, x,y \in \R^d\,.
	$$
\end{lemma}

\begin{lemma}
	\label{lem:sub-der-correspondence}
	Denote by \( \potential \colon \R^d \to \R \cup \{ +\infty \} \) a lower semi-continuous and convex function.
	Then, for \( x,y \in \R^d \),
	\begin{equation}
		\label{eq:mm}
		\potential^\ast(y) = \langle x , y \rangle - f(x)
		\iff y \in \partial \potential(x) \iff x \in \partial \potential^\ast(y).
	\end{equation}
\end{lemma}

\begin{lemma}
	\label{lem:conjugate-order}
	If \( f, g \colon \convexset \to \R \), \( f(x) \leq g(x) \) for all \( x \in \convexset \), then \( f^\ast(y) \geq g^\ast(y) \) for all \( y \in \R \).
\end{lemma}

\begin{lemma}
	\label{lem:g}
	Given two functions \( f, g : \convexset \to \R \) for a compact set \( \convexset \).
	Abusing notation, denote the convex conjugate of the function \( f + \iota_{\convexset} \) by \( f^\ast \).
	Then, $\| f^\ast - g^\ast \|_{L^\infty(\R^d)} \le \| f - g \|_{L^\infty(\convexset)}$.
\end{lemma}

The next lemma provides an explicit form for the convex conjugate of a quadratic function.
It is instrumental in the stability proof for the semi-dual objective function (see Proposition~\ref{prop:gigli}).

\begin{lemma}
\label{lem:quadratic-conjugate}
	Let \( a > 0 \), \( b, t \in \R^d \), \( c \in \R \) and let \( \convexset \subseteq \R^d\) be a closed, convex set.
Define
		\begin{equation*}
			q_t(x) =\frac{a}{2} \| x-t\|_2^2 + \langle b, x-t\rangle + c +\iota_U(x)\,, \quad x \in \R^d\,.
		\end{equation*}
		Then,
		\begin{equation}
			\label{eq:eq}
			q_t^\ast(y) 
			= \frac{\| y - b \|_2^2}{2a} + \langle t, y\rangle- c - \frac{a}{2} d^2\left( \frac{y - b}{a}-t , \convexset \right),
		\end{equation}
		where \( d^2 \) denotes the squared distance $d^2(x, \convexset) = \inf_{y \in \convexset} \| x - y \|_2^2$.
\end{lemma}

\begin{proof}
Note first that for any $y \in \R^d$,
\begin{equation}
\label{EQ:ess}
q_t^*(y)=\sup_{x \in \R^d} \langle x , y \rangle - q(x-t)=\sup_{x \in \R^d} \langle x+t , y \rangle - q(x)=q^*(y)+ \langle t, y\rangle\,,
\end{equation}
where 
$$
q(x)=\frac{a}{2} \| x\|_2^2 + b^\top x + c +\iota_{U+t}(x)
$$
Moreover, 
	\begin{equation}
		\label{eq:es}
		q^\ast(y) = \sup_{x \in \R^d} \langle x , y \rangle - q(x)
		= - \inf_{x \in \R^d} \{ q(x) - \langle x , y \rangle \}.
	\end{equation}
  Writing
	$$
	q(x) - \langle x , y \rangle= \frac{a}{2} \big\| x - \frac{y-b}{a} \big\|_2^2 - \frac{\| y - b \|_2^2}{2 a} + c+ \iota_{U+t}(x)
	$$
	we see that the infimum in~\eqref{eq:es} is achieved by the projection $\bar x$ of $(y-b)/a$ onto the closed convex set $U+t$. Moreover, the value of the objective at $\bar x$ is given by
	\begin{align}
	q(\bar x) - \langle \bar x , y \rangle
	= {} & \frac{a}{2}d^2\left( \frac{y-b}{a}, \convexset+t \right)- \frac{\| y - b \|_2^2}{2a} + c\\
	= {} & \frac{a}{2}d^2\left( \frac{y-b}{a}-t, \convexset \right)- \frac{\| y - b \|_2^2}{2a} + c
	\end{align}
Together with~\eqref{EQ:ess} and~\eqref{eq:es}, this completes the proof of the lemma. 
\end{proof}

\section{Wavelets and function spaces}
\label{sec:wavelets}

In this section, we give a brief overview of the different function spaces used in the paper and how their norms can be related to their wavelet coefficients.

First, we recall the definition of Hölder and Sobolev spaces.
Let  \( \Omega \subseteq \R^d \) be a closed set with non-empty interior and denote by \( C_u(\Omega) \) the set of uniformly continuous functions on \( \Omega \). The \emph{H\"older} norms for any $\potential \in C_u(\Omega)$ are defined as follows. For any integer $k\ge 0$ and \( \potential \) that admits continuous derivatives up to order \( k \), define
\begin{equation}
\| f \|_{C^k(\Omega)} := \sum_{| \flat | \le k} \| \partial^\flat f \|_{L^\infty(\Omega)},
\end{equation}
and for any real number $\alpha>0$, define
\begin{equation}
\| f \|_{C^\alpha(\Omega)} := \| f \|_{C^{\lfloor \alpha \rfloor}(\Omega)} + \sum_{| \flat | = \lfloor \alpha \rfloor} \sup_{x \neq y, x,y \in \Omega} \frac{| \partial^\flat f(x) - \partial^\flat f(y) |}{\| x - y \|_2^{\alpha - \lfloor \alpha \rfloor}}.
\end{equation}
The space \( C^\alpha(\Omega) \) is then defined as the set of functions for which this norm is finite.
For a vector-valued function \( \tmap \colon \Omega \to \R^d \), \( T = (T_1, \dots, T_d)^\top \), we similarly define the norms as the sum over the individual norms,
\begin{equation}
	\label{eq:sk}
	\| T \|_{C^\alpha(\Omega)} := \sum_{i = 1}^d \| T_i \|_{C^\alpha(\Omega)}.
\end{equation}

Similarly, for an integer \( k \ge 0 \) and \( p \in [1, \infty] \), the Sobolev norms are defined as
\begin{equation}
	\label{eq:sj}
	\| f \|_{W^{k,p}(\Omega)} := \sum_{|\flat| \le k} \| \partial^\flat \potential \|_{L^p(\Omega)},
\end{equation}
where the derivative \( \partial^\flat \) is to be understood in the sense of distributions and the Sobolev space \( W^{m,2}(\Omega) \) is the space of all functions for which this norm is finite.
This definition can be extended to \( \alpha > 0 \), for example by defining \( W^{\alpha,2}(\Omega) \) as the Besov space \( B^\alpha_{2,2}(\Omega) \), which we define shortly.

Next, we define wavelet bases and Besov spaces, following the definitions given in \cite[Section 3]{Tri06}, and we refer the reader to this reference for further details on wavelets.
Denote by \( \psi_{\mathfrak{M}} \in C^{r}(\R) \) and \( \psi_{\mathfrak{F}} \in C^{r}(\R) \) a compactly supported wavelet and scaling function, respectively, for example Daubechies wavelets.
This implies that
\begin{equation}
	\label{eq:hr}
	\psi^j_k =
	\left\{
	\begin{aligned}
		&\psi_{\mathfrak{F}}(x - k),& \quad &j = 0, k \in \mathbb{Z},\\
		&2^{(j-1)/2} \psi_{\mathfrak{M}}(2^{j-1} x - k),& \quad &j \in \mathbb{N}, k \in \mathbb{Z},
	\end{aligned}
	\right.
\end{equation}
is an orthonormal basis of \( L^2(\R) \).
To obtain a basis of \( L^2(\R^d) \), for \( j \in \mathbb{N} \), set
\begin{equation}
	\label{eq:hs}
	G^j = \{\mathfrak{F}, \mathfrak{M}\}^{d} \setminus \{(\mathfrak{F}, \dots, \mathfrak{F})\}, \quad G^0 = \{(\mathfrak{F}, \dots, \mathfrak{F})\},
\end{equation}
and for \( g \in G^j \cup G^0 \),
\begin{equation}
	\label{eq:ht}
	\Psi_k^g(x) = \prod_{i=1}^n \psi_{g_i}(x_i - k_i), \quad k \in \mathbb{Z}^d.
\end{equation}
This gives the orthonormal basis
\begin{equation}
	\label{eq:hv}
	\Psi^{j, g}_k =
	\left\{
	\begin{aligned}
		&\Psi^g_k(x), &\quad &j = 0, \, g \in G^0, \, k \in \mathbb{Z}^d\\
		&2^{(j-1)d/2} \Psi^g_k(2^{j-1} x), & \quad  & j \in \mathbb{N}, \, g \in G^j, \, k \in \mathbb{Z}^d.
	\end{aligned}
	\right.
\end{equation}

Wavelet coefficients, defined as the expansion coefficients with respect to the above basis for \( L^2(\R^d) \) functions, can be used to characterize the so-called \emph{Besov spaces}:
Let \( 1 \le p, q \le \infty \), \( s \ge 0 \), and let the regularity of the above wavelets satisfy

\begin{equation}
	\label{eq:hw}
	r > s \vee \left( \frac{2d}{p} + \frac{d}{2} - s \right).
\end{equation}
With this, if \( f \) admits the wavelet representation
\begin{equation}
	\label{eq:hy}
	f = \sum_{j=0}^{\infty} \sum_{g \in G^{j}} \sum_{k \in \mathbb{Z}^d} \gamma_k^{j, g} \Psi_k^{j, g},
\end{equation}
we define the family of norms \( \| \, . \, \|_{B^s_{p, q}} \) by 
\begin{align}
	\label{eq:fe}
	\| f \|_{B_{p,q}^s}
	:= {} & \| f \|_{B_{p, q}^s(\R^d)}
	:= \| \gamma \|_{b_{p,q}^s}\\
	:= {} & \left[ \sum_{j=0}^{\infty} 2^{j q (s + \frac{d}{2} - \frac{d}{p})} \sum_{g \in G^j} \left( \sum_{k \in \mathbb{Z}^d} | \gamma_k^{j, g} |^p \right)^{q/p} \right]^{1/q}.
\end{align}
In particular, by the orthonormality of the wavelets \( \{ \Psi_{k}^{j, g} \}_{j, g, k} \),
\begin{equation}
	\label{eq:la}
	\| f \|_{B_{2,2}^0} = \| \gamma \|_{\ell^2} = \| f \|_{L^2(\R^d)}.
\end{equation}
For a bounded Lipschitz domain \( \Omega \subseteq \R^d \), we can define the Besov spaces \( B^{s}_{p, q}(\Omega) \) by restrictions of functions on \( \R^d \) with norm
\begin{equation}
	\label{eq:sl}
	\| \potential \|_{B^{s}_{p, q}(\Omega)} := \inf \{ \| \dualpotential \|_{B^{s}_{p, q}(\R^d)} : \left. \dualpotential \right|_{\Omega} = \potential \}.
\end{equation}
\noeqref{eq:fe}

Note that since we work with compactly supported wavelets, for a function \( f \) with compact support, only a finite number of wavelet coefficients are non-vanishing in \eqref{eq:hy}.
In particular, for non-zero wavelet coefficients are contained in a set \( \Lambda(j) \) with \( | \Lambda(j) | \lesssim 2^{jd} \).

The following theorem collects some basic properties of Besov spaces and their relationship to H\"older and Sobolev spaces.

\begin{theorem}
	\label{thm:restrict-extend}
	Let \( \Omega \subseteq \R^d \) be a bounded Lipschitz domain.
	\begin{enumerate}
		\item \cite[Proposition 4.3.6]{GinNic16} Let \( s, s' > 0 \), \( 1 \le p, p', q, q' \le \infty \).
			Then, the following inclusions hold in the sense of continuous embeddings:
			\begin{enumerate}[label=(\alph*)]
				\item \( B^{s}_{p, q} \subseteq B^s_{p, q'} \), if \( q \le q' \),
				\item \( B^{s}_{p, q} \subseteq B^{s'}_{p, q'} \), if \( s > s' \),
			\end{enumerate}
		\item \cite[Theorem 1.122]{Tri06}
			Let \( k \in \mathbb{N} \).
			Then, $W^{k,2}(\Omega) = B_{2,2}^k(\Omega)$ and we define
			\( W^{\alpha,2}(\Omega) := B_{2,2}^\alpha(\Omega) \) for \( \alpha > 0 \) not an integer.
			
		\item \cite[Theorem 1.122]{Tri06}, \cite[Proposition 4.3.20]{GinNic16} Let \( \alpha > 0 \).
		If \( \alpha \) is not integer, then
			\begin{equation}
				\label{eq:fw}
				C^\alpha(\Omega) = B_{\infty, \infty}^\alpha(\Omega),
			\end{equation}
			If \( \alpha \) is integer, then
						\begin{equation}
				\label{eq:hu}
				B^\alpha_{1, \infty}(\Omega) \subseteq C^\alpha(\Omega) \subseteq B^\alpha_{\infty, \infty}(\Omega).
			\end{equation}
	\end{enumerate}
\end{theorem}

Note that	the proof in \cite{GinNic16} of \eqref{eq:hu} is given in 1D, but extends naturally to arbitrary dimensions by summability of the wavelet coefficients.

A very useful tool to handle function spaces on domains is the availability of extension operators that preserve the norms.
The following theorem guarantees the existence of such an extension operator that is the same among all Besov spaces.
This allows us to characterize Besov functions on domains via the wavelet coefficients of their extensions.

\begin{theorem}
	[{Extension operator, \cite[Theorem 1.105]{Tri06}, \cite{Ryc99}}]
	\label{lem:extension}
	Let \( \Omega \subseteq \R^d \) be a bounded Lipschitz domain.
	Then, there exists a linear extension operator \( \extension \) that preserves \( L^2 \)-, \( C^{\beta} \)-, and \( W^{\alpha,2} \)-norms.
	That is, there exists an extension operator \( \extension = \extension(\Omega) \) such that for \( A \in \{ L^2, C^{\beta}, W^{\alpha,2} : \beta > 0, \alpha > 0 \}\), there exist constants \( C = C(\Omega, A) \) with 
	\begin{equation}
		\label{eq:nz}
		\| \extension \potential \|_{A(\R^d)} \le C\| \potential \|_{A(\Omega)}  , \quad \text{and} \quad \left.\extension \potential \right|_{\Omega} = \potential, \quad \text{for } \potential \in A.
	\end{equation}
\end{theorem}

We conclude this section by a lemma that provides uniform control of a function by its wavelet coefficients. It is useful to control bracketing entropy numbers.

\begin{lemma}
	\label{lem:e}
	Let \( f \in V_J(\R^d) \) with wavelet coefficients \( \gamma^{j,g}_k \) for compactly supported mother and father wavelets.
	Then, $\| f \|_\infty \lesssim 2^{Jd/2} \| \gamma \|_\infty \le 2^{Jd/2} \| \gamma \|_2.$
\end{lemma}
\begin{proof}
	Let \( x \in \R^d \) and write
	\begin{align}
		| f(x) |
		= {} & \big| \sum_{j = 0}^J \sum_{g \in G^j, k \in \mathbb{Z}^d} \gamma^{j, g}_k \Psi^{j, g}_k (x) \big|
		\le  \| \gamma \|_\infty \sum_{j = 0}^ J \sum_{g \in G^j, k \in \mathbb{Z}^d} | \Psi^{j, g}_k (x) |\\
		\lesssim {} & \| \gamma \|_\infty \sum_{j = 0}^ J 2^{jd/2}
		 \lesssim  2^{Jd/2} \| \gamma \|_\infty,
	\end{align}
	where we used H\"older's inequality and the fact that only a finite number of \( k \) enters the summation for each fixed \( x \in \R^d \). The last inequality is trivial.
\end{proof}

\section{Metric entropy and suprema of stochastic processes}
\label{SEC:emp-process}

Here, we collect some basic results about empirical processes that are needed in the proofs.
Note that because all suprema we deal with are over subsets of finite-dimensional vector spaces, we do not consider issues of measurability in the remainder.

Denote the bracketing number of a set \( \mathcal{F} \) with respect to a norm \( \|.\| \) by $N_{[\,]}(\delta, \mathcal{F}, \|\,.\,\|)$
and define the Dudley integral as
\begin{equation}
	\label{eq:ez}
	\cD_{[\,]}(\sigma, \mathcal{F}, \|\,.\,\|) = \int_0^\sigma \sqrt{1 + \log N_{[\,]}(\delta, \mathcal{F}, \|\,.\,\|)} \, \ud \delta.
\end{equation}

\begin{theorem}
	[Bernstein chaining, { \cite[Lemma 3.4.2]{vanWel07} }]
	\label{thm:bernstein-chaining}
	Let \( \mathcal{F} \) be a class of measurable functions such that \( \E[f^2] < \sigma^2 \) and \( \| f \|_\infty \le \constbd \) for all \( f \in \mathcal{F} \).
	Then,
	\begin{equation}
		\label{eq:ex}
		\E \Big[ \sup_{f \in \mathcal{F}} | \sqrt{n}(\hat P - P) f |\Big]
		\lesssim \cD_{[\,]}(\sigma, \mathcal{F}, L_2(P)) \Big( 1 + \frac{\cD_{[\,]}(\sigma, \mathcal{F}, L_2(P))}{\sigma^2 \sqrt{n}} \constbd \Big).
	\end{equation}
\end{theorem}

\begin{theorem}
	[{Concentration, see \cite[Equation (5.50)]{Mas07}, going back to \cite{Bou02}}]
	Under the same assumptions as Theorem \ref{thm:bernstein-chaining},
	\label{thm:concentration}
	\begin{equation}
		\label{eq:im}
		\p \Big(\sup_{f \in \mathcal{F}} \sqrt{n} | (\hat P - P) f| \ge 2 \E[\sup_{f \in \mathcal{F}} \sqrt{n} | (\hat P - P) f|] + \sigma \sqrt{2x} + \frac{\constbd}{\sqrt{n}} x \Big) \le \exp(-x).
	\end{equation}
\end{theorem}

\begin{lemma}
	[Bracketing numbers from \( L^\infty \) covering numbers]
	\label{lem:h}
	Let \( \measureone \) be a probability measure and \( A \subseteq \mathcal{F} \) be a set of functions with $N(A, L^\infty(\measureone), \delta/2) \le \phi(\delta)$.
	Then, $N_{[]}(A, L^2(\measureone), \delta) \le \phi(\delta)$.

	Moreover, if for every function in the original class \( f \in \mathcal{F} \), $\E[f^2] < \sigma^2$ and $\| f \|_{L^\infty(\measureone)} < \constbd$, 
	then every bracket \( [f_1, f_2] \) in a \( \delta \)-cover above (note that \( f_1 \) and \( f_2 \) need not be members of \( \mathcal{F} \)) satisfies
	\begin{equation}
		\label{eq:he}
		\E[f_{j}^2] < 2 \sigma^2 + \frac{1}{2} \delta^2, \quad \| f_j \|_{L^\infty(\nu)} < \constbd + \delta/2, \quad j \in \{1, 2\}.
	\end{equation}
\end{lemma}

\begin{proof}
	Denote by \( \{f_1, \dots, f_N\} \) the centers of a minimal \( \delta/2 \)-covering of \( \mathcal{F} \) in \( L^\infty(\measureone) \).
	Let \( i \in \{1, \dots, N\} \).
	Then, each \( \delta/2 \) ball around \( f_i \) is contained in the bracket \( [f_i - \delta/2, f_i + \delta/2] \).
	Moreover, the \( L^2(\measureone) \) diameter of the above bracket is bounded by its \( L^\infty(\measureone) \) diameter, which is \( \delta \), so the collection of those brackets yields the desired covering with brackets.

	The rest of the lemma follows from
	\begin{equation}
		\label{eq:hf}
		\E[(f \pm \delta/2)^2] \le 2 \E[f^2] + \frac{1}{2} \delta^2,\qquad \| f \pm v\delta/2 \|_{L^\infty(\nu)} \le \| f \|_{L^\infty(\nu)} + \frac{\delta}{2}.
	\end{equation}
\end{proof}

The following lemma can be shown by directly specifying a grid or by a volume argument such as \cite[Proposition 4.2.12]{Ver18}.

\begin{lemma}
	[Covering numbers for norm balls]
	\label{lem:f}
	Fix \( p \in \mathbb{N} \) and denote by $B_\infty(A)$ the $\ell^\infty$ ball of $\R^d$ with radius $A$.
	Then, $N(B_\infty(A), \|\, . \,\|_\infty, \delta) \le \left( {3A}/{\delta} \right)^d$.
\end{lemma}

\begin{lemma}
	[{\cite[Lemma 10.3]{Rau10b}}]
	\label{lem:i}
	For \( \alpha > 0 \),
	\begin{equation}
		\label{eq:hl}
		\int_0^\alpha \sqrt{\log(1 + t^{-1})} \, \ud t \le \alpha \sqrt{1+\log(1+\alpha^{-1})}\,.	\end{equation}
\end{lemma}

\section{Tools for lower bounds}
\label{SEC:lower-bounds-tools}

For the convenience of the reader, in this section we restate the standard tools we use in the proof of Theorem \ref{thm:lb} to establish lower bounds based on Fano's inequality and the Varshamov-Gilbert Lemma, taken from \cite[Theorem 2.9, Lemma 2.9, Theorem 2.2]{Tsy09}.

\begin{theorem}
	[Lower bounds from multiple hypotheses]
	\label{thm:lower-bounds-master}
	\cite[Theorem 2.9]{Tsy09}
	Let \( K \ge 2 \), \( \Theta = \{T_0, \dots, T_K\} \) a collection of hypotheses, and let \( d \) be a pseudometric, i.e., a bi-variate function on \( \Theta \) such that
	\begin{enumerate}[label=(\alph*)]
		\item \( d(T_j, T_k) \ge 0 \);
		\item \( d(T_j, T_j) = 0 \);
		\item \( d(T_j, T_k) = d(T_k, T_j)  \);
		\item \( d(T_j, T_\ell) \le d(T_j, T_k) + d(T_k, T_\ell) \);
	\end{enumerate}
	for all \( 0 \le j, k, \ell \in \{0, \dots, K\} \).
	
	Suppose \( \Theta \) fulfills:
	\begin{enumerate}
			\item \( d(T_j, T_k) \ge 2s > 0 \), for all \( 0 \le j < k \le K \);
			\item \( P_j \ll P_0 \), for all \( j \in [K] \), and
				\begin{equation}
					\label{eq:th}
					\frac{1}{K} \sum_{j = 1}^{K} D(P_j \| P_0) \le C \log K
				\end{equation}
			with \( 0 < C < 1/8 \), where \( P_j = P_{T_j} \) denotes a probability distribution associated with every \( T_j \) for all \( j = 0,1, \dots, K\).
		\end{enumerate}
			Then,
			\begin{equation}
				\label{eq:ti}
				\inf_{\hat T} \sup_{T \in \Theta} P_{T} (d(\hat T, T) \ge s) \ge \frac{\sqrt{K}}{1 + \sqrt{K}} \left(1 - 2 C - \sqrt{\frac{2 C}{\log K}} \right) > 0.
			\end{equation}
\end{theorem}

\begin{lemma}
	[Varshamov-Gilbert lemma]
	\label{lem:varshamov-gilbert}
	\cite[Lemma 2.9]{Tsy09}
	Let \( D \ge 8 \).
	There exists a subset \( \tau^{(0)}, \dots, \tau^{(K)} \) of \( \{0, 1\}^{D} \) such that \( \tau^{(0)} = (0, \dots, 0) \),
	\begin{equation}
		\label{eq:tj}
		\sum_{\ell = 1}^{D} \1 (\tau^{(j)}_l, \tau^{(k)}_l) \ge \frac{D}{8}, \quad \text{for all } 0 \le j < k \le K,
	\end{equation}
	and
	\begin{equation}
		\label{eq:tk}
		K \ge 2^{D/8}.
	\end{equation}
\end{lemma}

\begin{theorem}
	[Lower bounds from two hypotheses]
	\label{thm:lower-bounds-master-twopoint}
	\cite[Theorem 2.2]{Tsy09}
	Let \( T_0, T_1 \) be two hypotheses with associated probability measures \( P_j = P_{T_j} \) for \( j \in \{0, 1\} \).
	Denoting \( s = d(T_0, T_1)/2 \), if
	\begin{equation}
		\label{eq:tm}
		D(P_1 \| P_0) \le C < \infty,
	\end{equation}
	then
	\begin{equation}
		\label{eq:tn}
		\inf_{\hat T} \max_{j \in \{0, 1\}} P_j(d(\hat T, T_j) \ge s) \ge \frac{1}{4} \exp(-C) \vee \frac{1 - \sqrt{C/2}}{2} > 0.
	\end{equation}
\end{theorem}

\section{Alternative assumptions via Caffarelli's regularity theory}
\label{sec:caffarelli}

In this section, we show how to apply Theorem \ref{thm:upper-bounds} under smoothness assumptions on the source and target distribution instead of the transport map.
This is enabled by Caffarelli's regularity theory \cite{Caf92,Caf92a,Caf96}, Theorem \ref{thm:caffarelli} below, which gives regularity estimates on the transport map \( \tmap_0 \) under regularity assumptions on the source and target densities and their supports.

For simplicity, we denote an open, convex set \( \Omega \) with \( C^2 \) boundary as \emph{uniformly convex} if it can be written as the sublevel set, \( \Omega = \{f < 0\} \), of a strongly convex function \( f \).
Note that uniform convexity can also be characterized by the positivity of the second fundamental form of \( \Omega \) on its boundary \( \partial \Omega \) \cite{Vil09}.

\begin{assumpalt}[Supports]
	\label{assumpalt:geometry}
	Assume \( \Omega_P, \Omega_Q \) are two uniformly convex, bounded, open subsets of \( \R^d \) with \( C^{\lfloor \alpha - 1 \rfloor + 2} \) boundaries.
\end{assumpalt}

\begin{assumpalt}[Densities]
	\label{assumpalt:regularity}
	With \( \Omega_P, \Omega_Q \) as in Assumption \ref{assumpalt:geometry}, assume that \( \rho_P \in C^{\alpha - 1}(\overline{\Omega}_P) \) and \( \rho_Q \in C^{\alpha - 1}(\overline{\Omega}_Q) \) are two probability densities on \( \R^d \) with \( \supp(\rho_P) = \overline{\Omega}_P \), \( \supp(\rho_Q) = \overline{\Omega}_Q \) that are bounded from above and below on their support.
	Further, denote the associated probability distributions by \( P \) and \( Q \), respectively.
\end{assumpalt}

\begin{theorem}
	[Caffarelli's global regularity theory]
	\ \\
	\label{thm:caffarelli}
	\cite[Theorem 12.50(iii)]{Vil09}, \cite{Caf96}.
	Let \( \alpha > 1 \) and assume that Assumptions \ref{assumpalt:geometry} and \ref{assumpalt:regularity} hold.
	Then, for the optimal transport potential \( \potential_0 \), i.e., the solution to \eqref{eq:pn} for \( P \) and \( Q \), which is unique up to an additive constant, it holds that \( \potential_0 \in C^{\alpha + 1}(\overline{\Omega_P}) \).
\end{theorem}

Moreover, we need the following extension lemma to smoothly extend the potential (and associated transport map) to a larger set as required by Assumption \ref{assumpgen:map}.
Similar statements have been shown in \cite{Gho02,Yan14,AzaMud19}.

\begin{lemma}
	\label{lem:convex-extension}
	Let \( \Omega \) be a convex compact subset of \( \R^d \), \( \potential : \Omega \to \R \) a convex \( C^\alpha(\Omega) \) function for \( \alpha > 2 \) with
	\begin{equation}
		\label{eq:to}
		\min_{x \in \Omega} \lambda_{\min}(D^2 \potential(x)) > 0.
	\end{equation}
	Then, there exists an \( \epsilon > 0 \) and an extension \( \tilde \potential \) of \( \potential \) to \( \Omega_\epsilon = \Omega + \epsilon B(0,1) \) such that \( \potential \in C^\alpha(\Omega_\epsilon) \) and
	\begin{equation}
		\label{eq:uj}
		\min_{x \in \Omega_\epsilon} \lambda_{\min}(D^2 \tilde \potential(x)) > 0.
	\end{equation}
\end{lemma}

\begin{proof}
The claim follows from extending \( \potential \) via Theorem \ref{lem:extension} to all of \( \R^d \) while preserving the H\"older class \( C^\alpha \), and observing that the strong convexity condition \eqref{eq:uj} on an enlargement of \( \Omega \) can be ensured by uniform continuity of the extension.
\end{proof}

\begin{corollary}
	\label{cor:upper-bounds-caffarelli}
	Let \( \alpha > 1 \). Under Assumptions \ref{assumpalt:geometry} and \ref{assumpalt:regularity}, writing \( \potential_0 \) and \( \tmap_0 = \nabla \potential_0 \) for the Kantorovich potential and optimal transport map between \( P \) and \( Q \), respectively, there exists an estimator \( \hat \tmap \) such that
	\begin{equation}
		\label{eq:tq}
		\E_{(X_{1:n}, Y_{1:n})} \left[ \int\| \hat \tmap(x) - \tmap_0(x) \|_2^2 \, \ud \measureone(x) \right]
		\lesssim \left[n^{-\frac{2 \alpha}{2 \alpha - 2 + d}} (\log(n))^2 \vee \frac1n \right],
	\end{equation}
	where \( X_{1:n}, Y_{1:n} \) denote \iid observations from \( P \) and \( Q \), respectively.
\end{corollary}

\begin{proof}
	The statement follows by verifying the assumptions \ref{assumpgen:source} and \ref{assumpgen:map} on the original transport map and concluding by Theorem \ref{thm:upper-bounds}, using the same estimator.

	Assumption \ref{assumpgen:source} is satisfied for \( \overline{\Omega}_P \) by the (stronger) conditions imposed on it in Theorem \ref{thm:caffarelli}.

	To show Assumption \ref{assumpgen:map}, denote by \( \potential_0 \) the optimal transport potential and by \( \tmap_0 \) the optimal transport map for the problem given by \( \rho_P \) and \( \rho_Q \).
By Theorem \ref{thm:caffarelli}, \( \potential_0 \in C^{\alpha + 1}(\overline{\Omega}_P) \) and thus \( \tmap_0 = \nabla \potential_0 \in C^\alpha(\overline{\Omega}_P, \R^d) \).
Denoting by \( \tilde \potential_0 \) the extension of \( \potential_0 \) to \( \tilde \Omega_P = (\overline{\Omega}_P)_\epsilon \) from Lemma \ref{lem:convex-extension}, we observe that \( \tilde \Omega_P \) is connected and has a Lipschitz boundary. This follows from the general fact that the Minkowski sum of a convex set with \( C^1 \) boundary and another convex set has a \( C^1 \) boundary \cite{KraPar91}.
	From the regularity of \( \tilde \tmap_0 = \tilde \potential_0 \) and the compactness of \( \tilde \Omega_P \), all requirements in \ref{assumpgen:map} can now be easily checked.
\end{proof}

An inspection of the proof of our lower bound results, Theorem \ref{thm:lb}, readily yields a lower bound under smoothness assumptions on \( P \) and \( Q \) as well.
In particular, after changing the domain from \( [0,1]^d \) to the unit ball in \( d \) dimensions, the explicit form of the density of \( Q_k \) in \eqref{eq:monge-ampere} guarantees that the candidate densities fulfill assumptions \ref{assumpalt:geometry} and \ref{assumpalt:regularity}.
This indicates that the above rate is optimal up to log factors.

\begin{remark}
	\label{rem:caf-no-minimax}
	In Corollary \ref{cor:upper-bounds-caffarelli}, we do not claim any uniform bound over a regularity class, such as an explicit dependence of the constant on the \( C^{\alpha - 1} \)-norms of \( \rho_P \) and \( \rho_Q \) alone.
	Although the statement of Theorem \ref{thm:caffarelli} strongly suggests a bound on \( \| \potential_0 \|_{C^{\alpha + 1}(\overline{\Omega}_P)} \) in terms of \( \| \rho_P \|_{C^{\alpha - 1}(\overline{\Omega}_P)} \) and \( \| \rho_Q \|_{C^{\alpha - 1}(\overline{\Omega}_Q)} \), to the best of our knowledge, such bounds are not available in the literature for the ``global'' version of Caffarelli's regularity theory as stated Theorem \ref{thm:caffarelli}.
	On the other hand, such bounds hold ``locally'', that is, \( \| \potential_0 \|_{C^{\alpha + 1}(U)} \) can be bounded for strictly open subsets of \( \overline{\Omega}_P \), see \cite[Theorem 12.56]{Vil09}, \cite{Caf92}.
	This deficiency is due to the non-constructive nature of the available proofs of Theorem \ref{thm:caffarelli} and could possibly be remedied, but we consider it out of the scope of this paper.
\end{remark}

\section{Omitted proofs}
\label{sec:proofs}

\subsection{Proof of Proposition-Definition~\ref{lem:immediate-bounds}}
\label{Proof:lem:immediate-bounds}

We proceed in order of statement of the results. 

\ref{itm:immediate-bounds-d01} It follows from  \ref{assumpgen:map} that  there exists $f_0$ such that $\nabla f_0 =T_0$ and by  \ref{assumpgen:map}\ref{itm:assump-boundedness} that $|\nabla f_0(x)| \le M$ for all $x \in \tilde \Omega_P$. Moreover, since $f_0$ is defined up an additive constant, assume that $f_0(x_0)=0$ for some $x_0 \in \tilde \Omega_P$. A first-order Taylor expansion yields that for any $x \in \tilde \Omega_P$, 
$$
|f_0(x)|=|f_0(x)-f_0(x_0)|\le \sup_{z \in \tilde \Omega_P}\|\nabla f_0 (z)\|_2\|x-x_0\|_2 \le M \diam(\tilde \Omega_P)\le 2 M^2,
$$
where in the second inequality, we used \ref{assumpgen:map}\ref{itm:assump-boundedness}, and in the third one, we used that $ \diam(\tilde \Omega_P)\le 2 M$ according to \ref{assumpgen:map}.

\ref{itm:immediate-bounds-d2} follows immediately from \ref{assumpgen:map}\ref{itm:assump-d2}.

	Next, as in \ref{itm:immediate-bounds-d01}, \eqref{EQ:smoothpotentials} follows from a first-order expansion: For any $x \in \tilde 
\Omega_P$, note that on the one hand
$$
| f_0(x)|=|f_0(x)-f_0(x_0)| \le 2 \|T_0\|_2 M\le 2 M^2
$$
so that $\|f_0\|_{L^\infty} \le M^2$ and \( \| f_0 \|_{L^2} \le 2 M^3 \).
Together with the fact that $T_0=\nabla f_0$, this allows us to shift the smoothness index by one speaking about potentials.

To show the statements about \( \measuretwo \), by  Lemma \ref{lem:diffconvex} and \ref{itm:immediate-bounds-d2} above, \( \potential_0 \) is strongly convex  on \( \interior(\tilde \Omega_\measureone) \) and can be extended to \( + \infty \) outside of \( \tilde \Omega_\measureone \) and thus be also be considered a strongly convex function on \( \R^d \).
	By Lemmas  \ref{lem:conj-lipschitz} and \ref{lem:sub-der-correspondence}, we  conclude that \( \tmap_0 = \nabla \potential_0 \) is a bijection from \( \Omega_\measureone \) onto its image \( \Omega_\measuretwo = \supp ((\nabla \potential_0)_\# \measureone) \), with both \( \nabla \potential_0 \) and \( (\nabla \potential_0)^{-1} \) being continuously differentiable; in other words, \( \nabla \potential_0 \) is a \( C^1 \)-diffeomorphism between \( \Omega_\measureone \) and \( \Omega_\measuretwo \).
	Since \( C^1 \)-diffeomorphisms preserve Lipschitz domains~\cite[Theorem 4.1]{HofMitTay07} and connectedness, we can conclude that $\Omega_Q$ is a connected and bounded Lipschitz domain, and the fact that \( \Omega_\measuretwo \subseteq M B_1 \) follows from \( T(x) \le \constbd \) for all \( x \in \tilde \Omega_\measureone \).
	
	Finally, we turn to check the condition on the density $\rho_Q$. To that end, note that  by the change of variables formula,  \( \measuretwo \) has the density
		\begin{equation}
		\label{eq:mr}
		\density_\measuretwo(y) = \frac{1}{|\det D^2 \potential_0((\nabla \potential_0)^{-1}(y))|} \density_\measureone((\nabla \potential_0)^{-1}(y)) \1(y \in \Omega_\measuretwo).
	\end{equation}
	This readily yields the desired bound in light of \ref{assumpgen:source} which assumes boundedness of $\rho_P$ and \ref{itm:immediate-bounds-d2} which gives boundedness of the Hessian $D^2 \potential_0$.

\subsection{Proof of Proposition \ref{lem:expectation-bound}}
\label{sec:proof-expectation-bound}

We first bound the expectation of the supremum of interest and then obtain the high-probability bound via concentration.
To begin, note that
	\begin{equation}
		\label{eq:lp}
		\emprisk_0(\potential)- \risk_0(\potential) 
		=  (\hat \measureone - \measureone) (\potential- \potential_0) + (\hat \measuretwo - \measuretwo) (\potential^\ast - \potential_0^\ast)\,,
	\end{equation}
	which yields
	\begin{align}
		\label{eq:fs}
		\leadeq{\E [\sup_{\potential\in \funcclass_J(\tau^2)} | \emprisk_0(\potential)-\risk_0(\potential) | ]}\\
	\le {} &
		 \E [\sup_{\potential\in \funcclass_J(\tau^2)} |(\hat \measureone - \measureone) (\potential- \potential_0)|]
		+ \E [\sup_{\potential\in \funcclass_J(\tau^2)} |(\hat \measuretwo - \measuretwo) (\potential^\ast - \potential_0^\ast)|]\\
		=: {} & T_1 + T_2.
	\end{align}
	We first focus on the \( T_2 \)-term.
	Once understood, the \( T_1 \)-term can be bounded similarly.

\medskip

\noindent\textbf{Bound on \( T_2 \)-term.}

		We estimate \( T_2 \) from above by two suprema, corresponding to low and high frequencies.
		To this end, we center all functions and consider the extension of a restriction of the functions to \( \Omega_\measuretwo \), which allows us to use wavelet expansions for harmonic analysis.

	First, because \( (\hat \measuretwo - \measuretwo)(g + c) = (\hat \measuretwo - \measuretwo)g \) for every function \( g \) and constant \( c \in \R \), we may assume without loss of generality that 
	$$
	\int_{\Omega_\measuretwo} (f^*(z)-f^*_0(z)) \, \ud \lambda(z)=0 \quad \forall f \in \cF_J.
	$$

Note that $f^*$ and $f_0^*$ are defined over all of $\R^d$ but we do not control their norms over the whole space. To overcome this limitation, with a slight abuse of notation, we denote by $\ext f^*$ (resp. $\extension f^*_0$) the Lipschitz extension of the restriction of $f^*$  (resp. $f^*_0$) to \( \Omega_\measuretwo \). The existence of a linear extension operator is guaranteed by Theorem~\ref{lem:extension}. In particular, we control the norms of $\ext f^*$ and $\ext f^*_0$.

For a function \( g \) with  wavelet expansion
\begin{equation}
		\label{eq:ft}
		g = \sum_{j=0}^{\infty} \sum_{g \in G^{j}} \sum_{k \in \mathbb{Z}^d} \gamma_k^{j, g} \Psi_k^{j, g},
	\end{equation}
	as defined in Section \ref{sec:wavelets}, we define the two $L^2$-projections
	\begin{equation}
		\label{eq:fu}
		\Pi_J g =  \sum_{j=0}^{J} \sum_{g \in G^{j}} \sum_{k \in \mathbb{Z}^d} \gamma_k^{j, g} \Psi_k^{j, g}\,, \qquad 
		\Pi_{>J} g = \sum_{j=J+1}^{\infty} \sum_{g \in G^{j}} \sum_{k \in \mathbb{Z}^d} \gamma_k^{j, g} \Psi_k^{j, g}.
	\end{equation}
With this, we write $T_2=T_{2,1} + T_{2,2}$, where
\begin{align}
	\label{eq:bi}
	T_{2,1}={} &\E[\sup_{\potential \in \funcclass_J(\tau^2)} |(\hat \measuretwo - \measuretwo) \Pi_{J} \extension (\potential^\ast - \potential_0^\ast)|]\,,\\
	T_{2,2}={} &\E[\sup_{\potential \in \funcclass_J(\tau^2)} |(\hat \measuretwo - \measuretwo) \Pi_{>J} \extension (\potential^\ast - \potential_0^\ast)|]\,.
\end{align}

	Note that the projected functions are again well-defined continuous functions: By assumption, all \( \potential\) and \( f_0 \) are strongly convex, hence their conjugates have Lipschitz-continuous gradients and therefore bounded \( B^2_{\infty, \infty}(\Omega_\measuretwo) \)-norm.
	In turn, their projections also have bounded \( B^2_{\infty, \infty}(\Omega_\measuretwo) \)-norm, and by Theorem~\ref{lem:extension}, this implies that both \( \Pi_J \extension (f^\ast - f_0^\ast) \) and \( \Pi_{>J} \extension (f^\ast - f_0^\ast) \) are in \( C^s(\R^d) \), for \( s < 2 \), and hence they are continuous.

	\medskip
\noindent\textbf{Bound on \( T_{2,1} \)-term.}
	
Recall that it follows from Proposition-Definition~\ref{lem:immediate-bounds} that \( \Omega_\measuretwo \) is a connected Lipschitz domain, so we can apply  the Poincar\'e-Wirtinger inequality, Lemma \ref{lem:poincare}, together with \eqref{eq:np} from Proposition~\ref{prop:gigli}, to get
		\begin{equation}
			\label{eq:bk}
			\int_{\Omega_\measuretwo} (f^\ast - f_0^\ast)^2 \, \ud \lambda  \lesssim \int_{\Omega_\measuretwo} \| \nabla f^\ast - \nabla f_0^\ast \|^2 \, \ud \lambda \lesssim M \tau^2,
		\end{equation}
		where we used that we assumed \( f^\ast - f_0^\ast \) to be centered.

		Hence, \( \potential\in \funcclass_J(\tau^2) \) implies \( \|f^\ast - f_0^\ast \|_{W^{1,2}(\Omega_\measuretwo)} \lesssim \tau \), and therefore due to the properties of the extension operator \( \extension \),
		\begin{equation}
			\label{eq:bm}
			\| \extension  (f^\ast - f_0^\ast) \|_{W^{1,2}(\R^d)} \le \constsobolev \, \tau,
		\end{equation}
		for some constant  \( \constsobolev = \constsobolev(\Omega_\measuretwo, \constbd) \).
Since $\Pi_J$ is a non-expansive operator on Besov spaces, it follows from the above display that 
		\begin{equation}
			\label{eq:bl}
			T_{2,1}
			\le \sup \{|(\hat \measuretwo - \measuretwo) h | : h \in V_J(\R^d), \; \| h \|_{W^{1,2}(\R^d)} \le \constsobolev \tau\}.
		\end{equation}
		
		Bounding the empirical process over this standard function class can now be performed as follows.
Observe first that for any function $h \in V_J(\R^d)$ with wavelet decomposition
		\begin{equation}
			\label{eq:bj}
			h = \sum_{j=0}^{J} \sum_{g \in G^j} \sum_{k \in \mathbb{Z}^d} \gamma^{j, g}_k \Psi^{j, g}_k,
		\end{equation}
		the condition \( \| h\|_{W^{1,2}(\R^d)} \le \constsobolev \, \tau \) is equivalent to
		\begin{equation}
			\label{eq:fz}
			\sum_{j = 0}^J \sum_{g \in G^j} \sum_{k \in \mathbb{Z}^d} 2^{2j} | \gamma^{j, g}_k |^2 \le \constsobolev^2 \tau^2.
		\end{equation}

Next, by symmetrization, for independent copies of Rademacher random variables \( \varepsilon_i \), \( i = 1, \dots, n \),
$$
\E  \sup |(\hat \measuretwo - \measuretwo)h | \le  \E  \sup \frac{1}{n} \sum_{i=1}^{n} \varepsilon_i h(Y_i),
$$
where both suprema are taken over the set $\mathfrak{G}_J=\{h \in V_J(\R^d), \; \| h \|_{W^{1,2}(\R^d)} \le \constsobolev \tau \}$. To control the Rademacher process, fix  \( h \in \mathfrak{G}_J \)  with wavelet decomposition~\eqref{eq:bj}.
		By the Cauchy-Schwarz inequality,
		\begin{align}
			\label{eq:bs}
	 \sum_{i=1}^n \varepsilon_i h(Y_i)
	 \le {} &  \Big( \sum 2^{2j} |\gamma^{j, g}_k|^2 \Big)^{1/2} \Big( \sum \Big(\sum_{i = 1}^{n} \frac{\varepsilon_i}{2^{j}} \Psi^{j, g}_k(Y_i) \Big)^2 \Big)^{1/2} \\
			\le {} & \constsobolev \, \tau \Big( \sum \Big( \sum_{i = 1}^{n} \frac{\varepsilon_i}{2^{j}} \Psi^{j, g}_k(Y_i) \Big)^2 \Big)^{1/2},
		\end{align}
		where here and below, all sums without indices are over $\{0 \le j \le J, \, g \in G^j, \, k \in \mathbb{Z}^d\}$. Since the right-hand side in the display above does not depend on $h$, we get by Jensen's inequality that
		$$
		\E  \sup \frac{1}{n} \sum_{i=1}^{n} \varepsilon_i h(Y_i)\le \frac{\constsobolev \, \tau}{n} \Big( \E \sum \Big( \sum_{i = 1}^{n} \frac{\varepsilon_i}{2^{j}} \Psi^{j, g}_k(Y_i) \Big)^2 \Big)^{1/2}\,.
		$$
		Since \( Y_i \in \Omega_\measuretwo \), a compact set by assumption, for given \( j \) and \( g \), \( \Psi^{j,g}_k(Y_i) \) is non-zero only for \( k \in \Lambda(j) \) where \( \Lambda(j) \) depends on the diameter of \( \Omega_\measuretwo \), and \( |\Lambda(j)| \lesssim 2^{jd} \). Together with the   independence of the \( \varepsilon_i \), this yields
		$$
		\E \sum \Big( \sum_{i = 1}^{n} \frac{\varepsilon_i}{2^{j}} \Psi^{j, g}_k(Y_i) \Big)^2 = \sum_{\substack{0 \le j \le J,\\ g \in G^j, k \in \Lambda(j)}} \sum_{i = 1}^{n} \frac{1}{2^{2j}} \E\left[ \Psi^{j, g}_k(Y_i)^2 \right].
		$$
		By Proposition-Definition~\ref{lem:immediate-bounds} and the fact that the \( \Psi^{j, g}_k \) form an orthonormal basis in \( L^2(\R^d) \),
		\begin{equation}
			\label{eq:bo}
			\E\left[\Psi^{j, g}_k(Y_i)^2 \right]
			\lesssim \int_{\Omega_\measuretwo} \Psi^{j, g}_k(y)^2 \, \ud \lambda(y)
			\le \int_{\R^d} \Psi^{j, g}_{k}(y)^2 \, \ud \lambda(y) = 1.
		\end{equation}
		Thus,
		\begin{align}
			\label{eq:ga}
			\sum_{\substack{0 \le j \le J,\\ g \in G^j, k \in \Lambda(j)}} \sum_{i = 1}^{n} \frac{1}{2^{2j}} \E\left[ \Psi^{j, g}_k(Y_i)^2 \right]
			\lesssim {} & n \sum_{0 \le j \le J} \frac{2^{jd}}{2^{2j}}
			\lesssim \left\{
			\begin{aligned}
				&n, &\quad &d = 1,\\
				&nJ, &\quad &d = 2,\\
				&n 2^{J(d-2)}, & \quad & d \ge 3.
			\end{aligned}
			\right.
		\end{align}
We have proved that
\begin{equation}
\label{EQ:T21}
T_{2,1}\lesssim 	\frac{\tau}{\sqrt{n}}\mathfrak{r}_J, \qquad 
\mathfrak{r}_J:=\left\{	\begin{aligned}
				& 1, & \quad & d = 1,\\
				& \sqrt{J}, & \quad & d = 2,\\
				& \ 2^{J\frac{d-2}{2}}, & \quad & d \ge 3.
			\end{aligned}\right.
\end{equation}

\noindent \textbf{Bound on \( T_{2,2} \)-term.}

To control the term $T_{2,2}$, we use a chaining bound for bracketing entropy (Theorem \ref{thm:bernstein-chaining}).
	To that end, we exhibit bounds on the \( L^{\infty} \)-covering numbers of the corresponding function space, which in turn implies control of \( L^2 \)-bracketing numbers.
	The idea is to bound the covering numbers in the original space \( \funcclass_J(\tau^2) \) and exploit continuity properties of the transformations that lead to the function class in the definition of \( T_{2,2} \), in particular the operation of taking the convex conjugate.

Define the function space
	\begin{equation}
		\label{eq:gw}
		\tilde{\funcclass}_J(\tau^2) = \{ \Pi_{>J} \extension  (\potential^\ast - \potential_0^\ast) : \potential \in \funcclass_J(\tau^2) \}\,,
	\end{equation}
	and let $f_1, f_2 \in \cF_J(\tau^2)$ have wavelet coefficients given by the sequences $\gamma_1$ and $\gamma_2$.
	Observe that by linearity of the projection and extension operators, we have for any cutoff $J' \ge 0$, 
\begin{align*}
\leadeq{\|\Pi_{>J} \extension (f_1^*-f_0^*)-\Pi_{>J} \extension (f_2^*-f_0^*)\|_{L^{\infty}(\Omega_Q)}}\\
={} &\|\Pi_{>J} \extension (f_1^*-f_2^*)\|_{L^{\infty}(\Omega_Q)}\\
	\le {} & \|\Pi_{>J'}\Pi_{>J} \extension (f_1^*-f_2^*)\|_{L^{\infty}(\Omega_Q)}+\|\Pi_{J'}\Pi_{>J} \extension (f_1^*-f_2^*)\|_{L^{\infty}(\Omega_Q)}\,.
\end{align*}

To control the first term in the right-hand side above, in the following lemma, we establish bounds on the potentials in \( \tilde F_J(\tau^2) \).
Its proof is deferred to the next section.

\begin{lemma}
\label{lem:boundedextensions}
There exists a constant \( \constbdder = \constbdder(\Omega_\measureone, \tilde \Omega_\measureone, \Omega_\measuretwo, \constbd) \) such that for all \( \potential \in \funcclassstd(2 \constbd) \),
		\begin{equation}
			\label{eq:jn}
			\| \extension \potential^\ast \|_{B^2_{\infty, \infty}(\R^d)} \le \frac{\constbdder}{2}.
		\end{equation}

		Moreover, under assumptions \ref{assumpgen:source} -- \ref{assumpgen:map}, \( \| \extension \potential_0^\ast \|_{B^2_{\infty, \infty}(\R^d)} \le \constbdder  /2 \) as well.
\end{lemma}

It follows from Lemma~\ref{lem:boundedextensions} that there exists $\constbdder$ such that for any $f \in \cF(\tau^2)$, we have 
\begin{equation}
	\label{eq:sz}
	\|\Pi_{>J} \extension (f_1^*-f_2^*)\|_{B^2_{\infty, \infty}} \le \constbdder \,.
\end{equation}
Therefore, by  Lemma~\ref{lem:wavelet-approximation}, we get
$$
\|\Pi_{>J'}\Pi_{>J} \extension (f_1^*-f_2^*)\|_{L^{\infty}(\Omega_Q)} \le 2^{-2J'}\|\Pi_{>J'} \extension (f_1^*-f_2^*)\|_{B^2_{\infty, \infty}} \le \eps
$$
if we choose $ J' = \left\lceil  \log \left( \constbdder/\varepsilon \right)/2 \right\rceil$
so that $\constbdder \, 2^{-2 J'} \le \varepsilon$. 

To control the second term, we get from Lemma~\ref{lem:e} that
\begin{align}
			\| \Pi_{J'} \Pi_{>J} \extension (\potential_1^\ast - \potential_2^\ast) \|_{L^\infty(\Omega_\measuretwo)}
			\lesssim {} & 2^{J' d/2} \| \Pi_{J'} \gamma \|_{2} \le 2^{J' d/2} \| \gamma \|_{2}\\
			= {} & 2^{J' d/2} \| \extension  (\potential_1^\ast - \potential_2^\ast) \|_{L^2(\R^d)}.
	\end{align}
Moreover, using respectively the fact that  $\extension$ is a Lipschitz operator on Besov spaces, \( \Omega_\measuretwo \) is bounded, and the convex conjugate is non-expansive in \( L^\infty \) by Lemma \ref{lem:g},  we have
\begin{align}
	\| \extension  (\potential_1^\ast - \potential_2^\ast) \|_{L^2(\R^d)}
	\lesssim {} & \| \potential_1^\ast - \potential_2^\ast \|_{L^2(\Omega_\measuretwo)}
	\lesssim  \| \potential_1^\ast - \potential_2^\ast \|_{L^\infty(\Omega_\measuretwo)}\\
	\le {} & \| \potential_1 - \potential_2 \|_{L^\infty(\tilde \Omega_\measureone)} 
	\lesssim 2^{Jd/2}\|\gamma_1-\gamma_2\|_\infty,
\end{align}
where in the last inequality, we used Lemma~\ref{lem:e}.
We have proved that
$$
\|\Pi_{>J} \extension (f_1^*-f_0^*)-\Pi_{>J} \extension (f_2^*-f_0^*)\|_{L^{\infty}(\Omega_Q)} \le \constlinfty \, 2^{(J+J')d/2}\|\gamma_1-\gamma_2\|_\infty+ \eps\,,
$$
for some constant $\constlinfty = \constlinfty(\tilde \Omega_{\measureone}, \Omega_{\measuretwo})$.
This inequality allows us to control the $L^\infty$-bracketing numbers of $\tilde \cF_J(\tau^2)$ using $\ell^\infty$-covering numbers for the wavelet coefficients. To  control the latter, note that for all \( \potential\in \funcclass_J(\tau^2) \subset \funcclassstd(2 \constbd) \), it holds \( \| \potential\|_{B^2_{\infty, \infty}(\R^d)} \lesssim \constbd^2 \) so that $\| \gamma \|_{\ell^\infty} \lesssim \constbd^2$. Moreover, these wavelet coefficients are in a space of dimension at most \(C \, 2^{Jd} \), $C(\tilde \Omega_\measureone, \constbd) >0$ because \( \funcclass_J(\tau^2) \subseteq V_J(\tilde \Omega_\measureone) \). Hence,   choosing \( \varepsilon = \delta/4 \), Lemmas~\ref{lem:h}, \ref{lem:f}, and the previous display yield
	\begin{align}
\log N_{[\,]}( \tilde{\funcclass}_J(\tau^2), \| \, . \, \|_{L^2(\Omega_\measuretwo)}, \delta)
\lesssim {} & \log N( \tilde{\funcclass}_J(\tau^2), \| \, . \, \|_{L^\infty(\Omega_\measuretwo)}, \delta/2)\\
\lesssim {} & 2^{Jd} J + 2^{Jd} \log \left( \frac{1}{\delta} \right).
\label{eq:lh}
\end{align}

To apply the chaining bound of Theorem~\ref{thm:bernstein-chaining}, note that by \eqref{eq:bm} and Lemma \ref{lem:boundedextensions}, respectively, combined with Lemma \ref{lem:wavelet-approximation}, we have
\begin{equation}
	\label{eq:ta}
	\| g \|_{L^2(\R^d)} \lesssim \constsobolev \, 2^{-J}  \, \tau, \quad \| g \|_{L^{\infty}(\R^d)} \lesssim \constbdder \, 2^{-2J}, \quad \text{for } g \in \tilde \cF_J(\tau^2).
\end{equation}
Thus,
	\begin{equation}
		\label{eq:hj}
		T_{2,2} 
		\lesssim \frac{1}{\sqrt{n}} \cD_{[\,]} \big( 1 + \frac{\cD_{[\,]}}{\constsobolev^2 \, 2^{-2J} \, \tau^2 \sqrt{n}} \constbdder \, 2^{-2J}  \big)=\frac{1}{\sqrt{n}} \cD_{[\,]} \left( 1 + \constdudleycombined \, \frac{\cD_{[\,]}}{ \tau^2 \sqrt{n}}   \right),
	\end{equation}
where $\cD_{[\,]}=\cD_{[\,]}(\constsobolev \, 2^{-J}\tau, \tilde{\funcclass}_{J}(\tau^2), L^2(\measuretwo))$ is the Dudley integral defined in~\eqref{eq:ez}.
Moreover, by~\eqref{eq:lh} and Lemma~\ref{lem:i}, we have
	\begin{align}
		\label{eq:hk}
		\cD_{[\,]}
		\lesssim {} & \int_0^{\constsobolev \, 2^{-J} \tau} \sqrt{1 + 2^{Jd} J + 2^{Jd} \log \left( 1/\delta \right)} \, \ud \delta\\
		\lesssim {} & \tau 2^{J\frac{d-2}{2}} \sqrt{J \log\left(1 + \constbdder /\tau\right)}\,,
	\end{align}
	and therefore,
	\begin{align}
		\label{eq:hm}
		T_{2,2} \lesssim\tau \frac{2^{J\frac{d-2}{2}}}{\sqrt{n}} \, \sqrt{J \log\left(1+\constbdder/\tau\right)} + \frac{2^{J(d-2)} J}{n} \, \log\left(1+\constbdder/\tau\right).
	\end{align}
	
	\medskip
	Together with~\eqref{EQ:T21}, we have $T_2\lesssim \phi_J(\tau^2)$
with $\phi_J$ defined as in~\eqref{EQ:defphi} and we absorbed $\mathfrak{r}_J$ for $d\ge 2$ into the first term on the right-hand side.

\medskip

\textbf{Bounding \( T_1 \)}.
	\( T_1 \) can be bounded completely analogously to how we bounded \( T_2 \), with the exception that Lemma \ref{lem:g} is not needed. Thus, we obtain $T_1 \lesssim  \phi_J(\tau^2)$.

	\medskip

	\textbf{Final bound and concentration.}
	Collecting the above bounds on $T_1$ and $T_2$, we get
	\begin{equation}
		\label{eq:gl}
		\E [\sup_{\potential\in \funcclass_J(\tau^2)} | \emprisk_0(\potential)-\risk_0(\potential) | \lesssim \phi_J(\tau^2)\,.
	\end{equation}
To obtain a bound that holds with high probability, we apply the concentration result of Theorem~\ref{thm:concentration}.
	For this, note that \eqref{eq:lp} can be written as
	\begin{equation}
		\label{eq:lo}
		\emprisk_0(\potential)  -\risk_0(\potential)
		= ((\hat \measureone \otimes \hat \measuretwo) - (\measureone \otimes \measuretwo))((\potential - \potential_0) \otimes (\potential^\ast - \potential_0^\ast)),
	\end{equation}	
	where \( \measureone \otimes \measuretwo \) denotes the product measure and
	\( (f \otimes g)(x, y) = f(x) + g(y) \).
	Following the same argument as in the proof of Lemma~\ref{lem:boundedextensions}, we get
	\begin{equation}
		\label{eq:lq}
	\| (\potential - \potential_0) \otimes (\potential^\ast - \potential_0^\ast) \|_{L^\infty(\Omega_\measureone \times \Omega_\measuretwo)} \le \constbdder \, .
	\end{equation}
Moreover, similarly to Proposition~\ref{prop:gigli}, we have for any $f \in \cF_J(\tau^2)$ that
	\begin{equation}
		\label{eq:lr}
		\| (\potential - \potential_0) \otimes (\potential^\ast - \potential_0^\ast) \|_{L^2(\measureone \otimes \measuretwo)} \le \constsobolev \, \tau.
	\end{equation}
We can therefore apply Theorem~\ref{thm:concentration} and conclude that with probability at least $1-e^{-t}$, it holds that
	\begin{equation}
		\label{eq:dt}
		\sup_{\potential \in \funcclass_{J}(\tau^2)} \left| \risk_0(\potential) - \emprisk_0(\potential) \right|
		\lesssim  \phi_{J}(\tau^2) + \tau \sqrt{\frac{t}{n}} +  \frac{t}{n} \,,
	\end{equation}
	which concludes our proof.

	\subsection{Proof of Lemma~\ref{lem:boundedextensions}}
		By the boundedness conditions in the definition of \( \funcclassstd(2M) \) in Proposition-Definition \ref{lem:immediate-bounds}, and the boundedness of \( \tilde \Omega_\measureone \) and \( \Omega_\measuretwo \), we have for \( y \in \Omega_\measuretwo \) that
		\begin{equation}
			\label{eq:gh}
			| \potential^\ast(y) | = | \sup_{x \in \tilde \Omega_\measureone} \langle x , y \rangle - \potential(x) | \le \frac{1}{2} \constbdder(\tilde \Omega_\measureone, \Omega_\measuretwo, \constbd).
		\end{equation}
Moreover, by Lemma \ref{lem:conj-lipschitz}, \(  \nabla \potential^\ast \) is \( 2M \)-Lipschitz.
	Therefore, since \( \Omega_\measuretwo \) is bounded,
		\begin{equation}
			\label{eq:gk}
			\| \extension \potential^\ast \|_{B^2_{\infty, \infty}(\Omega_\measuretwo)}
			\lesssim_{\Omega_\measuretwo} \| \potential^\ast \|_{C^2(\Omega_\measuretwo)}
			\lesssim \constbdder.
		\end{equation}
		We can similarly deduce the second claim by Proposition-Definition \ref{lem:immediate-bounds} since \( f_0 \in \funcclassstd(\constbd) \), possibly choosing a larger \( \constbdder \).

\section{Additional lemmas}
\label{sec:lemmas}

\begin{lemma}
\label{LEM:1to1}
In the notation of the proof of Theorem \ref{thm:lb}, for any $k=0, \ldots, K$ and $m \ge m_0$, \( \nabla \phi_k \) is a bijection from \( [0,1]^d \) to \( [0,1]^d \).
\end{lemma}

\begin{proof}
	[{Proof of Lemma \ref{LEM:1to1}}]
	By construction, \( \phi_k \) is strongly convex and has Lipschitz continuous derivatives, hence so does its convex conjugate \( \phi_k^\ast \) by Lemma \ref{lem:conj-lipschitz}.
	In particular, \( \phi_k^\ast \) is defined on all of \( \R^d \) and for each \( y \),
	\begin{equation}
		\label{eq:aq}
		\phi_k^\ast(y) = \sup_x \langle x , y \rangle - \phi_k(x).
	\end{equation}
	Hence, the equation $\nabla \phi_k(x) = y$
	has a unique solution $x(y)$ for every \( y \in \R^d \) which implies that $\nabla \phi_k$ is injective and that for any $y \in [0,1]^d$, there exists $x=x(y)\in \R^d$ such that $\nabla \phi_k(x)=y$. It remains to check that $x(y) \in [0,1]^d$ for all $y \in [0,1]^d$. To that end, note that $\phi_k(x)=\|x\|^2/2$ for $x \notin [0,1]^d$. Hence $x(y)=y$ whenever $y \notin [0,1]^d$. In particular,  if $y \in [0,1]^d$, we must have $x(y)\in [0,1]^d$. This completes the proof.
	\end{proof}

\begin{lemma}
	[{Poincar\'e inequality, \cite[Section 5.8.1]{Eva10}, \cite[Theorem 13.27]{Leo17}}]
	\label{lem:poincare}
	~\\
	Let \( \Omega \subseteq \R^d \) be a bounded and connected Lipschitz domain.
	Then, there exists a constant \( C = C(d, \Omega) \) such that for any function $f \in W^{1,2}(\Omega)$,
	\begin{equation}
		\label{eq:ec}
		\| \potential - \int_\Omega \potential(x) \, \ud \lambda(x) \|_{L^2(\Omega)}
		\le C \| \nabla u \|_{L^2(\Omega)}.
	\end{equation}
\end{lemma}

\section{Numerical Experiments, continued}
\label{sec:numerics-ctd}

In this section, we give additional details on the numerical experiments in Section \ref{sec:numerics}, using the same notation used there to define \( \hat T_{\mathrm{emp}} \), \( \hat T_{\mathrm{wav}} \) and \( \hat T_{\mathrm{ker}} \).

\subsection{Implementation details}
\label{sec:implementation}

All simulations are done with Python 3.8.0 and Numpy 1.17.3, where some calculations are accelerated by the just-in-time compiler of Numba 0.47.0, in particular the calculation of the Linear-Time Legendre transform \cite{Luc97}.
The discrete wavelet transform is calculated with PyWT 1.1.1 \cite{LeeGomWas2019}.
We use a second-order finite difference operator for the calculation of the numerical gradient provided by the findiff package, version 0.8.0.
The optimization \eqref{eq:td} is performed with the L-BFGS algorithm \cite{LiuNoc89} as implement in Scipy 1.4.1, stopping at a relative decrease in objective function value of less than \( 10^{-9} \) and a maximum iteration number of \( 10000 \).
The baseline estimator is computed with the {\tt ot.emd} function of the Python Optimal Transport package, version 0.6.0.
The kernel regression problem \eqref{eq:uf} is solved via scikit-learn, version 0.22.2 \cite{scikit-learn}.
Plots were made using Matplotlib 3.1.2 and Seaborn 0.9.0.

The boxes in the calculation of \( \hat \fvec_J \) are picked to be \( \tilde \Omega_\measureone = \tilde \Omega_\measuretwo = [-0.5, 1.5]^d \) in the case \eqref{eq:identity} and \( \tilde \Omega_\measureone = [-0.5, 1.5]^d \), \( \tilde \Omega_\measuretwo = [0, 4]^d \) in the case \eqref{eq:exponential}.
To compute \( \hat \fvec_J \) for different \( J \), we initialize the optimization at all scales of \( J \) with the all zeros vector.
An approximation to the ground truth semi-dual objective \eqref{eq:pn} is obtained by integrating both \( \invwavtrafo_J\wavcoeffs_J \) and the interpolation \( \interp_{\yvec \to \tmap_0(\xvec)} \legtrafo(\invwavtrafo_J \wavcoeffs_J) \) over \( [0,1]^d \) with Simpson's rule, exploiting the fact that integration over \( \measuretwo \) is equivalent to integration over \( \measureone \) under the push-forward \( \tmap_0 \).

The parameters \( \nu_{\mathrm{kernel}} \) and \( \nu_{\mathrm{ridge}} \) are chosen via an oracle procedure, picking the best \( \nu_{\mathrm{kernel}} \in \{10^{-9}, 10^{-8.5}, \dots, 10^{-5}\}\), \( \nu_{\mathrm{ridge}} \in \{10^{-5}, 10^{-4.5}, \dots, 10^{-1} \} \) as determined by evaluating \( \mathrm{MSE} \) for an independent draw from \( \measureone \).

The data for all quantitative plots are obtained by taking the median over 32 \iid replicates.
Error bars are not shown since 95-percentile bootstrapped confidence intervals were not visible for most estimators at the present scale of the plots.

Running on one core of server processors such as an Intel\textsuperscript{\textregistered} Xeon\textsuperscript{\textregistered} E5-2670 v3 (2.30GHz), the calculation of all wavelet scales of \( \hat T_{\mathrm{wav}} \) for one replicate takes between 10 and 70 minutes, depending on the sample size and hence the conditioning of the problem.
We note that the runtime and space complexity of the algorithm is determined both by the discretization size \( N \) and to a lesser extent by the sample size \( n \), as opposed to computing the optimal transport plan between empirical distributions, whose complexity in the regimes considered here is governed entirely by the sample size \( n \).
As a comparison, computing all \( 100 \) different parameter settings for \( \hat T_{\mathrm{ker}} \) with \( n = 10000 \) takes about five hours.

\subsection{Numerical error dominates for large sample size}
\label{sec:error-flattening}

To investigate the observed flattening out of the \( \hat T_{\mathrm{wav}} \) error curves in Figure \ref{fig:errors}, we repeat the experiment for \eqref{eq:identity} and \( d = 3 \) with a lower resolution discretization, \( N = 33 \).
In Figure \ref{fig:errors_low_acc}, we compare the resulting error to the \( N = 65 \) case considered in Section \ref{sec:numerics}.
The error bottoms out for much lower values of \( n \), suggesting that numerical accuracy is indeed responsible for this behavior.

\begin{figure}[ht]
	\centering
		\includegraphics[width=0.48\textwidth]{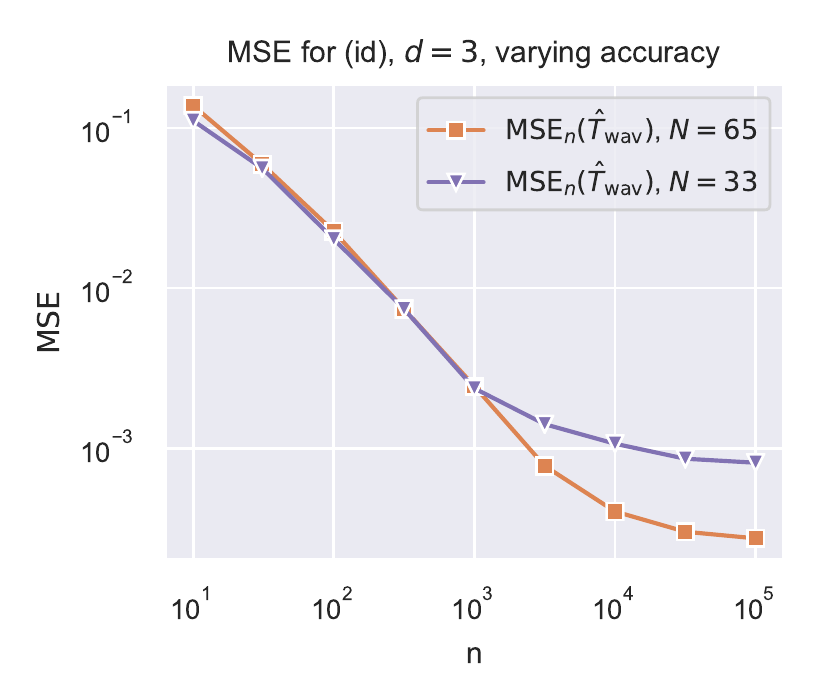}
		\caption{Estimation errors for \eqref{eq:identity}, \( d = 3 \), low and high accuracy discretization. Median over 32 replicates. The error curve flattens out earlier for \( N = 33 \), suggesting that numerical approximation errors are responsible for this phenomenon.}
	\label{fig:errors_low_acc}
\end{figure}

\begin{figure}[!ht]
	\centering
		\includegraphics[width=0.48\textwidth]{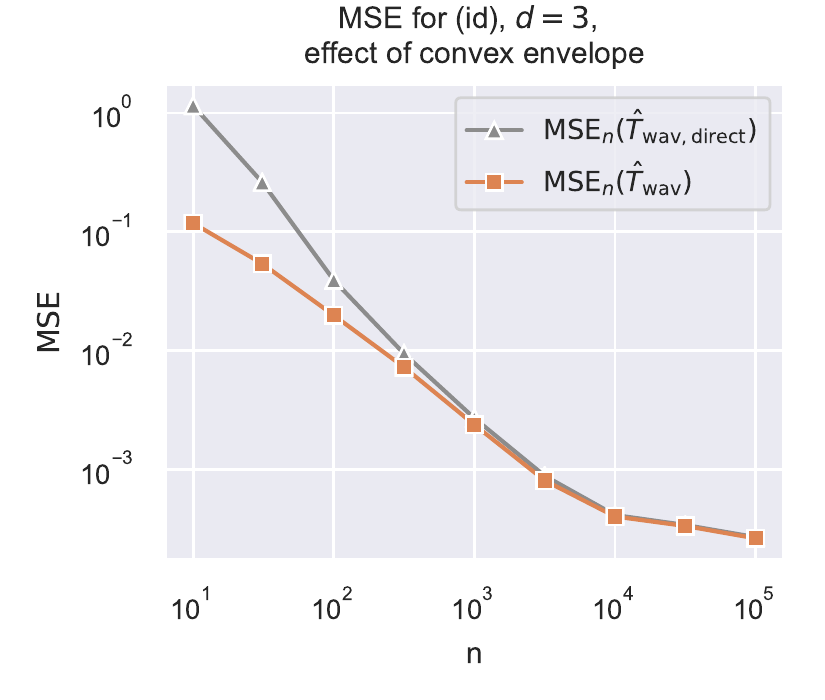}
		\caption{Estimation errors for \eqref{eq:identity}, \( d = 3 \), direct output of \eqref{eq:td}, \( \hat T_{\mathrm{wav, direct}} \), compared to its convex envelope \( \hat T_{\mathrm{wav}} \). Median over 32 replicates. For low sample sizes, taking the convex envelope improves the estimation rate.}
	\label{fig:errors_double}
\end{figure}

\subsection{Improved gradient estimation by computing convex envelope}
\label{sec:double-legendre}

\begin{figure}[hbtp!]
	\centering
	\subcaptionbox{
		Ground truth potential \( \potential_0 = \frac{1}{2} \|x\|^2 \)
		\label{fig:vis_pot_gt}
		}{
		\includegraphics[width=0.45\textwidth]{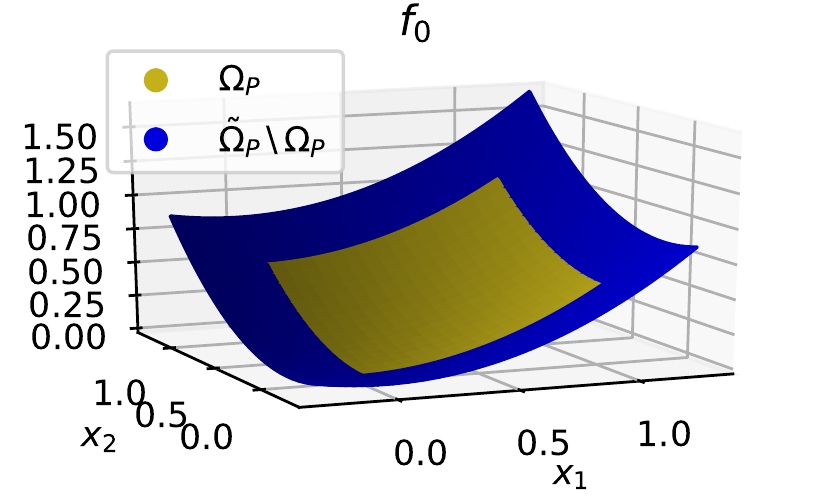}
	}
	\subcaptionbox{
		Ground truth transport map, first coordinate, \( (\tmap_0(x))_1 = (\nabla \potential_0(x))_1 = x_1 \)
		\label{fig:vis_grad_gt}
		}{
		\includegraphics[width=0.45\textwidth]{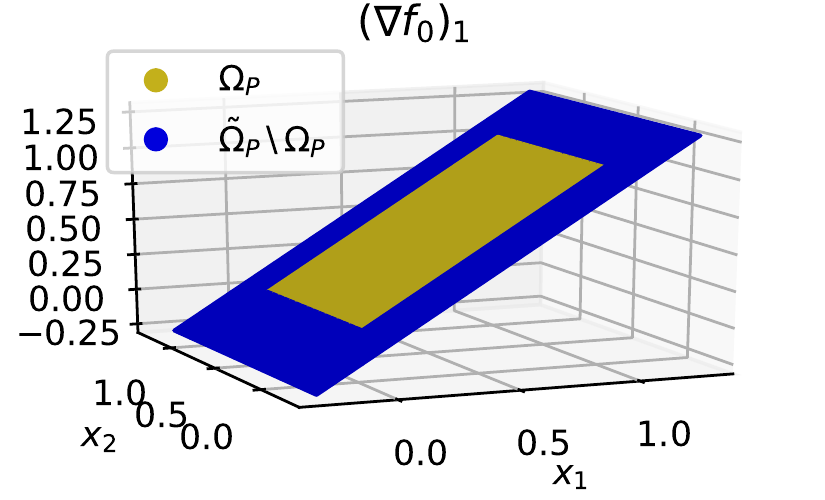}
	}

	\subcaptionbox{
		Estimator \( \hat \fvec_J \) with \( J =  1 \), see Section \ref{sec:double-legendre}
		\label{fig:vis_pot_est}
		}{
		\includegraphics[width=0.45\textwidth]{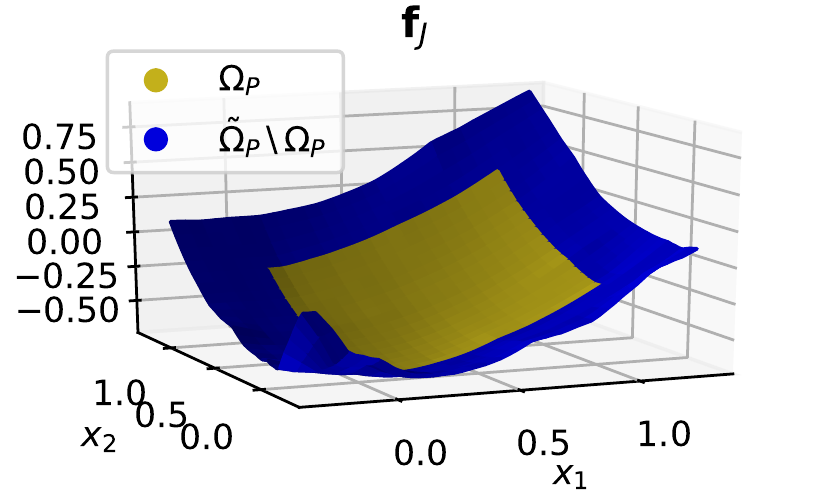}
	}
	\subcaptionbox{
		First coordinate of numerical gradient of \( \hat \fvec_J \) in Panel (\subref{fig:vis_pot_est})
		\label{fig:vis_grad_est}
		}{
		\includegraphics[width=0.45\textwidth]{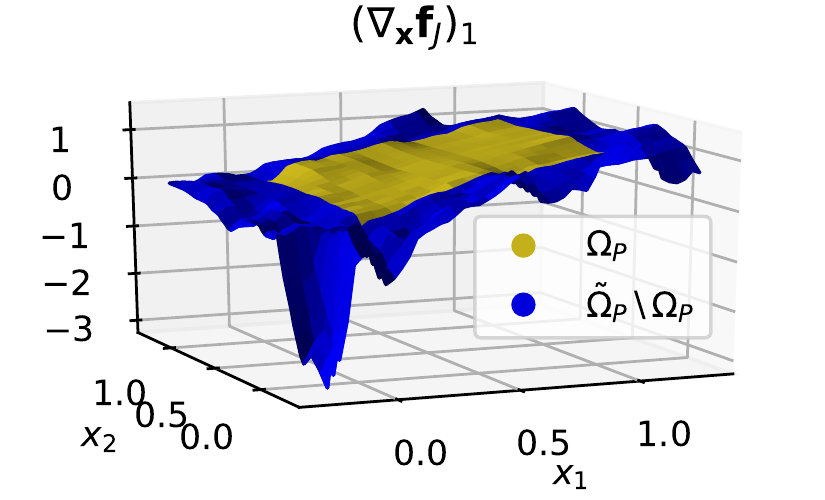}
	}

	\subcaptionbox{
		Convex hull \( \legtrafo(\legtrafo(\hat \fvec_J)) \) of estimator in Panel (\subref{fig:vis_pot_est})
		\label{fig:vis_pot_double}
		}{
		\includegraphics[width=0.45\textwidth]{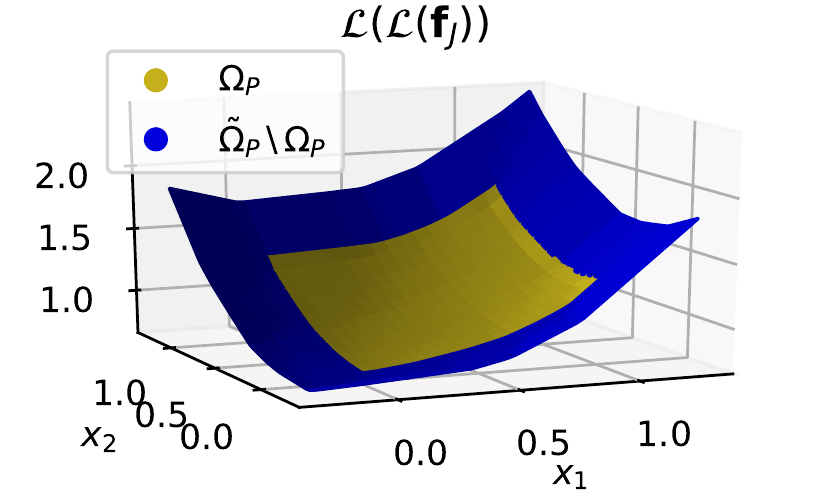}
	}
	\subcaptionbox{
		First coordinate of numerical gradient of \( \legtrafo(\legtrafo(\hat \fvec_J)) \) in Panel (\subref{fig:vis_pot_double})
		\label{fig:vis_grad_double}
		}{
		\includegraphics[width=0.45\textwidth]{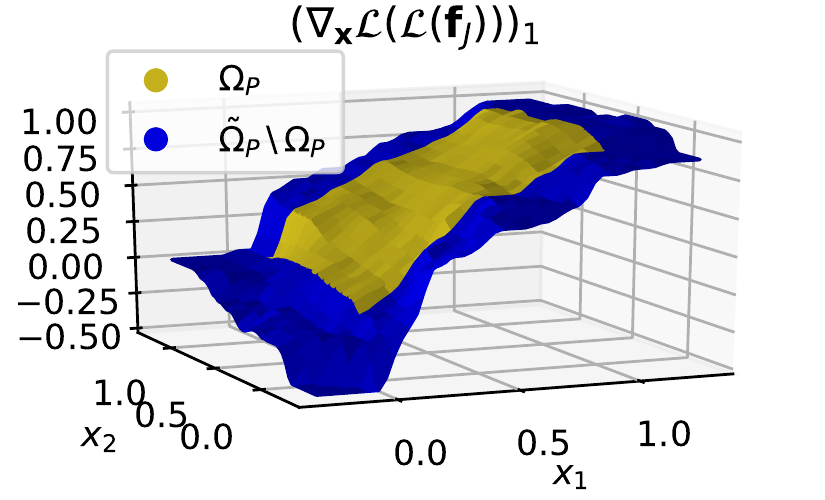}
	}

	\caption{Numerical instability of the boundary is remedied by computing the convex envelope of \( \hat \fvec \), see Section \ref{sec:double-legendre}. Visualization of potentials (ground truth (\subref{fig:vis_pot_gt}), estimated potential (\subref{fig:vis_pot_est}), and convex envelope of estimated potential (\subref{fig:vis_pot_double})), and the first coordinate of the associated gradients (\subref{fig:vis_grad_gt}, \subref{fig:vis_grad_est}, \subref{fig:vis_grad_double}, respectively) for the identity transport map and \( \measureone = \unif([0,1]^2) \) in \( d = 2 \), \( J = 1\), \( n = 100 \).}
	\label{fig:vis_pot}
\end{figure}

Write \( \hat \fvec_J = \invwavtrafo_J \hat \wavcoeffs_J \).
We observe that while \( \hat \fvec_J \) by itself often qualitatively yields good results on \( \Omega_\measureone \) (the support of \( \measureone \)), the lack of additional global regularity conditions in the optimization \eqref{eq:td} often leads to poor approximation outside of the support, i.e., on \( \tilde \Omega_\measureone \setminus \Omega_\measureone \).
In particular, certain regions could simply not enter the optimization at all due to the form of the discrete Legendre transform \( \legtrafo \) in \eqref{eq:ua}.
The numerical gradient estimate is somewhat sensitive to this, especially near the boundary, which prompts us to regularize the result further by instead considering its convex envelope, i.e., the largest convex function that lies entirely below the graph of \( \hat \fvec_J \).

We give a qualitative visual example of this in Figure \ref{fig:vis_pot}.
Depicted for \( d = 2 \) are first the ground truth potential \( \potential_0(x) = \frac{1}{2} \| x \|_2^2 \) corresponding to the identity transport map (Figure \ref{fig:vis_pot_gt}) and its derivative in the first coordinate (Figure \ref{fig:vis_grad_gt}), where we set \( \measureone = \unif([0,1]^2) \).
The parts of those functions corresponding to the support of \( \measureone \) are depicted in yellow, while those outside of the support are colored blue.
Results of the optimization for \( n = 100 \) \iid samples from \( \measureone \) and \( \measuretwo = \unif([0,1]^2) \) are shown in Figures \ref{fig:vis_pot_est} (potential) and \ref{fig:vis_grad_est} (first coordinate of numerical gradient) for \( J = 1 \).
While the estimator yields a good approximation to \( \potential_0 \) and \( \tmap_0 \) on the interior of \( \Omega_\measureone \), the outside appears very ragged.
This is remedied by applying \( \legtrafo \) twice, see Figures \ref{fig:vis_pot_double} (potential) and \ref{fig:vis_grad_double} (numerical gradient).

We also quantitatively compare the estimation accuracy of a linear interpolation of \( \nabla_{\xvec} \hat \fvec_{\hat J} \), denoted by \( \hat T_{\mathrm{wav, \,direct}} \) to that of \( \hat T_{\mathrm{wav}} \) by repeating the experiment for \eqref{eq:identity} and \( d = 3 \), plotting the results in Figure \ref{fig:errors_double}.
One can observe that for low sample sizes, considering the convex envelope indeed significantly improves the estimation accuracy.

\subsection{Performance for varying regularization strength}
\label{sec:num-params}

In order to further investigate the performance of the proposed regularized estimators \( \hat T_{\mathrm{wav}} \) and \( \hat T_{\mathrm{ker}} \), we show error plots for a selection of fixed regularization parameters in Figure \ref{fig:error-params}.
In the case of \( \hat T_{\mathrm{wav}} \), we plot \( \hat T_{\mathrm{wav}}^{(J)} \) for the whole range of \( J \in \{0, \dots, 3\} \) that can be calculated by means of the discrete wavelet transform when \( N = 65 \).
In the case of \( \hat T_{\mathrm{ker}} \), we plot a subset of all considered combinations \( \nu_{\mathrm{kernel}} \in \{10^{-9}, 10^{-8.5}, \dots, 10^{-5} \} \), \( \nu_{\mathrm{ridge}} \in \{10^{-5}, 10^{-4.5}, \dots, 10^{-1} \} \).

First, for the wavelet estimators \( \hat T_{\mathrm{wav}}^{(J)} \), we observe that the best bias-variance trade-off is achieved by \( J = 0 \) for smaller values of \( n \) and by \( J = 1 \) for \( n \ge 10^{3.5} \).
Moreover, the curve for \( J = 3 \), corresponding to no regularization, roughly follows the shape of the error curve of \( \hat T_{\mathrm{emp}} \), until it flattens out due to the numerical error.
This is most likely due to the grid approximations involved in the definition of \( \hat T_{\mathrm{wav}} \).

Next, the kernel estimators \( \hat T_{\mathrm{ker}}^{(\nu_{\mathrm{ridge}}, \nu_{\mathrm{kernel}})} \) in 3D show mostly similar error curves for the range of parameters considered, although a trade-off that depends on the sample size can be observed between \( (\nu_{\mathrm{ridge}}, \nu_{\mathrm{kernel}}) \in \{ (10^{-7}, 10^{-4}), (10^{-7}, 10^{-3}) \} \).
Moreover, in the case of very strong regularization, \( (\nu_{\mathrm{ridge}}, \nu_{\mathrm{kernel}}) = (10^{-5}, 10^{-5}) \), the error curve flattens out for values of \( n \) as low as \( 10^{2.5} \).
In general, for \( \hat T_{\mathrm{ker}} \), we observe that its performance is more sensitive to changes in \( \nu_{\mathrm{kernel}} \) than to those in \( \nu_{\mathrm{ridge}} \).

Last, In 10D, we observe a wider range of error curves according to the regularization strength, ranging from a curve that closely matches that of \( \hat T_{\mathrm{emp}} \) for \( (\nu_{\mathrm{ridge}}, \nu_{\mathrm{kernel}}) = (10^{-9}, 10^{-1}) \) to a curve that matches most of the error curve shown in Figure \ref{fig:errors_exp_10d} for \( (\nu_{\mathrm{ridge}}, \nu_{\mathrm{kernel}}) = (10^{-7}, 10^{-4}) \).

\begin{figure}[tbp]
	\centering
	\captionsetup[subfigure]{justification=centering}
	\subcaptionbox{
		Wavelet estimator \( \hat T_{\mathrm{wav}} \), \( d = 3 \)
		\label{fig:error-params-wav-3d}
		}{
		\includegraphics[height=0.25\textheight]{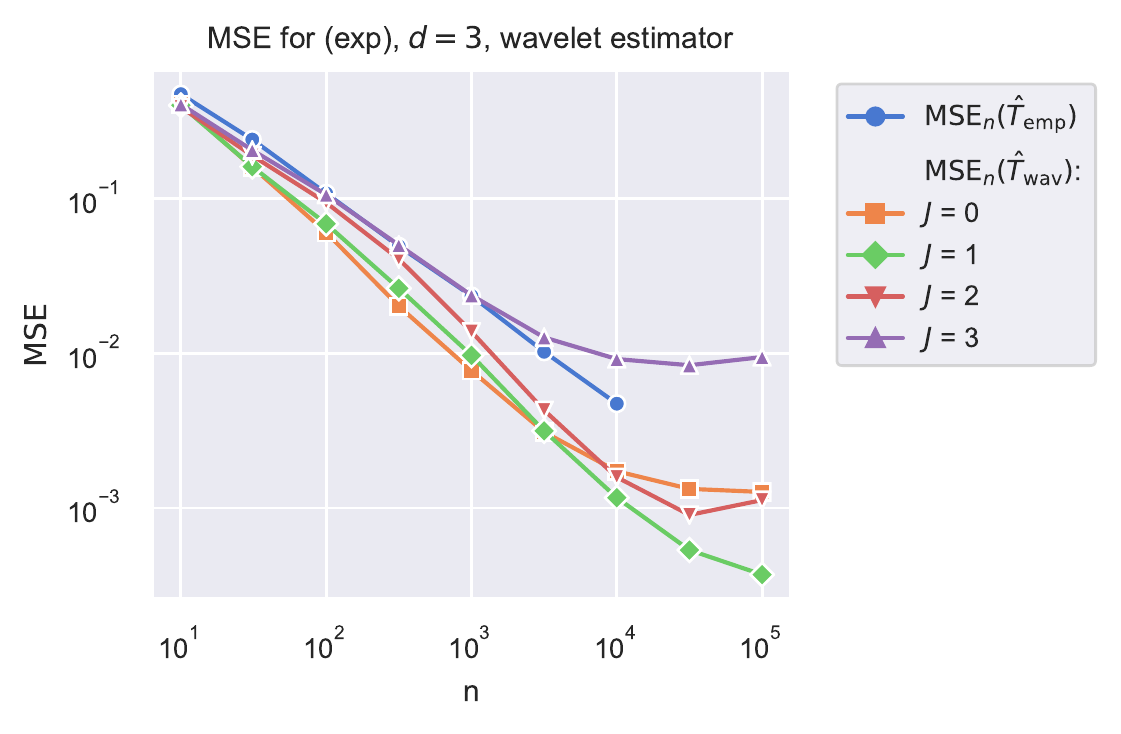}
	}
	\subcaptionbox{
		Kernel estimator \( \hat T_{\mathrm{ker}} \), \( d = 3 \)
		\label{fig:error-params-ker-3d}
		}{
		\includegraphics[height=0.25\textheight]{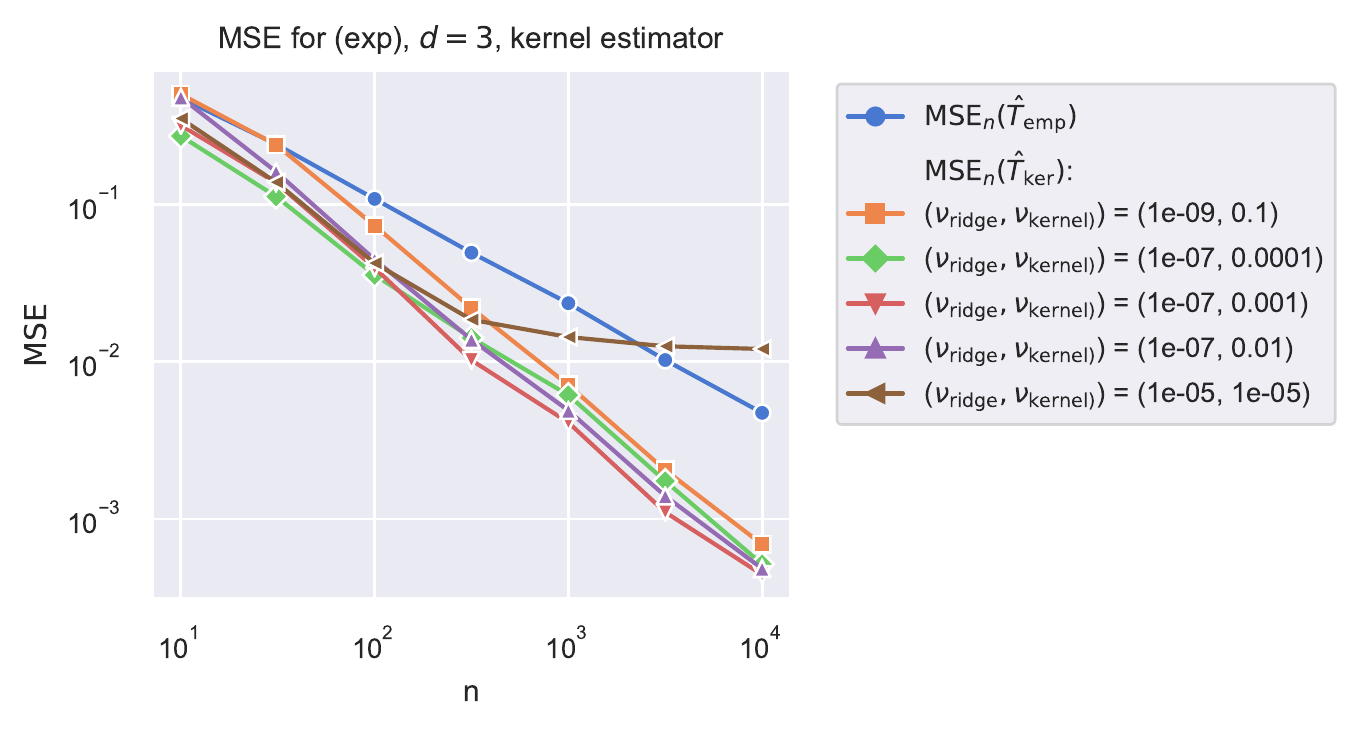}
	}

	\subcaptionbox{
		Kernel estimator \( \hat T_{\mathrm{ker}} \), \( d = 10 \)
		\label{fig:error-params-ker-10d}
		}{
		\includegraphics[height=0.25\textheight]{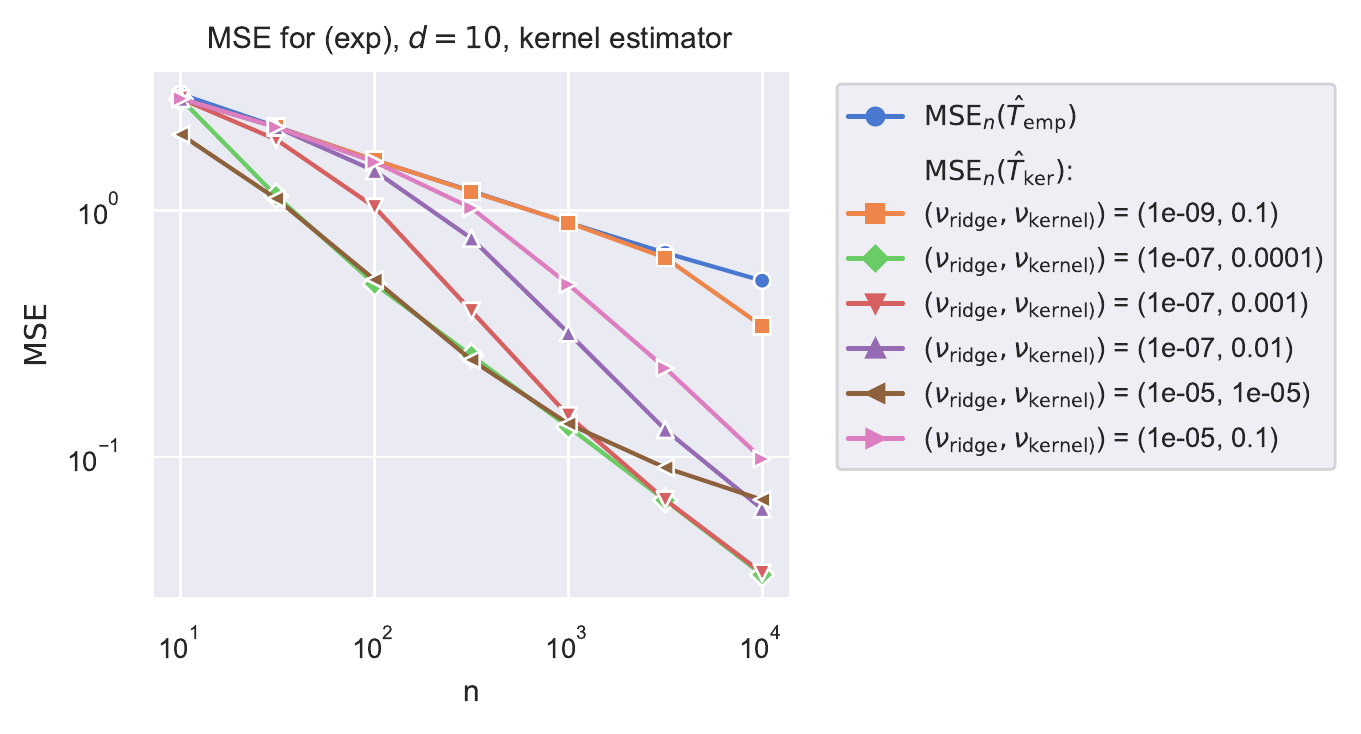}
	}

	\caption{Log-log plot of MSE for the exponential transport map, plotted against \( n \), showing the median error over 32 replicates. Individual curves correspond to the performance for different values of the regularization parameters, with errors for \( \hat T_{\mathrm{emp}} \) shown for comparison.}
	\label{fig:error-params}
\end{figure}

\clearpage

\bibliographystyle{alpha}
\bibliography{ot_stability}

\end{document}